\documentclass[letterpaper,10pt]{amsart}
\usepackage{indentfirst} 
\usepackage{amssymb}
\usepackage{mathrsfs}
\usepackage{amsmath}
\usepackage{amsbsy}
\usepackage{amsthm,stmaryrd}
\usepackage{thmtools}
\usepackage{enumitem}
\usepackage{bbm}            
\usepackage[colorlinks = true, linkcolor = black, citecolor = blue ]{hyperref}
\usepackage[usenames,dvipsnames]{xcolor} 
\usepackage{tikz}

\usepackage[left=2.6cm, right=2.6cm, bottom=3.4cm]{geometry}

\theoremstyle{plain}
\newtheorem{theorem}{Theorem}[section]

\theoremstyle{plain}
\newtheorem{proposition}{Proposition}[subsection]
\newtheorem{lemma}[proposition]{Lemma}
\newtheorem{corollary}[proposition]{Corollary}
\newtheorem{definition}[proposition]{Definition}

\theoremstyle{remark}
\newtheorem{remark}{Remark}[subsection]

\newcommand{\vertiii}[1]{{\left\vert\kern-0.25ex\left\vert\kern-0.25ex\left\vert #1 
    \right\vert\kern-0.25ex\right\vert\kern-0.25ex\right\vert}}

\def\e{\epsilon}
\def\R{{\mathbb R}}
\def\N{{\mathbb N}}

\def\T{{\bold T}}

\def\T{{\mathbf T}}
\def\z{{\mathbf{z}_{\mathrm{mod}}}}
\def\dr{{\mathrm d}}
\def\X{{\mathbf{X}}}
\def\Y{{\mathbf{Y}}}
\def\V{{\mathbf{V}}}
\def\L{{\mathscr L}}

\begin{document}

\title[Late-time asymptotics of small data solutions for the Vlasov--Poisson system]{Late-time asymptotics of small data solutions for the Vlasov--Poisson system}

\author[L\'eo Bigorgne]{L\'eo Bigorgne} \address{Institut de recherche math\'ematique de Rennes (IRMAR) - UMR 6625, CNRS, Universit\'e de Rennes, F-35000 Rennes, France.}
\email{leo.bigorgne@univ-rennes.fr}

\author[Renato Velozo Ruiz]{Renato Velozo Ruiz} \address{Department of Mathematics, University of Toronto, 40 St. George Street, Toronto, ON, Canada.}
\email{renato.velozo.ruiz@utoronto.ca}

\begin{abstract}
In this paper, we study the precise late-time asymptotic behaviour of small data solutions for the Vlasov--Poisson system in dimension three. First, we show that the spatial density and the force field satisfy asymptotic self-similar polyhomogeneous expansions. Moreover, we obtain an enhanced modified scattering result for this non-linear system. We show that the distribution function converges, with an arbitrary rate, to a regular distribution function along high order modifications to the characteristics of the linearised problem. We exploit a hierarchy of asymptotic conservation laws for the distribution function. As an application, we show late-time tails for the spatial density and the force field, where the coefficients in the tails are obtained in terms of the scattering state. Finally, we prove that the distribution function (up to normalisation) converges weakly to a Dirac mass on the zero velocity set.
\end{abstract}

\keywords{Vlasov--Poisson, polyhomogeneous expansions, modified scattering, late-time tails}
\subjclass{35Q70, 35Q83}

\maketitle

\setcounter{tocdepth}{1}
\tableofcontents

\section{Introduction}
In this paper, we investigate the late-time asymptotics of collisionless systems on $\R^3_x\times \R^3_v$ near vacuum. Specifically, we study collisionless systems described statistically by a distribution function $f(t,x,v)$ satisfying the \emph{Vlasov--Poisson system}
\begin{equation}\label{eq:VP}
\begin{cases}
\partial_t f+v\cdot\nabla_xf -\mu\nabla_x\phi \cdot \nabla_vf=0,\\ 
\Delta_x \phi=\rho(f),\\
\rho(f)(t,x):=\int_{\R_v^3}f(t,x,v) \mathrm{d} v,\\
f(t=0,x,v)=f_{0}(x,v),
\end{cases} \tag{VP}
\end{equation}
where $t\in \R$, $x\in \R^3_x$, $v\in \R^3_v$, and $\mu\in \{1,-1\}$. We call $\nabla_x\phi$ the \emph{force field}, and $\rho(f)$ the \emph{spatial density}. The non-linear Vlasov equation in \eqref{eq:VP} is a transport equation along the Hamiltonian flow defined by $\mathcal{H}_t(x,v):=\frac{|v|^2}{2}+\phi(t,x)$ in terms of the potential $\phi$. The non-linear term in this kinetic PDE system arises from the mean field generated by the many-particle system. Note that the interaction between the particles of the system is \emph{attractive} or \emph{repulsive}, when $\mu=1$ or $\mu=-1$, respectively. The Vlasov--Poisson system is a classical model for the description of collisionless many-particle systems in astrophysics (attractive case) \cite{BT} and plasma physics (repulsive case) \cite{LPit}. 

In this work, we are specifically interested in the study of fine asymptotic properties of small data solutions for the Vlasov--Poisson system. In particular, we will investigate regular solutions of the Vlasov--Poisson system near the vacuum solution $f\equiv 0$.

\subsection{Small data global existence for the Vlasov--Poisson system}

The non-linear dynamics of small data solutions for the Vlasov--Poisson system have been extensively studied in previous works. The first mathematical study of small data solutions for the Vlasov--Poisson system was performed by Bardos and Degond \cite{BD85}. In this work, small data global existence was proved for the Vlasov--Poisson system. For this, the authors obtained time decay for the spatial density by using the method of characteristics. Later in time, small data global existence was revisited by Hwang, Rendall, and Vel\'asquez \cite{HRV11}. Here, optimal time decay for high order derivatives of the spatial density was established.

More recently, the stability of the vacuum solution for the Vlasov--Poisson system was addressed once again by Smulevici \cite{Sm16}. In this paper, the small data global existence was obtained via energy estimates using a vector field method suitably adapted for kinetic equations. In particular, boundedness in time of a suitable energy norm is established, and also optimal space and time decay for the spatial density. We emphasise the novel modified vector field technique introduced in \cite{Sm16} to prove small data global existence in dimension three. Later in time, this work was revisited by Duan \cite{Du22}. In this article, the energy norms and the control of the force field are both simplified. 

See also Wang's work \cite{W23} for another proof of small data global existence for the Vlasov--Poisson system using Fourier techniques. Finally, Schaeffer \cite{SchPoisson} also studied small data solutions for \eqref{eq:VP}. Here, the smallness assumption on the derivatives of the distribution function was relaxed ($f$ is still assumed to be initially small).

\subsection{Small data modified scattering for the Vlasov--Poisson system}

In the last decade, there have been several works on the fine scattering properties of the distribution function for small initial data. Previous small data global existence results \cite{BD85,HRV11} proved that the non-linear term $\nabla_x\phi \cdot \nabla_vf$ decays in time. So, one could wonder if linear scattering holds. 

\begin{definition}
We say that \emph{linear scattering} holds for a solution of the Vlasov--Poisson system on $\R^3_x\times \R^3_v$, if there exists a regular function $f_{\infty}:\R^3_x\times \R^3_v \to \R$ such that the linear profile $f(t,x+vt,v)$ converges to $f_{\infty}(x,v)$. We call $f_{\infty}(x,v)$ the \emph{scattering state}.  
\end{definition}

This question was answered by Choi and Ha \cite{ChoiHa}, who proved that linear scattering \emph{does not} hold for any solution arising from non-trivial initial data, due to the long-range interaction of the particle system in dimension three. In this case, the non-linear term $\nabla_x\phi \cdot \nabla_vf$ only decays as $t^{-1}$ because of the long-range interaction. Nonetheless, one can still hope that modified scattering holds.

\begin{definition}
We say that \emph{modified scattering} holds for a solution $f$ of the Vlasov--Poisson system on $\R^3_x\times \R^3_v$, if $f$ does not enjoy a linear scattering dynamics, and if there exists a regular function $f_{\infty}:\R^3_x\times \R^3_v \to \R$ as well as lower order corrections $\mathscr{C}_x(t,x,v)$ and $\mathscr{C}_v(t,x,v)$, such that the modified profile $f(t,x+vt+\mathscr{C}_x,v+\mathscr{C}_v)$ converges in time. We call $f_{\infty}(x,v)$ the \emph{scattering state}.
\end{definition}

The first proof of small data modified scattering for the Vlasov--Poisson system was obtained by Choi and Kwon \cite{CK16}. Later, Ionescu, Pausader, Wang, and Widmayer \cite{IPWW}, obtained another proof of small data modified scattering using methods inspired from dispersive analysis. In particular, this paper identified an explicit logarithmic correction to the linearised characteristic system, in order to show modified scattering. For this, the authors considered the limit, when time goes to infinity, of the spatial average of the distribution function. This function allows to identify the precise self-similar asymptotic behavior of the force field. Finally, the explicit logarithmic correction of the linearised characteristic system can be found in terms of the asymptotics of the force field. Around the same time, Pankavich \cite{P22} proved modified scattering for a multispecies collisionless plasma assuming that the electric field decays sufficiently fast (instead of assuming smallness of the initial data). We note that \cite{P22} assumes compact support for the initial distribution function. 

We summarise this overview with a statement of small data modified scattering for the Vlasov--Poisson system according to \cite{IPWW, P22}.

\begin{theorem}[Small data modified scattering for the Vlasov--Poisson system]\label{thm_main_result_first_version}
Every solution $f$ to the Vlasov--Poisson system arising from regular and small initial data is global in time. Moreover, the following properties hold:
\begin{enumerate}[label = (\alph*)]
\item The spatial average of $f$ converges to a regular function $Q_{\infty}:\R^3_v\to \R$ such that $$\forall t\in [2,\infty), \qquad \Big|\int_{\R^3_x} f(t,x,v)\mathrm{d}x -Q_{\infty}(v) \Big|\lesssim \log(t)\langle t \rangle^{-1}.$$

\item The spatial density has a regular self-similar asymptotic profile $$\forall t\in [2,\infty), \qquad \Big|t^3\int_{\R^3_x} f(t,x,v)\mathrm{d}v -Q_{\infty}\Big(\frac{x}{t}\Big) \Big|\lesssim \log^{2}(t)\langle t \rangle^{-1}.$$

\item Let $\phi_{\infty}:\R^3_v\to \R$ be defined by $\Delta_v \phi_{\infty}=Q_{\infty}$. The force field has the regular self-similar asymptotic profile $v\mapsto \nabla_v \phi_{\infty}$, in the sense that $$\forall t\in [2,\infty), \qquad |t^2 \nabla_x\phi (t,x+tv)-
\nabla_v \phi_{\infty} (v)|\lesssim \langle x \rangle \log^{2}(t) \langle t \rangle^{-1}.$$ 
\item Modified scattering holds for the distribution function. There exists a regular distribution $\tilde{f}_{\infty}: \R^3_x\times\R^3_v\to \R$ such that $$\forall t\in [2,\infty), \qquad |f(t,x+tv+\mu\log (t) \nabla_v \phi_{\infty}(v),v)-\tilde{f}_{\infty}(x,v)|\lesssim \log^2(t)\langle t \rangle^{-1}.$$ 
\end{enumerate}
\end{theorem} 

We remark that small data linear scattering holds for the Vlasov--Poisson system on $\R^n_x\times\R^n_v$ when $n\geq 4$. This result was obtained by Pankavich \cite{P23}.

\subsection{The main results}


We first recall that the local well-posedness theory for this PDE system is standard. See \cite[Section 3]{HK19} for further details. Concerning the global regularity properties for \eqref{eq:VP}, seminal works by Pfaffelmoser \cite{Pf92} and Lions--Perthame \cite{LP91} established that this non-linear system is \emph{globally well-posed}. See also the proof of global well-posedness by Schaeffer \cite{Sch91}.

In the framework of the initial value problem, we study the evolution in time of small initial distributions $f_{0}:\R^{3}_x\times \R^{3}_v\to \R$, in a space of functions defined by a weighted $L^{\infty}_{x,v}$ norm
$$ \mathbb{E}_N^{N_x,N_v}[f_{0}]:=\sum_{|\beta|+|\kappa|\leq N}\sup_{(x,v)\in \R^3_x\times \R^3_v}\langle x \rangle^{N_x} \langle v \rangle^{N_v}|\partial_x^{\beta}\partial_v^{\kappa} f_{0}|,$$
 where $N, \,N_x,\, N_v \in \N$ and $\langle \cdot \rangle$ is the standard Japanese bracket. 

In the rest of the paper, the notation $A\lesssim B$ is used to specify that there exists a universal constant $C > 0$ such that $A \leq CB$, where $C$ depends only on the corresponding order of regularity, or other fixed constants.

\subsubsection{Late-time asymptotics for the Vlasov--Poisson system}

Let $N\geq 3$. Set the sequence $(r_n)_{n\geq 1} $ given by 
\begin{equation}\label{kev:defSn}
 r_n:=1+\frac{n(n+1)}{2},
 \end{equation}
  The main result of this paper establishes high order late-time asymptotics for small data solutions of the Vlasov--Poisson system. We note that this result requires a smallness assumption on $\mathbb{E}_N^{8,7}[f_{0}]$ and the finiteness of $\mathbb{E}_N^{N+4,7}[f_{0}]$. Note that we allow the norm $\mathbb{E}_N^{N+4,7}[f_{0}]$ to be large for $N>4$.

\begin{theorem}[High order late-time asymptotics for the Vlasov--Poisson system]\label{ThLatetime}
Let $N\geq 3$. Every solution $f\in C^N$ to the Vlasov--Poisson system arising from regular and small initial data is global in time. Let $n \in \mathbb{N}$ such that $ r_n \leq N$, $ \X_0(t,x,v) := x,$ and $\V_0(t,x,v):=v$. There exist regular modifications of the linear characteristic flow
\begin{align*}
\X_{n}(t,x,v) & =x+ \mu \nabla_v \phi_{\infty}(v) \log(t)+\sum_{1 \leq q \leq n-1} \, \sum_{|\alpha|+ p \leq q} \frac{x^\alpha \log^p(t)}{t^{q}} \mathbb{X}_{q,\alpha,p}(v), \\
\V_{n}(t,x,v) &=v+  \frac{\mu}{t}\nabla_v \phi_{\infty}(v)+\sum_{1 \leq q \leq n-1} \, \sum_{|\alpha|+ p \leq q} \frac{x^\alpha \log^p(t)}{t^{q+1}} \mathbb{V}_{q,\alpha,p}(v), 
\end{align*}
such that the following properties hold. For every integer $n$ such that $r_{n+1} \leq N$, 
\begin{enumerate}[label = (\alph*)]
 \item  The normalised spatial density $t^3 \!\rho (f)$ and the normalised force field $t^{2}\nabla_x  \phi$ satisfy asymptotic self-similar polyhomogeneous expansions of order $n$ according to Definitions \ref{Defselsimexp}--\ref{defpolhomoexpforce}. 
\item \label{stat_enhn_mod_scatt} Modified scattering holds with an enhanced rate of convergence. Let $g_{n+1}:\R_+^*\times \R^3_x\times\R^3_v\to \R$ be defined by
$$ g_{n+1}(t,x,v) := f \big(t,\X_{n+1}(t,x,v)+t\V_{n+1}(t,x,v),\V_{n+1}(t,x,v) \big).$$ There exists a regular distribution $f_{\infty}\colon \R^3_x \times \R^3_v \to \R$ such that for all $(t,x,v) \in [2,\infty) \times \R^3_x \times \R^3_v$ with $|x| \leq t$, we have
\begin{equation*}
   \big| g_{n+1}(t,x,v)-f_\infty (x,v)  \big| \lesssim  \frac{\log^{N(n+3)}(t)}{t^{n+1}}.
  \end{equation*}
\item The modified spatial average verifies enhanced convergence to the spatial average of $f_\infty$. There exists $\mathbf{Q}_{p,\xi} \in C^0 \cap L^\infty (\R^3_v) $ such that for all $(t,v) \in [2,\infty) \times \R^3_v$, we have
$$ \color{white} \square \qquad  \quad \color{black} \bigg| \int_{|x|<t }  g_{n+1}(t,x,v) \dr x~-\int_{\R^3_x}  f_\infty (x,v) \dr x~- \! \sum_{p +|\xi| \leq n} \frac{\log^p(t)}{t^{n+1}} \mathbf{Q}_{p,\xi} (v) \! \int_{\R^3_x} x^\xi  f_\infty (x,v) \dr x \bigg| \lesssim \frac{ \log^{N(n+3)}(t)}{t^{n+2}}.$$
\end{enumerate}
\end{theorem}

\begin{remark}
Note that for $n=0$, property $(a)$ and $(b)$ of Theorem \ref{ThLatetime}, corresponds to properties $(b)$--$(c)$ and $(d)$ of Theorem \ref{thm_main_result_first_version}, respectively. Already for $n=0$, the property $(c)$ of Theorem \ref{ThLatetime} is an improvement of property $(a)$ of Theorem \ref{thm_main_result_first_version}.
\end{remark}

\begin{remark}
The asymptotic self-similar polyhomogeneous expansion for the spatial density may be expected to hold due to the modified scattering of the distribution. However, the problem is far from being trivial because of the need to show an enhanced convergence estimate for the spatial averages of the distribution function. A posteriori, the polyhomogeneous expansion for the force field can be shown by using the \emph{asymptotic Poisson equation} (see Section \ref{SubsecforcefieldAxpan} for more details). 
\end{remark}

\begin{remark}
The quadratic loss of derivatives to derive the expansion of order $n$ for $t^3 \rho (f)$, and the convergence estimate for $g_{n+1}$, seem optimal for our method. So far, the previous modified scattering results for the small data solutions of \eqref{eq:VP} require to control $\nabla_{x}f$ and $\nabla_{v}f$. In our approach, we need to control $g_n$ in $W^{n+1,\infty}_{x,v}$ to show the convergence estimate for $g_{n+1}$. In this sense, the quadratic loss for the high order asymptotics is consistent with previous small data modified scattering results for \eqref{eq:VP}.
\end{remark}

\subsubsection{Non-linear tails and weak convergence}

The first part of Theorem \ref{ThLatetime} consists in proving asymptotic self-similar polyhomogeneous expansions for the normalised spatial density and the normalised force field. As an application, we can show non-linear late-time tails for the spatial density and the force field. For this purpose, we consider a hierarchy of asymptotic conservation laws for the solutions of the Vlasov--Poisson system.

Let $f_{\infty}\colon\R^3_x\times\R^3_v\to \R$ be a regular scattering state. Let $\alpha \in \N^3$ be a multi-index. We consider the \emph{weighted spatial averages} $\mathcal{A}^{\alpha}:\R^3_v\to\R$ defined by 
\begin{equation*}
\mathcal{A}^{\alpha}(v):=\int_{\R^3_x} x^{\alpha}\partial_v^{\alpha}f_{\infty}(x,v)\dr x.
\end{equation*}
We will later show that $\mathcal{A}^{\alpha}(v)$ are well-defined. The weighted spatial averages $\mathcal{A}^{\alpha}(v)$ can be characterised as 
\begin{equation}\label{charact_asymp_conse_laws}
\mathcal{A}^{\alpha}(v)=\lim_{t\to \infty} \int_{\R^3_x}x^{\alpha}\partial_{v}^{\alpha}g_1(t,x,v)\dr x.
\end{equation}
See Section \ref{SecRhoOrder2} for further details. In the sense of \eqref{charact_asymp_conse_laws}, the functions $\mathcal{A}^{\alpha}(v)$ define \emph{asymptotic conservation laws} for small data solutions of the Vlasov--Poisson system. We can now state our result on late-time tails for the spatial density and the force field.

\begin{theorem}[Late-time tails for the spatial density and the force field]
Let $N\geq 2$. Let $f\in C^N$ be a solution to the Vlasov--Poisson system arising from regular and small initial data. There exist constants $C^{p,q}\in\R$ such that, for $n= \frac{N}{2}-1$, the spatial density satisfies
\begin{align*}
\forall t\geq 2\qquad \bigg|t^3\int_{\R^3_v}f(t,x,v)\dr v-\sum_{p\leq q\leq n}\sum_{|\gamma|\leq n-q} \frac{C^{p,q}}{\gamma !}\partial_v^{\gamma}\mathbf{F}_{p,q}(0)\frac{x^{\gamma}\log^p (t)}{t^{|\gamma|+q}}\bigg|\lesssim \dfrac{\log^{N(n+3)}(t) }{t^{n+1}} \langle x\rangle^{n+1},
\end{align*}
where $\partial_v^{\gamma}\mathbf{F}_{p,q}(v)$ can be computed in terms of $\mathcal{A}^{\alpha}(v)$ and its derivatives. A similar expansion holds for the normalised force field $t^2 \nabla_x \phi$.
\end{theorem}

\begin{remark}
These non-linear tails for the spatial density and the force field, are obtained by using the asymptotic self-similar polyhomogeneous expansions for the spatial density and the force field.
\end{remark}

Geometrically, the decay in time of the spatial density holds due to the concentration in time of the support of the distribution function in the zero velocity set. We express the concentration of the support of the distribution with a suitable weak convergence statement.

\begin{theorem}[Concentration of the distribution in the zero velocity set]\label{thmconc_rough}
Let $\varphi\in C^{\infty}_{x,v}$ be a compactly supported test function. Then, for every solution $f$ to the Vlasov--Poisson system arising from regular and small initial data, we have $$\lim_{t\to\infty}\int_{\R^3_x\times \R^3_v}t^3f(t,x,v)\varphi(x,v)\dr x\dr v= \int_{\R^3_x} f_{\infty}(x,0)\dr x\int_{\R^3_x\times \R^3_v} \delta_{v=0}(v)\varphi(x,v)\dr x\dr v.$$ In other words, the distribution $t^3f(t,x,v)$ converges weakly to $(\int_{\R^3_x} f_{\infty}(x,0)\dr x)\delta_{v=0}(v)$.
\end{theorem}

\begin{remark}
The mass of the Dirac measure in Theorem \ref{thmconc_rough} is explicitly identified as the mass of the zero velocity set in terms of the scattering state. Later in the paper, we also prove a more general weak convergence statement for $t^{3}f(t,x+\bar{v}t,v+\bar{v})$, for a fixed $\bar{v}\in \R^3_v$. We show that $t^{3}f(t,x+\bar{v}t,v+\bar{v})$ converges weakly to the Dirac measure $(\int f_{\infty}(x,\bar{v})\mathrm{d}x)\delta_{v=0}(v)$. We observe that the masses $\int f_{\infty}(x,\bar{v})\mathrm{d}x$ of these Dirac measures are the masses along the energy levels $\{v=\bar{v}\}$ in terms of the scattering state. We note that the mass of the energy levels $\{v=\bar{v}\}$ defines the self-similar asymptotic profile of the spatial density. 
\end{remark}

\textit{Previous weak convergence results for the Vlasov--Poisson system.}
\begin{enumerate}[label = (\alph*)]
\item The previous weak convergence result can be compared with the weak convergence of the distribution function for solutions of the Vlasov--Poisson system on $\mathbb{T}^3_x\times \R^3_v$ in the breakthrough work on Landau damping by Mouhot and Villani \cite{MV09}. For this comparison, it is important to also consider our weak convergence result for $t^{3}f(t,x+\bar{v}t,v+\bar{v})$ to the Dirac mass $(\int f_{\infty}(x,\bar{v})\mathrm{d}x)\delta_{v=0}(v)$. For further details see Theorem \ref{thm_weak_convergence_prop_complete}.
\item Theorem \ref{thmconc_rough} can also be compared with the weak convergence result \cite[Theorem 1.3]{BVV23} for the distribution function in the context of small data solutions for the Vlasov--Poisson system with a trapping potential $\frac{-|x|^2}{2}$. In this context, the distribution function (up to normalisation) converges weakly to a Dirac mass on the unstable manifold of the origin. For this system, the mass of the limiting Dirac measure is equal to the mass of the stable manifold of the origin with respect to the scattering state. 
\end{enumerate}

\subsection{Key elements of the proof of high order asymptotics}

In this subsection, we show the key elements in the proof of Theorem \ref{ThLatetime}.

\subsubsection{Modified scattering: A story of asymptotics}\label{Subsecstrategymodifiedscatt}

We begin recalling the basic ideas to prove small data modified scattering for the Vlasov--Poisson system. In particular, we emphasise the role of late-time asymptotics.

\textbf{Step $1$.} Although we cannot expect $f$ to have a linear behaviour for large times due to \cite{ChoiHa}, we can still expect a weaker quantity to verify such a property. It turns out that the spatial average of $f$, which is conserved in the linear case and moreover governs the asymptotic behaviour of $\rho(f)$, converges as $t \to \infty$. Using the Vlasov equation, and performing integration by parts in $x$, we have
$$ \bigg|\frac{\dr}{\dr t}\int_{\R^3_x} f(t,x,v) \dr x\bigg|=\bigg| \int_{\R^3_x} \mu \nabla_x \phi(t,x) \cdot \big[t\nabla_x+\nabla_v f \big](t,x,v) \dr x+t \int_{\R^3_x} \mu \Delta_x \phi(t,x) f(t,x,v) \dr x \bigg| , $$ where the RHS is bounded above by the time integrable function $\log(t) t^{-2}$.

\textbf{Step $2$.} Then, we can isolate the leading order contribution of the charge density by 
\begin{equation}\label{Intro0order}
\bigg| t^3 \int_{\R^3_v} f(t,x,v) \dr v -  Q_\infty \Big( \frac{x}{t} \Big) \bigg| \lesssim \frac{\log(t)}{t},  \qquad Q_\infty(v) := \lim_{t \to + \infty}\int_{\R^3_x} f(t,x,v) \dr x  .
\end{equation}

\textbf{Step $3$.} This allows us to consider the asymptotic Poisson equation $\Delta_v \phi_\infty =Q_\infty$. Its solution captures the asymptotic behaviour of the force field along the trajectories of the particle system 
\begin{equation}\label{Introexpphi} 
\big| t^2 \nabla_x \phi(t,x+tv)-\nabla_v \phi_\infty(v) \big|  \lesssim \langle x \rangle \, \log^2(t) t^{-1}.
\end{equation}

\textbf{Step $4$.} We are finally able to prove that $f$ converges along a logarithmic correction of the linear characteristics by $$ \big| f \big(t,\X_1(t,x,v)+tv,v)-f_\infty(x,v) \big| \lesssim \log^{3}(t) t^{-1}, \qquad  \X_1(t,x,v):= x+\mu \log(t) \nabla_v \phi_\infty (v).$$ 

\subsubsection{Improved late-time asymptotic expansion for the spatial density}\label{Subsecsecondorde}

We now sketch the key arguments to prove an improved late-time asymptotic expansion for the spatial density. Here, we begin focusing on the first order expansion in powers of $t^{-1}$. 

\textbf{Step $1'$ and $2'$.} We now wonder, if we could obtain an improved asymptotic expansion for the spatial density $\rho(f)$. For this, the starting point of the analysis is the same as for the derivation of \eqref{Intro0order}, that is
$$ \rho(f)(t,x) = \int_{\R^3_v} g_0(t,x-tv,v) \dr v = \int_{\R^3_y} g_0 \Big(t,y,\frac{x-y}{t}\Big) \dr y , \qquad g_0(t,x,v):=f(t,x+tv,v).$$
Instead of applying the mean value theorem, we use this time a first order Taylor expansion to show
\begin{align}
\bigg| t^3\rho (f)(t,x)-  \int_{\R^3_y} g_0 \Big(t,y,\frac{x}{t} \Big) \dr y +\frac{1}{t} \int_{\R^3_y} y \cdot \big[ \nabla_v g_0 \big] \Big(t,y,\frac{x}{t} \Big) \dr y \bigg| \lesssim  \frac{ \log^{6}(t)}{t^{2}} . \label{eq:introrho}
 \end{align}
The problems to derive a first order expansion for the normalised density $t^3 \rho(f)$ in $t^{-1}$, are:
 \begin{itemize}
 \item The spatial average of $g_0$ merely converges to $Q_\infty$ at the rate $t^{-1}\log(t)$. 
 \item The third term in \eqref{eq:introrho} is a weighted spatial average of $\nabla_v g_0$, so we cannot treat it by applying previous techniques. Moreover, this quantity does not converge in general, as it is suggested by the modified scattering dynamics. 
 \end{itemize}
 
We deal with the second issue by rewriting $y \cdot \nabla_v g_0(t,y,v)$ in terms of derivatives of $f(t,\X_1+tv,v)$ or $g_1$, introduced thereafter. For the first issue, it is required to consider higher order corrections for the linear velocity characteristics as well. Since $\frac{\mathrm{d}V}{\mathrm{d} t}=\mu \nabla_x \phi (t,X)$ and in view of \eqref{Introexpphi}, we consider the modification
$$g_1(t,x,v):= f(t,\X_1(t,x,v) +t\V_1(t,x,v) ,\V_1(t,x,v) ), \qquad  \mathbf{V}_1(t,x,v) := v+\mu t^{-1} \nabla_v \phi_\infty (v).$$ 
The main idea then consists in deriving a similar estimate to \eqref{eq:introrho}, in terms of the modified profile $g_1$ instead of $g_0$. Then, we show the following two properties:
 \begin{itemize}
 \item An enhanced convergence estimate for the spatial average of $g_1$ towards the spatial average of $f_\infty$.
 \item The weighted spatial average of $\nabla_v g_1$ converges to the corresponding weighted spatial average of $\nabla_v f_\infty$ at the rate $\log^6(t)t^{-1}$.
 \end{itemize} 
 
 \textbf{Step $3'$ and $4'$.} On the top of these estimates, we improve our estimate on the force field. Finally, we consider second order modified characteristics $(\X_2,\V_2)$, and we show that the distribution $$g_2(t,x,v):= f\Big(t,\X_2(t,x,v)+t\V_2(t,x,v),\V_2(t,x,v)\Big),
$$ satisfies $$ \big| g_2(t,x,v)-f_\infty(x,v) \big| \lesssim \log^{16}(t) t^{-2}.$$

\subsubsection{About the rest of the proof}

We derive high order late-time asymptotics for $\rho(f)$ by iterating this process. Let us however mention that the $x$-dependency of the modified characteristics of order $n \geq 2$ gives rise to new difficulties that force us to restrict our convergence statement to the domain $\{|x| \leq t\}$. After, we establish an asymptotic self-similar polyhomogeneous expansion for the spatial density, we use the asymptotic Poisson equation to show a polyhomogeneous expansion for the force field. In particular, we can then prove asymptotics for the force field along the high order spatial modified characteristics $\X_n+t\V_n$. We use these estimates to show the enhanced modified scattering result for the high order modified profile $g_{n+1}(t,x,v)$, which converges to $f_\infty$ as $t^{-n-1}\log^{N(n+3)}(t)$ .

Finally, we show strong convergence estimates for the spatial averages of the modified profile $g_{n+1}$. The key idea of these estimates consists in observing that the modified characteristics $(\X_{n+1},\V_{n+1})$ of order $n+1$ allow us to prove the next property in an induction argument. Up to error terms decaying as $t^{-n-3}$, the derivative 
$$\partial_t  \int_{|x| <t} g_{n+1}(t,x,v) \dr x $$
is equal to a sum of terms of the form $\mathbf{Q}(v)t^{-n-2} \log^p(t)$ with $p \leq n$, where the quantities $\mathbf{Q}(v)$ can be computed in terms of $\mathcal{A}^\alpha(v)$. With these estimates, we will complete the proof of Theorem \ref{ThLatetime}.

\subsection{Related works on modified scattering}

In this subsection, we give a list of related works on modified scattering for other non-linear Vlasov equations. Specifically, we discuss previous modified scattering results for the relativistic Vlasov--Maxwell system on Minkowski spacetime, the Vlasov--Poisson system with a trapping potential on $\R^2_x\times \R^2_v$, and the repulsive Vlasov--Poisson system on $\R^3_x\times \R^3_v$ with a point charge. In these three different settings, one could consider the problem of studying high order late-time asymptotics.

\subsubsection{For the relativistic Vlasov--Maxwell system on Minkowski spacetime}

Small data modified scattering for the relativistic Vlasov--Maxwell system on Minkowski spacetime \cite{B22} has been recently established by the first author. This non-linear PDE system models the dynamics of a collisionless plasma of charged particles. We note that the small data modified scattering result in \cite{B22} \emph{does not} require smallness on the Maxwell field. In contrast with previous works on the subject, this paper establishes small data modified scattering in \emph{high order regularity}. See the work of Pankavich and Ben-Artzi \cite{PB23} for another proof of small data modified scattering for the relativistic Vlasov--Maxwell system for compactly supported initial data. 

\subsubsection{For the Vlasov--Poisson system with a trapping potential on $\R^2_x\times \R^2_v$}

Small data modified scattering for the Vlasov--Poisson system with the trapping potential $\frac{-|x|^2}{2}$ on $\R^2_x\times \R^2_v$, has been recently established in a joint work \cite{BVV23} of the authors with A. Velozo Ruiz. The motivation to consider small data solutions for the Vlasov--Poisson system with the potential $\frac{-|x|^2}{2}$ comes from studying the stability of systems for which the dynamics of their particles are \emph{hyperbolic}. Here, we considered the Vlasov--Poisson system with the simplest external potential for which \emph{unstable trapping} holds for the associated linearised characteristic flow. We note the use of hyperbolic type coordinates in order to prove the modified scattering result. See also \cite{VV24} for more information on this non-linear system.

\subsubsection{For the repulsive Vlasov--Poisson system with a point charge on $\R^3_x\times \R^3_v$}

Beyond vacuum solutions for non-linear Vlasov equations, we remark the modified scattering dynamics for perturbations of a point charge for the repulsive Vlasov--Poisson system obtained by Pausader, Widmayer, and Yang \cite{PWY}. We note the use of asymptotic action-angle coordinates to show this modified scattering result.

\subsection{Related works on high order late-time asymptotics}

Here, we comment about related works on late-time asymptotics for hyperbolic PDEs in general relativity. There has been a lot of progress in the last years concerning this problem. 

\subsubsection{Late-time tails for solutions of the wave equation on black hole spacetimes}

Precise late-time asymptotics have been derived for the solutions of the wave equation on black hole spacetimes by Angelopoulos, Aretakis, and Gajic \cite{AAGsph, AAGextRN, AAGkerr}. In these works, late-time tails are derived for linear scalar fields. For this purpose, the authors use a hierarchy of asymptotic conservation laws for the solutions of the wave equation on the corresponding black hole spacetimes. For comparison with the results in this paper, see the late-time tails for the linearised system in Section \ref{seclinear}.

In this context, a logarithmic term arises when studying the leading order asymptotics of solutions to the wave equation on a Schwarzschild black hole for initial data slowly decaying at infinity. This result was established by Kehrberger \cite{K21}. In the problem considered in this article, we obtain logarithmic terms in the non-linear expansions for completely different reasons. Indeed, these terms arise for any non-trivial small data solution to \eqref{eq:VP}. The logarithmic terms in the expansions of this paper appear because of non-linear effects. 

We refer to the overview \cite{GK22} for more information on the recent progress concerning the relation of conservation laws and late-time tails for massless scalar fields on black hole spacetimes. 

\subsubsection{Late-time tails for solutions of wave equations on dynamic asymptotically flat spacetimes}

On recent work by Oh and Luk \cite{OL24} a general method is developed for the study of late-time tails of solutions to wave equations on asymptotically flat spacetimes with odd space dimensions. In this paper, late-time tails are obtained for wave equations on dynamical backgrounds, and also for non-linear wave equations. We remark novel corrections obtained for the Price law rates concerning the decay of the solutions of wave equations on black hole spacetimes.

\subsubsection{Polyhomogeneity of the metric for solutions of the Einstein vacuum equations}

We also mention the work by Hintz and Vasy \cite{HV20} on the nonlinear stability of Minkowski spacetime for the Einstein vacuum equations. In this result, the authors use the framework of Melrose’s b-analysis to prove that the metric satisfies a polyhomogeneous expansion for polyhomogeneous initial data. We note that for the problem addressed in this article, we merely require the initial data to be non-trivial in order to obtain such expansions.

\subsection{Outline of the paper}
The rest of the article is structured as follows.

\begin{itemize}

\item \textbf{Section \ref{secpreliminaries}.} We study the linearisation of the Vlasov--Poisson system with respect to the vacuum solution. Here, we introduce the vector fields used to define the energy norms in the main results. We also define the notion of an asymptotic self-similar polyhomogeneous expansion.

\item \textbf{Section \ref{secmain}.} We state the precise statements of the main results of the article.

\item \textbf{Section \ref{seclinear}.} We study the late-time asymptotic behaviour of the linearised system. First, we prove asymptotic self-similar expansions for the spatial density. We set the hierarchy of conservation laws for the linearised system. Then, we show late-time tails for the spatial density in terms of the conservation laws. We prove that the distribution (up to normalisation) converges weakly to a Dirac mass in the zero velocity set. We also capture the shearing of the system with a weak convergence statement. 

\item \textbf{Section \ref{section_small_data_global_existence}.} We show small data global existence in high order regularity for the Vlasov--Poisson system. For this, we prove that weighted $L^{\infty}_{x,v}$ norms of the distribution function grow at most polynomially. 

\item \textbf{Section \ref{SecModiscat}.}   We obtain small data modified scattering in high order regularity for the Vlasov--Poisson system. For this, we show that the spatial density and the force field have self-similar asymptotic profiles. 

\item \textbf{Section \ref{SecRhoOrder2}.} We begin the study of high order late-time asymptotics of the spatial density for the Vlasov--Poisson system. In particular, we show a second order expansion for the spatial density. 

\item \textbf{Section \ref{SEcLatetime}.} We prove asymptotic polyhomogeneous self-similar expansions for the spatial density and the force field. Then, we obtain an enhanced modified scattering result for the distribution function.

\item \textbf{Section \ref{sec_tail_weak_conv}} We obtain late-time tails for the spatial density and the force field. We also prove that the distribution (up to normalisation) converges weakly to a Dirac mass in the zero velocity set. We also capture the shearing of the system with a weak convergence statement. 
\end{itemize}

\subsection{Acknowledgements}
LB conducted this work within the France 2030 framework programme, the Centre Henri Lebesgue ANR-11-LABX-0020-01. RVR would like to express his gratitude to Jacques Smulevici for many stimulating discussions. RVR acknowledges partial support from the European Union’s Horizon 2020 research and innovation programme under the Marie Skłodowska-Curie grant 101034255.

\section{Preliminaries}\label{secpreliminaries}

In this section, we introduce the set of commuting vector fields $\boldsymbol{\lambda}$ and the set of weights $\mathbf{k}$ used to define the space of functions considered in the paper. The vector fields in $\boldsymbol{\lambda}$ and the weights in $\mathbf{k}$, are motivated by the dynamics of the characteristic flow for the linearised system. We then state the commuted equations for the non-linear Vlasov equation, and for the Poisson equation. Finally, we define the notion of asymptotic self-similar polyhomogeneous expansion that we will use throughout this paper.

\subsection{The linearised system}
The linearisation of the Vlasov--Poisson system with respect to its vacuum solution $f\equiv 0$, corresponds to the \emph{Vlasov equation} 
\begin{equation}\label{vlasov_linear_flow}
\begin{cases}
\partial_t f+v\cdot\nabla_xf=0,\\ 
f(t=0,x,v)=f_{0}(x,v),
\end{cases}\tag{V}
\end{equation}
where $f_{0}:\R^{3}_x\times \R^{3}_v\to \R$ is a regular initial data. The Vlasov equation \eqref{vlasov_linear_flow} can be explicitly solved by 
\begin{equation*}
f(t,x,v)=f_{0}(x-vt,v)
\end{equation*}
in terms of the linear flow map 
\begin{equation*}
(X_{\L}(t,x,v),V_{\L}(t,x,v)):=(x+vt,v).
\end{equation*}
For future reference, we denote the linear transport operator in \eqref{vlasov_linear_flow} as 
\begin{equation*}
\T_0:=\partial_t +v\cdot\nabla_x.
\end{equation*}

The Vlasov equation on $\R^{3}_x\times \R^{3}_v$ is a transport equation along the Hamiltonian flow 
\begin{equation}\label{linear_hyperbolic_ode_system}
    \dfrac{\dr x}{\dr t}=v,\qquad \dfrac{\mathrm{d}v}{\dr t}=0.
\end{equation}
In other words, the characteristic flow \eqref{linear_hyperbolic_ode_system} is the Hamiltonian flow induced by the Hamiltonian system $(\R^{3}_x\times \R^{3}_v, \mathcal{H})$ with $$\mathcal{H}(x,v):=\dfrac{1}{2}(v^1)^2+\dfrac{1}{2}(v^2)^2+\dfrac{1}{2}(v^3)^2.$$ 

\begin{proposition}
The Hamiltonian system $(\R^{3}_x\times \R^{3}_v,\mathcal{H})$ is completely integrable in the sense of Liouville.
\end{proposition}

\begin{proof}
Consider the three independent conserved quantities in involution $v^i$ where $i\in \{1,2,3\}$. 
\end{proof}

\subsection{Commuting vector fields}

In this subsection, we introduce a set $\boldsymbol{\lambda}$ of vector fields on $\R^3_x\times \R^3_v$ that we will use to show time decay for small data solutions of the Vlasov--Poisson system. We consider vector fields that satisfy good commuting properties with the linearised system \eqref{vlasov_linear_flow}.

Specifically, we consider 
\begin{enumerate}[label = (\alph*)]
    \item translation vector fields $T_i:=\partial_{x^i}$,
    \item Galilean vector fields $G_{i}:=t\partial_{x^i}+\partial_{v^i}$,
\end{enumerate}
and we define 
\begin{equation*}
\boldsymbol{\lambda}:=\Big\{ T_i,G_i: ~i\in \{1,2,3\}  \Big\}.
\end{equation*}
We will further use the notation $G=t\nabla_x+\nabla_v$. The vector fields in $\boldsymbol{\lambda}$ were previously used in \cite{Sm16, Du22} to set the energy spaces on which small data global existence was obtained. Note that \cite{Sm16, Du22} also used the scaling vector field $L:=\sum_{i=1}^3 x^i\partial_{x^i}+v^i\partial_{v^i}$ and the rotational vector fields $R_{ij}=x^i\partial_{x^j}-x^j\partial_{x^i}+v^i\partial_{v^j}-v^j\partial_{v^i}$ in their energy spaces. In this article, we \emph{only} use the translations and the Galilean vector fields.

\begin{lemma}\label{lemma_commutators_Vlasov_external_potential}
Let $Z\in \boldsymbol{\lambda}$. Let $f$ be a regular solution of the Vlasov equation. Then, $Zf$ is also a solution of this equation. 
\end{lemma}
 
Lemma \ref{lemma_commutators_Vlasov_external_potential} justifies the terminology \emph{commuting vector fields} for the elements in $\boldsymbol{\lambda}$.

\subsection{Weights preserved by the linear flow}

We consider 
\begin{enumerate}[label = (\alph*)]
    \item translation weights $v^i$,
    \item Galilean weights $z_{i}:=x^i-v^it$,
\end{enumerate}
and we define 
\begin{equation*}
\mathbf{k}:=\Big\{ v_i,z_i: ~i\in \{1,2,3\}  \Big\}.
\end{equation*}
The weights in $\mathbf{k}$ are conserved along the characteristic flow $(x,v)\mapsto (x+vt,v)$ of the Vlasov equation. In particular, the weight functions are solutions to the Vlasov equation. 

\begin{lemma}
For every weight $w\in \mathbf{k}$, we have $\T_0(w)=0$. 
\end{lemma}

If $\T_0(g)=0,$ then the same property is satisfied by $wg$ if $w\in\mathbf{k}$. Hence, weighted $L^{\infty}_{x,v}$ norms are conserved for solutions to the Vlasov equation. In the nonlinear small data regime, these norms will grow logarithmically in time and will then provide useful decay properties for the Vlasov field.

It will be useful to work with $$g_0(t,x,v):= f(t,x+vt,v),$$ when studying the asymptotic properties of $\rho(f)$ and its derivatives. The following lemma suggests that $g_0$ enjoys strong decay properties.

\begin{lemma}\label{lem_properties_linear_change_variable_trivial}
Let $f\colon [0,\infty)\times \R^3_x\times \R^3_v\to \R$ be a regular distribution function and $g_0(t,x,v):=f(t, x+tv,v)$. Then, we have 
$$\langle x\rangle^{N_x} \langle v\rangle^{N_v}  |g_0|(t,x,v)= \langle x\rangle^{N_x} \langle v\rangle^{N_v}|f|(t,x+vt,v),$$
and
$$ \nabla_x g_0(t,x,v)=\nabla_x f(t,x+tv,v), \qquad \nabla_v g_0(t,x,v)= Gf(t,x+tv,v)= \big[ t\nabla_xf+\nabla_vf\big](t,x+tv,v).$$
\end{lemma}

\begin{remark}There is an explicit correspondence between the commuting vector fields in $\boldsymbol{\lambda}$ and the weights in $\mathbf{k}$. We have $\partial_{x^i}=\{v^i,\cdot\}$ and $G_i=\{x^i-v^i t,\cdot\}$, where  $\{\cdot,\cdot\}$ is the Poisson bracket of the standard symplectic structure on $\R_x^3\times \R_v^3$.
\end{remark}

\subsection{Multi-index notations}

Let $(Z^i)_i$ be an arbitrary ordering of the vector fields in $\boldsymbol{\lambda}$. We use a multi-index notation for the differential operators of order $|\alpha|$ given by 
\begin{equation*}
Z^{\alpha}:=Z^{\alpha_1} Z^{\alpha_2}\dots Z^{\alpha_{n}},
\end{equation*}
for every $\alpha\in \N^{6}$. We denote by $\boldsymbol{\lambda}^{|\alpha|}$ the set of differential operators obtained as a composition of $|\alpha|$ vector fields in $\boldsymbol{\lambda}$. 

We can uniquely associate a differential operator on $\R^3_x$ to any differential operator $Z^{\alpha}\in \boldsymbol{\lambda}^{|\alpha|}$ by replacing every vector field $Z$ on $\R^3_x\times \R^3_v$ by the corresponding vector field $Z_x$ on $\R^3_x$. As a result, we have 
\begin{enumerate}[label = (\alph*)]
    \item translation vector fields $T_{i,x}:=\partial_{x^i}$,
    \item projected Galilean vector fields $G_{i,x}:=t\partial_{x^i}$.
\end{enumerate}
We define 
\begin{equation*}
\mathbf{\Lambda}:=\Big\{ T_{i,x},G_{i,x}: ~i\in \{1,2,3\}  \Big\}.
\end{equation*}
We denote by $\boldsymbol{\Lambda}^{|\alpha|}$ the family of differential operators on $\R^3_x$ of order $|\alpha|$ obtained as a composition of $|\alpha|$ vector fields in $\boldsymbol{\Lambda}$. By a small abuse of notation, we denote by $Z^{\alpha}$ the associated differential operator on $\R^3_x$ for an arbitrary differential operator $Z^{\alpha}$ on $\R^3_x\times \R^3_v$.

Finally, for every $\alpha \in \mathbb{N}^3$, we denote by $\partial_x^{\alpha}$, $\partial_v^\alpha$, and $G^\alpha$, the differential operators $$\partial_x^{\alpha}:=\partial^{\alpha_1}_{x^{1}}\partial^{\alpha_2}_{x^2}\partial^{\alpha_3}_{x^3},  \qquad \partial_v^{\alpha}:=\partial^{\alpha_1}_{v^{1}}\partial^{\alpha_2}_{v^2}\partial^{\alpha_3}_{v^3},  \qquad G^\alpha = G_1^{\alpha_1}G_2^{\alpha_2} G_3^{\alpha_3}.$$

\begin{lemma}\label{lemma_commuting_in_lambda}
Let $\alpha$ and $\beta$ be two multi-indices. Then, the commutator $[Z^{\alpha},Z^{\beta}]$ vanishes. 
\end{lemma}

Next, we relate the derivatives $Z^\alpha_x \rho(f)$ and $\rho(Z^\alpha f)$.
\begin{lemma}\label{lemma_connection_microscopic_macroscopic_vector_fields}
Let $f$ be a regular distribution function, and let $\alpha$ be a multi-index. Then, 
\begin{equation}
    Z^{\alpha}_x\rho(f)=\rho(Z^{\alpha}f)
\end{equation}
where $Z_x^\alpha \in \mathbf{\Lambda}^{|\alpha|}$ and $Z^\alpha \in \boldsymbol{\lambda}^{|\alpha|}$.
\end{lemma}

\subsection{The commuted equations}

Let us denote the transport operator in the Vlasov--Poisson system by $$\T_\phi :=\partial_t +v\cdot\nabla_x -\mu\nabla_x\phi\cdot \nabla_v.$$ Here, the force field $\nabla_x\phi$ is defined through the Poisson equation $\Delta_x \phi=\rho(f)$.

\begin{lemma}\label{LemComfirstorder}
Let $Z \in \boldsymbol{\lambda}$, and $Z_x \in \mathbf{\Lambda}$ be the corresponding vector field on $\R^3_x$. Then,
$$ [\T_\phi,Z]=\mu \nabla_x Z_x \phi \cdot \nabla_v .$$
\end{lemma}

Iterating the lemma above, we obtain the higher order case.

\begin{lemma}\label{lemma_commuted_nonliner_Vlasov}
There exist constant coefficients $C^{\beta}_{\alpha \gamma} \in \mathbb{Z}$ such that 
\begin{equation}\label{eq_comm}
    [\T_{\phi},Z^{\beta}]=\sum_{|\alpha|\leq |\beta|-1} \, \sum_{\gamma+\alpha= \beta} C^{\beta}_{\alpha \gamma}\nabla_x Z^{\gamma}_x\phi\cdot  \nabla_v Z^{\alpha},
\end{equation}
where $Z^{\beta}\in\boldsymbol{\lambda}^{|\beta|}$, $Z^{\gamma}_x\in \boldsymbol{\Lambda}^{|\gamma|}$, and $Z^{\alpha}\in \boldsymbol{\lambda}^{|\alpha|}$. 
\end{lemma}

\begin{remark}
In particular, if $Z^\beta$ contains $p$ Galilean vector fields, then $Z^\beta=\partial_x^\kappa G^\xi$ with $|\xi|=p$.
\end{remark}

We conclude this subsection with the commuted Poisson equation.

\begin{lemma}\label{lemma_commuted_poisson_equation}
Let $f$ be a regular distribution function, and let $\phi$ be the solution to the Poisson equation $\Delta_x \phi = \rho(f)$. Then, for any multi-index $\alpha$ the function $Z^{\alpha}_x\phi$ satisfies 
\begin{equation*}
\Delta_x Z^{\alpha}_x\phi=\rho( Z^{\alpha} f).
\end{equation*} 
\end{lemma}

\subsection{Asymptotic self-similar polyhomogeneous expansions}

We now introduce the terminology that we will use for the expansions satisfied by the normalised spatial density and the normalised force field. We further prove some preparatory results.
\begin{definition}\label{Defselsimexp}
Let $N_0 \in \mathbb{N}$ and $\varrho : \R_+ \times \R^3_x \to \R$. We say that $\varrho$ admits an \emph{asymptotic self-similar polyhomogeneous expansion of order $N_0$} if there exist functions $[\varrho]_{q,p} : \R_x^3 \to \R$ and $S \in \mathbb{N}$ such that
\begin{equation}\label{defn_polyhom_spatialdensit}
 \forall \, (t,x)\in [2,\infty)  \times \R^3_x, \qquad  \bigg|\varrho (t,x)  -  \sum_{0 \leq p \leq q \leq N_0}  \frac{\log^p(t)}{t^q}[\varrho]_{q,p} \Big( \frac{x}{t} \Big) \bigg| \lesssim \frac{\log^S(t)}{\langle t+|x| \rangle^3t^{N_0-2}} .
\end{equation}
\end{definition}

\begin{remark}
The spatial decay in \eqref{defn_polyhom_spatialdensit} will be important to perform elliptic estimates (see \eqref{remark_uniform_integral_bound_kernel_convolution_duan} below).
\end{remark}

The next definition will be applied to quantities of the form $\nabla_x \partial_x^\gamma\phi (t,\X_n+t\V_n)$, that is, to derivatives of the force field along corrections to the spatial linear characteristics $t \mapsto x+tv$.

\begin{definition}\label{defpolhomoexpforce}
Let $N_0 \in \mathbb{N}$ and $\Psi : \R_+ \times \R^3_x \times \R^3_v \to \R$. We say that $\Psi$ admits an \emph{asymptotic polyhomogeneous expansion of order $N_0$} if there exist functions $[\Psi]_{q,\alpha,p} : \R^3_v \to \R$ and $S \in \mathbb{N}$ such that, for all $(t,x,v) \in [2,\infty) \times \R^3_x \times \R^3_v$ with $|x| \leq t$, we have
\begin{equation}\label{eq:Expan}
  \bigg| \Psi (t,x,v)  -   \sum_{q \leq N_0} \,  \sum_{ |\alpha|+p \leq q} \frac{x^\alpha \log^p(t)}{t^{q}}[\Psi]_{q,\alpha,p} (v) \bigg| \lesssim \frac{\langle x \rangle^{N_0+1} \log^S(t)}{t^{N_0+1}} .
   \end{equation}
Consider a function $\psi :\R_+ \times \R^3_x  \to \R$ and a family of curves $\gamma_{x,v} : [2,\infty) \to  \R^3_x$ parameterised by $(x,v) \in \R^3_x \times \R^3_v$. We say that $\psi$ admits an \emph{asymptotic polyhomogeneous expansion of order $N_0$ along the curves $\gamma_{x,v}$} if $\Psi (t,x,v):= \psi \big(t, \gamma_{x,v}(t) \big)$ verifies \eqref{eq:Expan}.
\end{definition}

Let us now prove the uniqueness of the coefficients in the expansions. We also show that under suitable assumptions on $\Psi$, we can differentiate \eqref{eq:Expan} without losing the structure of the expansion.
\begin{lemma}\label{LemForexp}
Let $N_0 \in \mathbb{N}$ and $\Psi \in C^0 \cap L^\infty( \R_+ \! \times \R^3_x \times \R^3_v)$ admitting an asymptotic polyhomogeneous expansion of order $N_0$. Then:
\begin{enumerate}
\item The functions $[\Psi]_{q,\alpha,p}$ are unique, and belong to $C^0 \cap L^\infty (\R^3_v)$.
\item Let $1 \leq i \leq 3$. Assume that $\partial_{v^i} \Psi$ and $t\partial_{x^i} \Psi$ admit asymptotic polyhomogeneous expansions of order $N_0' \leq N_0$. Then, for any $q \leq N'_0$ and any $|\alpha|+p \leq q$, we have
$$ \partial_{v^i} [\Psi]_{q,\alpha,p}=[\partial_{v^i}\Psi]_{q,\alpha,p}.$$
Moreover, if $N_0' \leq N_0-1$, $\alpha\in\N^3$, and $\overline{\alpha}:=\alpha+(\delta_1^i,\delta_2^i,\delta_3^i)$, then $$ [t\partial_{x^i}\Psi]_{q,\alpha,p}=(\alpha_i+1) [\Psi]_{q,\overline{\alpha},p}.$$
\end{enumerate}
\end{lemma}
\begin{remark}
Note that if $N_0'=N_0$, we cannot compute $[t\partial_{x^i}\Psi]_{q,\alpha,p}$ using the coefficients of the expansion of $\Psi$.
\end{remark}
\begin{proof}
The uniqueness of $[\Psi]_{q,\alpha , p }$ follows from an induction, on $(q,p)$ for the lexicographical order, and the uniqueness of the (pointwise) limit. Concerning the regularity, we use the convergence in $L^{\infty}(K \times \R^3_v)$, for any compact subset $K \subset \R^3_x$, of
$$\frac{t^{n}}{\log^{m}(t)}\bigg( \Psi (t,x,v)  -  \sum_{ q \leq n-1} \, \sum_{|\alpha|+p \leq q}  \frac{x^\alpha \log^p(t)}{t^{q}}[\Psi]_{q,\alpha,p} (v)+ \sum_{p \leq m} \,  \sum_{|\alpha| \leq n-p}  \frac{x^\alpha \log^p(t)}{t^n}[\Psi]_{q,\alpha,p} (v)\bigg),$$
where $0 \leq n \leq N_0$, $0 \leq m \leq n$. For the second statement, we use the convergence in $L^{\infty}(K \times \R^3_v)$ of the previous quantity corresponding to $\partial_{v^i} \Psi$ and $t \partial_{x^i} \Psi$.
\end{proof}

By similar considerations, we also have the next result.

\begin{lemma}\label{LemForexpbis}
Let $N_0 \in \mathbb{N}$ and $\varrho \in C^0 \cap L^\infty( \R_+ \! \times \R^3_x )$ admitting an asymptotic self-similar polyhomogeneous expansion of order $N_0$. Then:
\begin{enumerate}
\item The functions $[\varrho]_{q,p}$ are unique, and belong to $C^0 \cap L^\infty (\R^3_v)$.
\item Assume that $t\nabla_x \varrho$ admits an expansion of order $N_0' \leq N_0$. Then, for any $p \leq q \leq N'_0$, we have
$$  t\nabla_x[ \varrho]_{p,q}=[t\nabla_x \varrho]_{p,q}.$$
\end{enumerate}
\end{lemma}

\section{Statement of the main results}\label{secmain}

We now provide a detailed formulation of the main results of the article. 

\subsection{Small data modified scattering in high order regularity}

First, we prove small data global existence for the Vlasov--Poisson system in high regularity. We provide estimates for weighted $L^{\infty}_{x,v}$ norms for the distribution function. These properties will be used later when studying the scattering properties of small data solutions.

\begin{theorem}\label{thm_detail_modified_scattering}
Let $\epsilon >0$ and $N \geq 1$. Let $f_{t=0}$ be an initial data of class $C^N$ for the Vlasov--Poisson system such that 
\begin{equation}\label{eq:smallnessassump}
\sup_{|\kappa| \leq N} \, \sup_{(x,v) \in \R^3_x \times \R^3_v } \langle x \rangle^8 \, \langle v \rangle^7 \big| \partial_{x,v}^\kappa f_{t=0} \big|(x,v) \leq \epsilon.
\end{equation}
\textbf{Global existence.} There exists $\varepsilon_0>0$, depending only on $N$, such that the following statement holds. If $\epsilon \leq \varepsilon_0$, then there exists a unique global solution $f$ to the Vlasov-Poisson system \eqref{eq:VP} arising from the initial data $f_{t=0}$. Let $N_x\geq 8$, and $N_v\geq 7$. Then, if
\begin{equation}\label{eq:defininorm}
 \mathbb{E}_N^{N_x,N_v}[f_{t=0}]:= \sup_{|\kappa| \leq N} \, \sup_{(x,v) \in \R^3_x \times \R^3_v } \langle x \rangle^{N_x} \langle v \rangle^{N_v} \big| \partial_{x,v}^\kappa f_{t=0}|(x,v) <\infty ,
\end{equation}
there exists a constant $C[N_x,N_v]>0$ depending only on $(N_x,N_v)$ such that, for any $|\kappa_x|+|\kappa_v| \leq N$ and all $(t,x,v) \in \R_+ \times \R^3_x \times \R^3_v$, we have for $g_0(t,x,v):=f(t,x+tv,v)$ that
\begin{equation}\label{estimate_point_globa_exist}
 \langle x \rangle^{N_x}  \langle v \rangle^{N_v} \big| \partial_{x,v}^\kappa g_0 (t,x,v) \big| \leq \left\{ 
	\begin{array}{ll}
        C[N_x,N_v] \, \mathbb{E}_N^{N_x,N_v}[f_{t=0}] \log^{N_x+|\kappa_v|}(2+t) \quad & \mbox{if $|\kappa_x|+|\kappa_v| \leq N-1$, } \\[3pt]
       C[N_x,N_v] \, \mathbb{E}_N^{N_x,N_v}[f_{t=0}]\,  \langle t \rangle^{\frac{1}{3}} & \mbox{if $|\kappa_x|+|\kappa_v|=N$}.
    \end{array} 
\right. 
\end{equation}
\textbf{Modified scattering.} Assume further that $N \geq 2$. Then:
\begin{enumerate}[label=(\alph*)]
\item \label{first_prop_mod_scatt_main} The spatial average of $ f$ converges to the function $Q_{\infty} \in C^{N-2}\cap W^{N-2,\infty}(\R^3_v)$. For any $|\beta|\leq N-2$, we have $$\forall (t,v)\in [2,\infty) \times \R^3_v, \qquad \langle v \rangle^{4}\bigg|\int_{\R^3_v}Z^\beta f(t,x,v)\mathrm{d}x -\partial_v^{\beta}Q_{\infty}(v)\bigg|\lesssim \e  \frac{\log^{N-1}(t)}{t} .$$ 

\item \label{second_prop_mod_scatt_main} The spatial density $\rho(Z^{\beta}f)$ has a self-similar asymptotic profile. For any $|\beta| \leq N-2$, we have $$\forall (t,x) \in [2,\infty) \times \R^3_x,\qquad \bigg| t^{3}\int_{\R^3_v}Z^{\beta}f(t,x,v)\mathrm{d}v-\partial_v^{\beta}Q_{\infty}\Big(\dfrac{x}{t}\Big)\bigg|\lesssim \e  \frac{\log^{N}(t)}{t}.$$

\item The distribution function $f$ satisfies modified scattering to a distribution $f_{\infty} \!\in C^{N-2} \cap~W^{N-2,\infty}(\R^3_x \times \R^3_v)$. Let $$\X_1(t,x,v):=x+ \mu \log(t) \nabla_v \phi_\infty(v)\qquad \textrm{and}\qquad \V_1(t,x,v):=v+\mu t^{-1} \nabla_v \phi_\infty (v).$$ For any $|\beta|\leq N-2$ and all $(t,x,v)\in [2,\infty) \times \R^3_x \times \R^3_v$, we have 
$$\color{white} \square \quad \quad \color{black} \langle x\rangle^{N_x-1}\langle v \rangle^{N_v} \Big| \partial_{x,v}^{\beta}\big[f(t,\X_1+t\V_1,\V_1 )\big] - \partial_{x,v}^{\beta}f_{\infty}(x,v) \Big| \lesssim  \log^{2N}(t)t^{-1}.$$ 
\end{enumerate}
\end{theorem}
\begin{remark}\label{RqModonederiv}
The estimate \eqref{estimate_point_globa_exist} for the top order derivatives could be improved to $\langle t \rangle^{\delta}$ with $\delta >0$, but it would require to consider a smaller $\varepsilon_0$. Because of the weaker control on the top order derivatives, we prove modified scattering for the derivatives up to order $N-2$. Nonetheless, one could easily adapt our proof to derive that $f_\infty \in C^{N-1} (\R^3_x \times \R^3_v)$.
\end{remark}

\subsection{High order late-time asymptotics}

We now state the main result of the article. Recall for this the sequence $(r_n)_{n\geq 1}$ introduced in \eqref{kev:defSn}.

\begin{theorem}\label{ThLatetimeFull}
Let $N \geq 3$, an integer $n \geq 1$ with $r_n\leq N$, $ \X_0(t,x,v) := x,$ and $\V_0(t,x,v):=v$. Let $f$ be a small data solution to the Vlasov--Poisson system arising from initial data satisfying the smallness assumption \eqref{eq:smallnessassump} as well as
\begin{equation}\label{eq:assump}
 \mathbb{E}_N^{N_x,N_v}[f_{t=0}] <\infty, \qquad N_x \geq 2\sqrt{2N}+5, \quad N_v \geq 7.
 \end{equation} There exist modifications of the linear characteristic flow
\begin{align*}
\X_{n}(t,x,v) & =x+ \mu \nabla_v \phi_{\infty}(v) \log(t)+\sum_{1 \leq q \leq n-1} \, \sum_{|\alpha|+ p \leq q} \frac{x^\alpha \log^p(t)}{t^{q}} \mathbb{X}_{q,\alpha,p}(v), \\
\V_{n}(t,x,v) &=v+  \frac{\mu}{t}\nabla_v \phi_{\infty}(v)+\sum_{1 \leq q \leq n-1} \, \sum_{|\alpha|+ p \leq q} \frac{x^\alpha \log^p(t)}{t^{q+1}} \mathbb{V}_{q,\alpha,p}(v),
\end{align*}
where $\mathbb{X}_{q,\alpha,p}, \; \mathbb{V}_{q,\alpha,p} \in C^{N-2-2q}\cap W^{N-2-2q} (\R^3_v),$ such that the following properties hold. For any integer $n \geq 0$ and multi-index $\beta$ such that $ r_{n+1} \leq N$ and $|\beta| \leq N-r_{n+1}$: 
\begin{enumerate}[label = (\alph*)]
 \item  The normalised spatial density $t^3 \!\rho (G^\beta \! f)$ and normalised force field $t^{2+|\beta|}\nabla_x \partial_x^\beta \phi$ admit an asymptotic self-similar polyhomogeneous expansion of order $n$ according to Definition \ref{Defselsimexp}. 
 
\vspace{1mm} 
 
\item Modified scattering with an enhanced rate of convergence. Let $g_{n+1}:\R_+^*\times \R^3_x\times\R^3_v\to \R$ be defined by
$$ g_{n+1}(t,x,v) := f \big(t,\X_{n+1}(t,x,v)+t\V_{n+1}(t,x,v),\V_{n+1}(t,x,v) \big).$$ There exists a regular distribution $f_{\infty}\colon \R^3_x \times \R^3_v \to \R$ such that for any $ |\kappa| \leq N-r_{n+1}$ and all $(t,x,v) \in [2,\infty) \times \R^3_x \times \R^3_v$ with $|x| \leq t$, we have
\begin{equation*}
  \langle x \rangle^{N_x-1-n} \, \langle v \rangle^{N_v} \big|\partial_{x,v}^\kappa \big[ g_{n+1}(t,x,v)-f_\infty (x,v) \big] \big| \lesssim  \frac{\log^{N(n+3)}(t)}{t^{n+1}}.
  \end{equation*}
  
\vspace{1mm}  
  
\item The modified spatial average verifies an enhanced convergence to the spatial average of $f_\infty$. For any $|\kappa| \leq N-1-r_{n+1}$, 
there exists $\mathbf{Q}_{p,\xi}^{\kappa,\beta} \in C^0 \cap L^\infty (\R^3_v) $ such that for all $(t,v) \in [2,\infty) \times \R^3_v$, we have 
\begin{align*}
\color{white} \square \qquad  \quad \color{black} \bigg| \int_{|x|<t } \partial_v^\kappa g_{n+1}(t,x,v) \dr x -\int_{\R^3_x} \partial_v^\kappa f_\infty (x,v) \dr x- \sum_{|\beta| \leq |\kappa|} \sum_{p +|\xi| \leq n}& \frac{\log^p(t)}{t^{n+1}} \mathbf{Q}_{p,\xi}^{\kappa,\beta} (v)\int_{\R^3_x} x^\xi \partial_v^\beta f_\infty (x,v) \dr x \bigg| \\
  \lesssim t^{-n-2} \log^{N(n+3)}(t)&.
  \end{align*}  
\end{enumerate}
\end{theorem}
\begin{remark}
The first order asymptotic self-similar polyhomogeneous expansion of $t^3 \rho(f)$ is written in full details in Proposition \ref{ProfirstorderexpansionRho}. In particular, one can see that the self-similar coefficients of the expansion with factors $1$, $t^{-1}\log(t)$, and $t^{-1}$, do not vanish in general. 
\end{remark}

\begin{remark}
In the expansions for $t^3\rho (G^\beta f)$ and $t^{2+|\beta|}\nabla_x \partial_x^\beta \phi$, the coefficients of index $(p,q)$ are related through a Poisson equation. See Proposition \ref{Proexpa} for more details. These self-similar expansions are written in terms of functions $v\mapsto \int_x f_\infty (x,v) \dr x$ and its derivatives up to order $|\beta|$.
\end{remark}
\subsection{Non-linear tails and weak convergence}

We next show non-linear late-time tails for the spatial density and the force field. For this, we define a hierarchy of asymptotic conservation laws for the solutions of the Vlasov--Poisson system. 

Let $f_{\infty}\in C^{N-2}(\R^3_x\times\R^3_v)$ be a regular scattering state. Let $|\alpha|+|\beta| \leq N-2$ be multi-indices. We consider the \emph{weighted spatial averages} $\mathcal{A}^{\alpha}_{\beta}:\R^3_v\to\R$ given by 
\begin{equation*}
\mathcal{A}_{\beta}^{\alpha}(v):=\int_{\R^3_x} x^{\alpha}\partial_v^{\alpha+\beta}f_{\infty}(x,v)\dr x.
\end{equation*}
The weighted spatial averages $\mathcal{A}_{\beta}^{\alpha}$ can be characterised as $$\mathcal{A}^{\alpha}_{\beta}(v)=\lim_{t\to \infty} \int_{\R^3_x}x^{\alpha}\partial_{v}^{\alpha+\beta}g_1(t,x,v)\dr x,$$ as we will prove in Section \ref{SecRhoOrder2}. 

\begin{definition}
Let $f\colon[0,\infty)\times \R^3_x\times\R^3_v\to \R$ be a regular small data solution for the Vlasov--Poisson system on $\R^3_x\times\R^3_v$. The \emph{hierarchy of asymptotic conservation laws for the Vlasov field} $f$ is defined as $$\mathcal{A}(f):=\Big\{\mathcal{A}^{\alpha}_{\beta}(v)=\int_{\R^3_x} x^{\alpha}\partial_v^{\alpha+\beta}f_{\infty}(x,v)\dr x : \alpha,\beta \in \N^3\Big\}.$$
\end{definition}

We can now state in details the late-time tails for the spatial density and the force field.

\begin{theorem}
Let $N\geq 4$ and $n \in \mathbb{N}$ such that $r_{n+1}\leq N$. Consider further $|\beta| \leq N-r_{n+1}-1$ and $f$ a small data solution to the Vlasov--Poisson system arising from initial data satisfying the assumptions of Theorem \ref{thm_detail_modified_scattering} and Theorem \ref{ThLatetimeFull}. There exist constants $C^{p,q}\in\R$ such that the spatial density satisfies
\begin{align*}
\forall t \geq 2, \qquad \bigg|t^3\int_{\R^3_v}G^{\beta}f(t,x,v)\dr v-\sum_{p\leq q\leq n}\sum_{|\gamma|\leq n-q} \frac{C^{p,q}}{\gamma !}\partial_v^{\gamma+\beta}\mathbf{F}_{p,q}(0)\frac{x^{\gamma}\log^p (t)}{t^{|\gamma|+q}}\bigg|\lesssim \dfrac{\log^{N(n+3)}(t) }{t^{n+1}} \langle x \rangle^{n+1},
\end{align*} 
where $\partial_v^{\gamma+\beta}\mathbf{F}^{\beta}_{p,q}(v)$ can be computed in terms of $\mathcal{A}_{\beta}^{\alpha}(v)$ and its derivatives for $|\alpha|+|\beta|\leq N-2$. Moreover, the force field satisfies
\begin{align*}
\forall t\geq 2\qquad  \bigg|t^{2+|\beta|} \nabla_x \partial_x^\beta \phi (t,x)-\sum_{p\leq q\leq n}\sum_{|\gamma|\leq n-q} \frac{C^{p,q}}{\gamma !}\nabla_v \partial_v^{\gamma+\beta}\Phi_{p,q}(0)\frac{x^{\gamma}\log^p (t)}{t^{|\gamma|+q}}\bigg|\lesssim \dfrac{\log^{N(n+3)}(t) }{t^{n+1}} \langle x \rangle^{n+1},
\end{align*}
where $\Phi_{p,q}$ is defined by $\Delta_v\Phi_{p,q}=\mathbf{F}_{p,q}$.
\end{theorem}

We next show the concentration of the support of the distribution in the zero velocity set in a weak convergence sense. We also show the shearing of the solutions of the system in a weak convergence sense for $f(t,x+\bar{v}t,v)$ for a fixed $\bar{v}\in\R^3_v$.

\begin{theorem}\label{thm_weak_convergence_prop_complete}
Let $\varphi\in C^{\infty}_{x,v}$ be a compactly supported test function. Let $\bar{v}\in\R^3_v$, and $|\beta|\leq N-2$. Let $f$ be a small data solution to the Vlasov--Poisson system arising from initial data satisfying the assumptions of Theorem \ref{thm_detail_modified_scattering}. Then, the distribution $f$ satisfies $$\lim_{t\to\infty}\int_{\R^3_x\times \R^3_v}t^3G^{\beta}f(t,x,v)\varphi(x,v)\dr x\dr v= \int_{\R^3_x\times \R^3_v} \Big(\mathcal{A}_{\beta}(0) \delta_{v=0}(v)\Big)\varphi(x,v) \dr x\dr v.$$ In other words, the distribution $t^{3}G^{\beta}f(t,x,v)$ converges weakly to $\mathcal{A}_{\beta}(0)\delta_{v=0}(v)$ as $t \to  \infty$. Moreover, the distribution $f$ satisfies $$\lim_{t\to\infty}\int_{\R^3_x\times \R^3_v}t^3G^{\beta}f(t,x+t\bar{v},v+\bar{v})\varphi(x,v)\dr x\dr v=\int_{\R^3_x\times \R^3_v} \Big(\mathcal{A}_{\beta}(v) \delta_{v=\bar{v}}(v)\Big)\varphi(x,v) \dr x\dr v.$$ In other words, the distribution $t^{3}G^{\beta}f(t,x+t\bar{v},v+\bar{v})$ converges weakly to $\mathcal{A}_{\beta}(v)\delta_{v=\bar{v}}(v)$ as $t \to  \infty$.
\end{theorem}

\section{Asymptotics of the linearised system}\label{seclinear}

In this section, we address the problem of late-time asymptotics for the linearised system \eqref{vlasov_linear_flow}. In this case, we show that the spatial density satisfies an asymptotic self-similar expansion. As an application, we obtain late-time tails for the spatial density, where the coefficients in the tails are \emph{exact} conservation laws for the linearised system. Finally, we prove that the distribution function (up to normalisation) converges weakly to a Dirac mass in the zero velocity set. 

\subsection{Asymptotic self-similar expansion for the spatial density}

In this subsection, we show an asymptotic self-similar expansion for the velocity average $t^3\rho(G^{\beta}f)$ in terms of the \emph{weighted spatial averages} 
\begin{equation*}
\int_{\R^3_x} x^{\alpha}\partial_v^{\alpha+\beta}f_{0}(x,v)\dr x.
\end{equation*} 
of the initial distribution function $f_{0}$. 

\begin{proposition}\label{Prolin}
Let $N\in\N$. Let $f_{0}$ be a regular initial data for the Vlasov equation on $\R^3_x\times \R^3_v$. Then, there exist universal constants $C_{\alpha}\in\R$ such that the solution $f$ to the Vlasov equation satisfies
\begin{equation*}
\bigg|t^3\int_{\R^3_v}f(t,x,v)\dr v-\sum_{|\alpha|\leq N}\dfrac{C_{\alpha}}{t^{|\alpha|}}\int_{\R^3_y} y^{\alpha}\partial^{\alpha}_{v}f_{0}\Big(y,\frac{x}{t}\Big)\dr y\bigg|\lesssim \dfrac{1}{t^{N+1}}\sum_{|\alpha| =  N+1} \sup_{(z,v)\in\R^3_z \times \R^3_v}  \langle z\rangle^{N+5}|\partial_v^{\alpha}f_{0}|(z,v)
\end{equation*}
for all $t\geq 2$ and $x \in \R^3_x$.
\end{proposition}

\begin{proof}
Taylor expanding the function $z\mapsto f_{0}(y,\frac{x-z}{t})$ around $z=0$, we obtain $$f_{0}\Big(y,\frac{x-y}{t}\Big)=\sum_{|\alpha|\leq N} \dfrac{y^{\alpha}}{\alpha!}D_z^{\alpha}\Big(f_{0}\Big(y,\frac{x-z}{t}\Big)\Big)\Big|_{z=0}+\sum_{|\beta|=N+1}R_\beta(y)y^{\beta},$$ where the coefficients $R_\beta(y)$ of the remainder satisfy 
$$|R_\beta(y)|\leq \frac{1}{\beta!} \sup_{|\gamma|=|\beta|} \, \sup_{|z'| \leq |x|} \Big|D_z^{\gamma}\Big(f_{0}\Big(y,\frac{x-z}{t}\Big)\Big) \Big|_{z=z'}\Big|.$$ 
By the Faà di Bruno formula, there exist coefficients $C_{\alpha}\in\R$ such that 
$$\sum_{|\alpha|\leq N} \dfrac{y^{\alpha}}{\alpha!}D_z^{\alpha}\Big(f_{0}\Big(y,\frac{x-z}{t}\Big)\Big)\Big|_{z=0}=\sum_{|\alpha|\leq N}\dfrac{C_{\alpha}}{t^{|\alpha|}} y^{\alpha}\partial^{\alpha}_{v}f_{0}\Big(y,\frac{x}{t}\Big).$$ 
Applying the Faà di Bruno formula again, the remainder is bounded by 
\begin{align*}
|R_\beta(y)y^{\beta}|&\lesssim  \sup_{|\gamma|=|\beta|} \sup_{|z'| \leq |x|}\Big|D_z^{\gamma}\Big(f_{0}\Big(y,\frac{x-z}{t}\Big)\Big) \Big|_{z=z'}\Big||y|^{N+1}\lesssim \dfrac{1}{t^{N+1}}\sum_{|\alpha| = N+1} \sup_{v\in\R^3}\langle y\rangle^{N+1}|\partial_v^{\alpha}f_{0}|(y,v).
\end{align*}
Finally, we obtain the estimate 
\begin{align*}
\bigg|t^3\int_{\R^3_v}f(t,x,v)\dr v-&\sum_{|\alpha|\leq N}\dfrac{C_{\alpha}}{t^{|\alpha|}}\int_{\R^3_y} y^{\alpha}\partial^{\alpha}_{v}f_{0}\Big(y,\frac{x}{t}\Big)\dr y\bigg|\\
&\lesssim \dfrac{1}{t^{N+1}}\sum_{|\alpha| = N+1} \sup_{(z,v)\in \R^3_x\times\R^3_v}  \langle z\rangle^{N+5}|\partial_v^{\alpha}f_{0}|(z,v)\int_{\R^3_y}\dfrac{\dr y}{\langle y\rangle^4}\\
&\lesssim \dfrac{1}{t^{N+1}}\sum_{|\alpha| = N+1} \sup_{(z,v)\in \R^3_x\times\R^3_v}  \langle z\rangle^{N+5}|\partial_v^{\alpha}f_{0}|(z,v).
\end{align*}
\end{proof}

Applying Proposition \ref{Prolin} to the distribution $G^{\beta}f$ we obtain the following result.

\begin{corollary}
Let $N\in\N$. Let $f_{0}$ be a regular initial data for the Vlasov equation \eqref{vlasov_linear_flow} on $\R^3_x\times\R^3_v$. Then, there exist universal constants $C_{\alpha}\in\R$ such that the solution $f$ to the Vlasov equation satisfies 
$$\bigg|t^3\!\int_{\R^3_v} \!G^{\beta}\!f(t,x,v)\dr v- \! \sum_{|\alpha|\leq N} \!\dfrac{C_{\alpha}}{t^{|\alpha|}} \! \int_{\R^3_y} y^{\alpha}\partial^{\alpha+\beta}_{v}f_{0}\Big(y,\frac{x}{t}\Big)\dr y\bigg|\lesssim \dfrac{1}{t^{N+1}}\!\sum_{|\alpha|=N+1} \sup_{(z,v)\in\R^3\times\R^3}  \langle z\rangle^{N+5}|\partial_v^{\alpha+\beta}f_{0}|(z,v),$$ 
for all $t\geq 2$ and $ x \in \R^3_x$.
\end{corollary}

\begin{proof}
Note that $G^{\beta}f(0,x,v)=\partial^{\beta}_{v}f_{0}$. The result follows by Proposition \ref{Prolin} to the solution $G^{\beta}f$ of the Vlasov equation.
\end{proof}

\subsection{Hierarchy of conservation laws}

Let $f\colon \R_t\times \R^3_x\times\R^3_v\to \R$ be a regular solution for the free Vlasov equation on $\R^3_x\times\R^3_v$. Let $\alpha,\beta \in \N^3$ be multi-indices. We consider the \emph{weighted spatial averages} $\mathcal{A}^{\alpha}_{\beta}(f):\R^3_v\to \R$ given by 
\begin{equation*}
\mathcal{A}^{\alpha}_{\beta}(f)(v):=\int_{\R^3_x} (x-vt)^{\alpha}(t\partial_x+\partial_v)^{\alpha+\beta}f(t,x,v)\dr x.
\end{equation*}
We use the convention $\mathcal{A}^{\alpha}(f):=\mathcal{A}^{\alpha}_{0}(f)$. We first show that $\mathcal{A}^{\alpha}_{\beta}(f)$ is well-defined. In other words, we prove that for a Vlasov field on $\R^3_x\times \R^3_v$, the weighted spatial averages $\mathcal{A}^{\alpha}_{\beta}(f)$ are conserved in time.

\begin{proposition}\label{prop_cons_laws_linear}
Let $f\colon \R_t\times \R^3_x\times\R^3_v\to \R$ be a regular solution to the Vlasov equation on $\R^3_x\times\R^3_v$. Then, the weighted spatial average $\mathcal{A}^{\alpha}_{\beta}(f)$ satisfies
\begin{equation*}
\mathcal{A}^{\alpha}_{\beta}(f)(v)=\int_{\R^3_x} x^{\alpha}\partial_v^{\alpha+\beta}f_{0}(x,v)\dr x,
\end{equation*}
for every $\alpha,\beta \in \N^3$, and all $(t,v)\in \R_t\times \R^3_v$.
\end{proposition}

\begin{proof}
We first prove the conservation law for $\mathcal{A}^{0}_{0}(f)$. By the Vlasov equation, we have
\begin{equation*}
\dfrac{\dr}{\dr t}\mathcal{A}^{0}_{0}(f(t))=\int_{\R^3_x} \partial_t f(t,x,v)\dr x=-\int_{\R^3_x} v\cdot \nabla_x f(t,x,v)\dr x=0,
\end{equation*}
In particular, the weighted spatial average $\mathcal{A}^{0}_{0}(f)$ is constant in time. For the general case, we note that $(x-vt)^{\alpha}(t\partial_x+\partial_v)^{\alpha+\beta}f(t,x,v)$ is also a solution of the Vlasov equation since $x-vt$ is conserved along the characteristic flow, and $t\partial_x+\partial_v$ is a commuting vector field. Thus, the conservation law is obtained similarly as for $\mathcal{A}^{0}_{0}(f)$. 
\end{proof}

We next show that $\mathcal{A}^{\alpha}_{\beta}(f)$ is controlled by a weighted $L^{\infty}_{x,v}$ norm of the initial distribution function $f_{0}$.

\begin{proposition}\label{prop_def_cons_laws}
Let $\alpha,\beta \in \N^3$. There exists $C>0$ such that for every regular Vlasov field $f\colon \R_t \times\R^3_x\times\R^3_v\to \R$, the weighted spatial average $\mathcal{A}^{\alpha}_{\beta}(f)$ satisfies 
\begin{equation*}
\|\mathcal{A}^{\alpha}_{\beta}(f)\|_{L^{\infty}_v}\leq C \big\|\langle x \rangle^{4}x^{\alpha}\partial_v^{\alpha+\beta}f_{0}\big\|_{L^{\infty}_{x,v}}.
\end{equation*}
\end{proposition}

\begin{proof}
Putting the weighted initial distribution function in $L^{\infty}_{x,v}$, we obtain
\begin{equation*}
\forall v\in \R^3_v,\qquad |\mathcal{A}^{\alpha}_{\beta}(f)(v)|=\Big|\int_{\R^3_x} x^{\alpha}\partial_v^{\alpha+\beta}f_{0}(x,v)\dr x\Big|\leq \int_{\R^3_x}\frac{\dr x}{\langle x \rangle^{4}} \|\langle x \rangle^{4}x^{\alpha}\partial_v^{\alpha+\beta}f\|_{L^{\infty}_{x,v}}.
\end{equation*}
\end{proof}

We conclude this section setting the hierarchy of conservation laws for the Vlasov equation.

\begin{definition}
Let $f\colon\R_t\times \R^3_x\times\R^3_v\to \R$ be a regular solution for the Vlasov equation on $\R^3_x\times\R^3_v$. The \emph{hierarchy of conservation laws for the Vlasov field} $f$ is $$\mathcal{A}(f):=\Big\{\mathcal{A}^{\alpha}_{\beta}(f)(v)=\int_{\R^3_x} (x-vt)^{\alpha}(t\partial_x+\partial_v)^{\alpha+\beta}f(t,x,v)\dr x : \alpha,\beta \in \N^3\Big\}.$$
\end{definition}

We recall that the hierarchy of conservation laws $\mathcal{A}(f)$ for a Vlasov field $f$ is well-defined by the conservation in time of the averages $\mathcal{A}^{\alpha}_{\beta}(f)(v)$. From now on, we often write $\mathcal{A}^{\alpha}_{\beta}$ to refer to $\mathcal{A}^{\alpha}_{\beta}(f)$ without making reference to the distribution function $f$

\subsection{Late-time tails for the spatial density}

We begin writing the tails of the self-similar profile $$\mathcal{A}^{\alpha}_{\beta}\Big(\frac{x}{t}\Big)=\int_{\R^3_y} y^{\alpha}\partial_v^{\alpha+\beta}f_{0}\Big(y,\frac{x}{t}\Big)\dr y$$ in terms of the conservation laws on the zero velocity $\mathcal{A}^{\alpha}_{\beta+\gamma}(0)$.

\begin{lemma}\label{lemexpconslawprofillin}
Let $N\in\N$ and $\alpha,\beta \in \N^3$. Let $f\colon [0,\infty)\times  \R^3_x\times\R^3_v\to \R$ be a regular solution of the Vlasov equation. Then, the weighted spatial average $\mathcal{A}^{\alpha}_{\beta}$ satisfies   
$$\bigg|\mathcal{A}^{\alpha}_{\beta}\Big(\frac{x}{t}\Big)-\sum_{|\gamma|\leq N} \frac{1}{\gamma !}\mathcal{A}^{\alpha}_{\beta+\gamma}(0)\frac{x^{\gamma}}{t^{|\gamma|}}\bigg|\lesssim  \frac{|x|^{N+1}}{t^{N+1}}\sum_{|\gamma|=N+1}\sup_{(z,v)\in \R^3_x\times\R^3_v}\big| \langle z \rangle^{4}z^{\alpha}\partial_v^{\alpha+\beta+\gamma}f_{0}(z,v)\big|,$$ for all $x\in\R^3_x$, and all $t\geq 2$.
\end{lemma}

\begin{proof}
Taylor expanding the function $v\mapsto \mathcal{A}^{\alpha}_{\beta}(v)$ around $v=0$, we obtain $$\mathcal{A}^{\alpha}_{\beta}\Big(\frac{x}{t}\Big)=\sum_{|\gamma|\leq N} \frac{1}{\gamma!}\dfrac{x^{\gamma}}{t^{|\gamma|}}\partial_v^{\gamma}\mathcal{A}^{\alpha}_{\beta}(0)+\sum_{|\gamma|=N+1}R_\gamma\Big(\frac{x}{t}\Big)\frac{x^{\gamma}}{t^{|\gamma|}},$$ where the coefficients $R_\gamma(\frac{x}{t})$ of the remainder satisfy 
$$\Big|R_\gamma\Big(\frac{x}{t}\Big)\Big|\leq \frac{1}{\gamma!} \sup_{|\gamma|=N+1} \sup_{v\in\R^3_v } |\partial_v^{\gamma}\mathcal{A}^{\alpha}_{\beta}|.$$
 By $\partial_v^{\gamma}\mathcal{A}^{\alpha}_{\beta}=\mathcal{A}^{\alpha}_{\beta+\gamma}$, and the estimate in Proposition \ref{prop_def_cons_laws}, the remainder is bounded by 
\begin{align*}
\Big|\sum_{|\gamma|=N+1}R_\gamma\Big(\frac{x}{t}\Big)\frac{x^{\gamma}}{t^{|\gamma|}}\Big|&\lesssim \frac{|x|^{N+1}}{t^{N+1}} \sum_{|\gamma|=N+1}\sup_{(z,v)\in \R^3_x\times\R^3_v}\big| \langle z \rangle^{4}z^{\alpha}\partial_v^{\alpha+\beta+\gamma}f_{0}(z,v)\big|.
\end{align*}
The result follows by putting the previous estimates together, and using $\partial_v^{\gamma}\mathcal{A}^{\alpha}_{\beta}(0)=\mathcal{A}^{\alpha}_{\beta+\gamma}(0)$.
\end{proof}

Finally, we obtain the late-time tails for the spatial density in terms of the hierarchy of conservation laws on the zero velocity $\mathcal{A}(0)$.

\begin{theorem}\label{thm_latetime_linear}
Let $N\in\N$. Then, there exist universal constants $C_{\alpha}\in\R$ such that the solution $f$ to the Vlasov equation satisfies
\begin{align*}
\bigg|t^3&\int_{\R^3_v}f(t,x,v)\dr v-\sum_{|\alpha|\leq N,}\sum_{|\gamma|\leq N-|\alpha|} \frac{C_{\alpha}}{\gamma !}\mathcal{A}^{\alpha}_{\gamma}(0)\frac{x^{\gamma}}{t^{|\gamma|+|\alpha|}}\bigg|\\
&\lesssim \dfrac{1}{t^{N+1}}\bigg(\sum_{|\alpha|\leq N+1} \big\|  \langle z\rangle^{N+4}\partial_v^{\alpha}f_{0} \big\|_{L^{\infty}_{z,v}}+\sum_{|\alpha|\leq N,}\sum_{|\gamma|=N+1-|\alpha|}|x|^{N+1-|\alpha|} \big\| \langle z \rangle^{4}z^{\alpha}\partial_v^{\alpha+\gamma}f_{0} \big\|_{L^{\infty}_{z,v}}\bigg).
\end{align*}
for all $t\geq 2$ and $x \in \R^3_x$.
\end{theorem}

\begin{proof}
By Lemma \ref{lemexpconslawprofillin}, we have 
\begin{align*}
\bigg|\sum_{|\alpha|\leq N}\dfrac{C_{\alpha}}{t^{|\alpha|}}\int_{\R^3_y} y^{\alpha}&\partial^{\alpha}_{v}f_{0}\Big(y,\frac{x}{t}\Big)\dr y-\sum_{|\alpha|\leq N,}\sum_{|\gamma|\leq N-|\alpha|} \frac{C_{\alpha}}{\gamma !}\mathcal{A}^{\alpha}_{\gamma}(0)\frac{x^{\gamma}}{t^{|\gamma|+|\alpha|}}\bigg|\\
&\lesssim \frac{1}{t^{N+1}}\sum_{|\alpha|\leq N,} \, \sum_{|\gamma|=N+1-|\alpha|}C_{\alpha}|x|^{N+1-|\alpha|}\sup_{(z,v)\in\R^3_x\times\R^3_v }| \langle z \rangle^{4}z^{\alpha}\partial_v^{\alpha+\gamma}f_{0}(z,v)|.
\end{align*}
The result follows by using the self-similar expansion for the spatial density in Proposition \ref{Prolin}.
\end{proof}

Applying Theorem \ref{thm_latetime_linear} to the distribution $G^{\beta}f$ we obtain the following result.

\begin{corollary}
Let $N\in\N$. Then, there exist universal constants $C_{\alpha}\in\R$ such that the solution $f$ to the Vlasov equation satisfies
\begin{align*}
&\bigg|t^3\int_{\R^3_v}G^{\beta}f(t,x,v)\dr v-\sum_{|\alpha|\leq N,}\sum_{|\gamma|\leq N-|\alpha|} \frac{C_{\alpha}}{\gamma !}\mathcal{A}^{\alpha}_{\gamma+\beta}(0)\frac{x^{\gamma}}{t^{|\gamma|+|\alpha|}}\bigg|\\
&\quad \lesssim \dfrac{1}{t^{N+1}}\bigg(\sum_{|\alpha|\leq N+1} \big\|  \langle z \rangle^{N+4}\partial_v^{\alpha+\beta}f_{0}\big\|_{L^{\infty}_{z,v}}+\sum_{|\alpha|\leq N,} \, \sum_{|\gamma|=N+1-|\alpha|}|x|^{N+1-|\alpha|} \big\| \langle z \rangle^{4}z^{\alpha}\partial_v^{\alpha+\gamma +\beta}f_{0} \big\|_{L^{\infty}_{z,v}}\bigg).
\end{align*}
for all $t\geq 2$ and $x \in \R^3_x$.

\end{corollary}

\subsection{Weak convergence of the Vlasov field}

We next show the weak convergence of the normalised distribution function $t^3G^{\beta}f$ for solutions to the Vlasov equation in terms of the derivatives $\partial_v^{\beta}f_{0}$ of the initial data. 

\subsubsection{Concentration in the zero velocity set}

In this subsection, we show that $t^{3}f(t,x,v)$ converges weakly to the Dirac mass $(\int f_{0}(x,0)\mathrm{d}x)\delta_{v=0}(v)$. Let $\varphi\in C^{\infty}_{x,v}$ be a compactly supported test function. We first perform the change of variables $y=x-vt$ in the integral $\int g(x-vt,v)\varphi(x,v)\mathrm{d}x\mathrm{d}v$.

\begin{lemma}\label{lemma_change_variables_decay_weakconvlin}
Let $g\colon  \R^3_x\times \R^3_v\to\R$ be a regular distribution, and let $\varphi\in C^{\infty}_{x,v}$ be a compactly supported test function. Then, for all $t\in [2,\infty)$ we have $$\! \int_{\R^3_x\times \R_v^3}t^{3}g(x-vt,v)\varphi(x,v)\mathrm{d}x\mathrm{d}v=\!\int_{\R^3_x\times \R^3_y} g\Big(y,\frac{x-y}{t}\Big)\varphi\Big(x,\frac{x-y}{t}\Big)\mathrm{d}x\mathrm{d}y.$$
\end{lemma}

Next, we prove a technical lemma that will be used to show the weak convergence result.

\begin{lemma}\label{lem_asympt_exp_spatial_density_estimate_weak_conv_first_order}
Let $g\colon \R^3_x\times \R^3_v\to\R$ be a regular distribution and $\varphi\in C^{\infty}_{c}(\R^3_x \times \R^3_v)$. Then, for all $t\in [2,\infty)$ we have 
\begin{align*}
\bigg| \int_{\R_x^3\times\R_v^3}t^{3}g(x-vt,v)\varphi (x,v)\mathrm{d}x\mathrm{d}v~-&\int_{\R_x^3} \varphi \Big(x,\frac{x}{t}\Big) \int_{\R_y^3}g\Big(y,\frac{x}{t}\Big)\mathrm{d}y\mathrm{d}x\bigg|\lesssim \frac{1}{t}\sup_{(x,v)\in\R^3_x\times \R^3_v}\langle x \rangle^5(|g|+|\nabla_vg|).
\end{align*}
\end{lemma}

\begin{proof}
Applying the mean value theorem, we have 
\begin{align*}
\Big|g\Big(y,\frac{x-y}{t}\Big)&\varphi\Big(x,\frac{x-y}{t}\Big)
-g\Big(y,\frac{x}{t}\Big)\varphi\Big(x,\frac{x}{t}\Big)\Big| \lesssim \frac{1}{t}\sup_{v\in \R^3_v}|y| \big(|\nabla_v\varphi|(x,v)|g|(y,v)+|\nabla_vg|(y,v)|\varphi|(x,v) \big).
\end{align*}
Then, as $\varphi\in C^{\infty}_{c}(\R^3_x \times \R^3_v)$ and $y \mapsto \langle y \rangle^{-4}$ belongs to $L^1(\R^3_y)$, the difference $$\bigg|\int_{\R^3_x\times \R^3_y}g\Big(y,\frac{x-y}{t}\Big)\varphi\Big(x,\frac{x-y}{t}\Big)\mathrm{d}x\mathrm{d}y-\int_{\R^3_x\times \R^3_y}g\Big(y,\frac{x}{t}\Big)\varphi\Big(x,\frac{x}{t}\Big)\mathrm{d}x\mathrm{d}y\bigg|$$ satisfies the corresponding time decay estimate. Finally, we apply Lemma \ref{lemma_change_variables_decay_weakconvlin}.  
\end{proof}

We now prove the main result of this subsection.

\begin{proposition}\label{propconce_lin}
Let $\varphi\in C^{\infty}_{x,v}$ be a compactly supported test function. Then, the Vlasov field $f$ satisfies $$\lim_{t\to\infty}\int_{\R^3_x\times \R^3_v}t^3f(t,x,v)\varphi(x,v)\dr x\dr v= \int_{\R^3_x\times \R^3_v} \Big(\mathcal{A}_{0}(0)\delta_{v=0}(v)\Big)\varphi(x,v)\dr v\dr x.$$ In other words, the distribution $t^3f(t,x,v)$ converges weakly to $\mathcal{A}_{0}(0)\delta_{v=0}(v)$ as $t \to \infty$.
\end{proposition}

\begin{proof}
Applying the previous lemma to the distribution $g(x,v)=f_{0}(x,v)$, we have
\begin{align*}
\Big| \int_{\R_x^3\times\R_v^3}t^{3}f(t,x,v)\varphi (x,v)\mathrm{d}x\mathrm{d}v~-&\int_{\R_x^3} \varphi \Big(x,\frac{x}{t}\Big) \int_{\R_y^3}f_{0}\Big(y,\frac{x}{t}\Big)\mathrm{d}y\mathrm{d}x\Big| \lesssim \frac{1}{t}\sup_{(x,v)\in\R^3_x\times \R^3_v}\langle x \rangle^{5}(|f_{0}|+|\nabla_vf_{0}|).
\end{align*}

Finally, we apply Fubini and the dominated convergence theorem to obtain
\begin{align*}
\lim_{t\to\infty}\int_{\R^3_x}\varphi\Big(x,\frac{x}{t}\Big)\int_{\R^3_y}f_{0}\Big(y,\frac{x}{t}\Big)\dr y\dr x=\int_{\R^3_x}\varphi(x,0)\dr x\int_{\R^3_y}f_{0}(y,0)\dr y.
\end{align*}
\end{proof}

Applying Proposition \ref{propconce_lin} to the distribution $G^{\beta}f$ we obtain the following result.

\begin{corollary}
Let $\beta \in \N^3$. Let $\varphi\in C^{\infty}_{x,v}$ be a compactly supported test function. Then, the Vlasov field $f$ satisfies $$\lim_{t\to\infty}\int_{\R^3_x\times \R^3_v}t^3G^{\beta}f(t,x,v)\varphi(x,v)\dr x\dr v= \int_{\R^3_x\times \R^3_v}\Big(\mathcal{A}_{\beta}(0)\delta_{v=0}(v)\Big) \varphi(x,v)\dr v\dr x.$$ In other words, the distribution $t^3G^{\beta}f(t,x,v)$ converges weakly to $\mathcal{A}_{\beta}(0)\delta_{v=0}(v)$ as $t \to \infty$.
\end{corollary}

\subsubsection{Shearing of the Hamiltonian flow}

Let $\bar{v}\in \R^3_v$. In this subsection, we show that $t^{3}f(t,x+t\bar{v},v+\bar{v})$ converges weakly to the Dirac mass $(\int f_{0}(x,\bar{v})\mathrm{d}x)\delta_{v=0}(v)$.

\begin{proposition}\label{prop_shearing_linear}
Let $\varphi\in C^{\infty}_{x,v}$ be a compactly supported test function. Let $\bar{v}\in \R^3_v$. Then, the Vlasov field $f$ satisfies $$\lim_{t\to\infty}\int_{\R^3_x\times \R^3_v}t^3f(t,x,v)\varphi(x-t\bar{v},v)\dr x\dr v=\int_{\R^3_x\times \R^3_v} \Big(\mathcal{A}_{0}(\bar{v})\delta_{v=\bar{v}}(v)\Big)\varphi(x,v)\dr v\dr x.$$ In other words, the distribution $t^{3}f(t,x+t\bar{v},v+\bar{v})$ converges weakly to $\mathcal{A}_{0}(\bar{v})\delta_{v=\bar{v}}(v)$ as $t \to \infty$.
\end{proposition}

\begin{proof}
We apply Proposition \ref{propconce_lin} to $f_{\bar{v}}(t,x,v):=f(t,x+t\bar{v},v+\bar{v})=f_0(x,v+\bar{v})$.
\end{proof}

Applying Proposition \ref{prop_shearing_linear} to the distribution $G^{\beta}f$ we obtain the following result.

\begin{corollary}
Let $\beta \in \N^3$. Let $\varphi\in C^{\infty}_{x,v}$ be a compactly supported test function. Then, the Vlasov field $f$ satisfies $$\lim_{t\to\infty}\int_{\R^3_x\times \R^3_v}t^3G^{\beta}f(t,x,v)\varphi(x-t\bar{v},v)\dr x\dr v=\int_{\R^3_x\times \R^3_v} \Big(\mathcal{A}_{\beta}(\bar{v})\delta_{v=\bar{v}}(v)\Big)\varphi(x,v)\dr v\dr x.$$ In other words, the distribution $t^{3}G^{\beta}f(t,x+t\bar{v},v+\bar{v})$ converges weakly to $\mathcal{A}_{\beta}(\bar{v})\delta_{v=\bar{v}}(v)$ as $t \to  \infty$.
\end{corollary}

\section{Global existence of small data solutions}\label{section_small_data_global_existence}
In this section we prove the first part of Theorem \ref{thm_detail_modified_scattering}, the global existence of small data solutions of \eqref{eq:VP} with respect to a weighted $L_{x,v}^\infty$ norm. The estimates obtained here are the starting point for the enhanced modified scattering results in Theorem \ref{ThLatetimeFull}. These latter results are derived below in Sections \ref{SecModiscat}--\ref{SEcLatetime}.

\subsection{The bootstrap argument}
Let $N\geq 1$. We consider an initial data $f_{0}$ satisfying the smallness assumption \eqref{eq:smallnessassump} of Theorem \ref{thm_detail_modified_scattering}. By a standard local well-posedness argument, there exists a unique maximal solution $f$ to the Vlasov--Poisson system \eqref{eq:VP} arising from this initial data. Let $T_{\mathrm{max}}\in (0,\infty]$ be the maximal time such that the solution $f$ to the Vlasov--Poisson system is defined on $[0,T_{\mathrm{max}})$. By continuity, there exists a largest time $T\in (0,T_{\mathrm{max}}]$ and a constant $C_{\mathrm{boot}}>0$ such that the following bootstrap assumptions hold:
\begin{enumerate}[label = (BA\arabic*)]
\item For all $(t,x)\in [0,T)\times \R^3_x$ and any $ |\beta|\leq N-1$, we have $$\Big|\int_{\R^3_v}Z^{\beta}f(t,x,v)\mathrm{d}v\Big| \leq \dfrac{C_{\mathrm{boot}}\e}{\langle t+|x| \rangle^{3}}.$$ \label{boot2}
\item For all $(t,x)\in [0,T)\times \R^3_x$ and any $ |\beta| = N$, we have 
\begin{equation*}
\Big|\int_{\R^3_v}Z^{\beta}f(t,x,v)\mathrm{d}v\Big| \leq \frac{ C_{\mathrm{boot}}\epsilon \, \langle t \rangle^{\frac{1}{4}} }{\langle t+|x| \rangle^{3}}.
\end{equation*}
\label{boot3}
\end{enumerate}
We will improve these estimates when $\e>0$ is small enough, for $C_{\mathrm{boot}}>0$ chosen large enough. \\

\textbf{Structure of the proof of small data global existence}
\begin{enumerate}[label = (\alph*)]
\item As a first step, we prove decay estimates in time for the force field $\nabla_x\phi$ and its derivatives $\nabla_xZ^\gamma\phi$.
\item Then, we consider a suitable modification of the linear weight $\langle x-vt \rangle$ denoted by $\z$, which is preserved by the non-linear Vlasov equation. We prove that the correction $\z~-~\langle x-vt\rangle$ grows logarithmically in time.
\item We prove that for every $|\beta|\leq N$, a weighted $L^{\infty}_{x,v}$ norm of $Z^{\beta}f$ grows at most as $t^{\frac{1}{4}}$. This will suffice to improve the bootstrap assumption \ref{boot3}.
\item Finally, we control uniformly in time the weighted spatial averages of $Z^{\beta}f$ with $|\beta| \leq N-1$. These estimates will allow us to obtain optimal time decay for the velocity averages $\rho(Z^{\beta} f)$ for every $|\beta|\leq N-1$. In particular, we will improve the bootstrap assumption \ref{boot2}.
\end{enumerate}

\subsection{Pointwise decay estimates for the force field}
We start by controlling an integral term related to the Green function of the Laplacian on $\R^3_x$. 

\begin{lemma}\label{lemma_uniform_integral_bound_kernel_convolution_duan}
There holds
\begin{equation*}
    \int_{\R_y^3}\dfrac{\mathrm{d}y}{|y|^2(1+|x+y|)^3} \leq 7 \pi.
\end{equation*}
\end{lemma}
\begin{proof}
We first remark that
$$ \int_{|y| \leq 1 }\dfrac{\mathrm{d}y}{|y|^2(1+|x+y|)^3} \leq \int_{r=0}^1 \int_{\theta=0}^{ \pi} \int_{\varphi=0}^{2\pi} \frac{r^2 \, \mathrm{d} r}{r^2} \dr \theta \dr \varphi \leq 4 \pi.$$
We deal with the remaining region by applying the Cauchy-Schwarz inequality
\begin{align*}
\int_{|y| \geq 1 }\dfrac{\mathrm{d}y}{|y|^2(1+|x+y|)^3} &\leq \bigg| \int_{|y| \geq 1 }\dfrac{\mathrm{d}y}{|y|^4} \bigg|^{\frac{1}{2}} \, \bigg| \int_{|y| \geq 1 }\dfrac{\mathrm{d}y}{(1+|x+y|)^6}   \bigg|^{\frac{1}{2}} \leq 2 \sqrt{\pi}\bigg|\int_{\R^3_z} \frac{\mathrm{d} z}{(1+|z|)^6}\bigg|^{\frac{1}{2}}  \leq \frac{4\pi}{\sqrt{3}}.
\end{align*}
\end{proof}

By Lemma \ref{lemma_uniform_integral_bound_kernel_convolution_duan}, we obtain decay in time for the integral term
\begin{equation}\label{remark_uniform_integral_bound_kernel_convolution_duan}
    \int_{\R^3_y}\dfrac{\mathrm{d}y}{|y|^2(1+t+|x-y|)^3} =\dfrac{1}{t^2}\int_{\R^3_{z}} \dfrac{\mathrm{d}z}{|z|^2(t^{-1}+1+|z-\frac{x}{t}|)^3}\leq \dfrac{7 \pi}{t^2},
\end{equation}
by using the change of variables $y=t z$. We use the estimate \eqref{remark_uniform_integral_bound_kernel_convolution_duan} to obtain decay for $\nabla_x Z^\gamma\phi$. 

\begin{proposition}\label{proposition_estimate_phi}
For any $|\kappa|\leq N-1$ and all $(t,x)\in [0,T)\times \R^3_x$, we have 
$$ t^{|\kappa|} \big|\nabla_x \partial_x^\kappa \phi \big| (t,x) \lesssim \dfrac{\e}{\langle t \rangle^{2}}.$$ 
For the top order derivatives $|\kappa|=N$, there holds
$$t^{|\kappa|} \big|\nabla_x \partial_{x}^{\kappa} \phi \big| (t,x) \lesssim \dfrac{\e }{\langle t \rangle^{\frac{7}{4}}}.$$
\end{proposition}
\begin{proof}
Let $(t,x) \in [0,T_{\mathrm{max}}) \times \R^3_x$. Using $[\Delta, \partial_{x^i}]=0$ and the vector fields $G_i$, we obtain
$$\Delta_x \partial_x^\kappa \phi=\partial_x^\kappa \rho(f)= \rho \big( \partial_x^\kappa f \big)= t^{-|\kappa|}\rho \big(G^\kappa f \big) , $$ for every $|\kappa| \leq N$. We use the Green function for the Poisson equation on $\R^3$ to write the gradient of the solution to the commuted Poisson equation as 
\begin{equation}\label{nabla_phi}
\nabla_x \partial_x^{\kappa} \phi(t,x)= \frac{1}{4 \pi t^{|\kappa|}}\int_{\R_y^3} \dfrac{y}{|y|^3}\rho(G^{\kappa}f)(t,x-y)\mathrm{d}y ,
\end{equation} 
for every $|\kappa| \leq N$. Thus, for any $|\kappa| \leq N$ we obtain
 \begin{align*} t^{|\kappa|}|\nabla_x \partial_x^\kappa \phi(t,x)|& \lesssim  \int_{\R_y^3} \dfrac{\mathrm{d}y}{|y|^3\langle t+|x-y| \rangle^3} \big\| \langle t+|x| \rangle^3 \rho \big( G^\kappa f\big) (t,x) \big\|_{L^\infty(\R^3_x)} \\ 
 &\lesssim \frac{1}{\langle t \rangle^2} \big\| \langle t+|x| \rangle^3 \rho \big( G^\kappa f\big) (t,x) \big\|_{L^\infty(\R^3_x)},
 \end{align*} 
by using \eqref{remark_uniform_integral_bound_kernel_convolution_duan}. The estimate on the lower order derivatives $|\kappa| \leq N-1$ follows from the bootstrap assumption \ref{boot2}, whereas the one for the top order derivatives $|\kappa| =N$ ensue from  \ref{boot3}. 
\end{proof}

\subsection{The modified weights}

Since we expect $x-tv$ to grow as $\log(t)$ along the nonlinear flow, we will work with the following modification of this linear weight.
\begin{definition}
Let $\varphi=(\varphi^1,\varphi^2,\varphi^3): [0,T) \times \R^3_x \times \R^3_v \to \R^3$ be the unique solution to
$$\mathbf{T}_\phi (\varphi^i)=- \mathbf{T}_\phi (x^i-tv^i), \qquad \varphi(0,x , v) =0.$$
We define the \emph{modified weight function} $\mathbf{z}_{\mathrm{mod}}$ as
$$\mathbf{z}_{\mathrm{mod}}(t,x,v) := \big\langle x-tv+\varphi(t,x,v) \big\rangle.$$
\end{definition}

One important property of the weight $\mathbf{z}_{\mathrm{mod}}$ is that it is, by definition, constant along the nonlinear flow. In order to exploit this property, we need to show that $\z$ does not deviate too much from the linear weight $x-tv$. 

\begin{lemma}\label{lemma_varphi}
There holds $\T_\phi (\mathbf{z}_{\mathrm{mod}} )=0$. Moreover, the correction $\varphi$ satisfies that
$$ \forall \, (t,x,v) \in [0,T) \times \R^3_x \times \R^3_v, \qquad |\varphi|(t,x,v) \lesssim \epsilon \log \langle t \rangle , \qquad |\nabla_v\varphi|(t,x,v) \lesssim \epsilon \, \langle t \rangle .$$
\end{lemma}

\begin{proof}
The first property is straightforward, since we have defined $\varphi$ so that $\T_\phi (\mathbf{z}_{\mathrm{mod}} )=0$. Next, we have
$$  \left|\mathbf{T}_\phi (x^i-t v^i)\right|  = \big| \nabla_x \phi (t,x) \cdot \nabla_v (x^i-tv^i ) | \lesssim t | \nabla_x \phi |(t,x) \lesssim  \epsilon \, \langle t \rangle^{-1},$$ for $i\in\{1,2,3\}$. We then get $|\mathbf{T}_\phi (\varphi)| \lesssim \epsilon \, \langle t \rangle^{-1}$ on $[0,T) \times \R^3_x \times \R^3_v$, which implies the estimate for $\varphi$ according to Duhamel formula.

To conclude the proof it suffices to show, as $\nabla_v=G-t \nabla_x$, that there exists $C>0$ such that
\begin{align}
 |\nabla_x \varphi|(t,x,v) &\leq C  \epsilon , \label{eq:boot1} \\
  |G\varphi|(t,x,v)  &\leq C  \epsilon  \log\langle t \rangle , \label{eq:boot2}
 \end{align}
for all $(t,x,v) \in [0,T) \times \R^3_x \times \R^3_v$. By continuity, there exists a maximal time $0<T_{\mathrm{boot}} \leq T$ such that \eqref{eq:boot1}--\eqref{eq:boot2} holds on $[0,T_{\mathrm{boot}}) \times \R^3_x \times \R^3_v$. Let us prove by a bootstrap argument that $T_{\mathrm{boot}}=T$. Consider $Z \in \boldsymbol{\lambda}$ and apply the commutation formula in Proposition \ref{lemma_commuted_nonliner_Vlasov} to show
\begin{align*}
 \T_\phi (Z \varphi)&=[\T_\phi, Z](\varphi)+Z \T_\phi(\varphi)=-\nabla_x Z_x\phi \cdot \nabla_v \varphi+ t \nabla_x Z_x \phi+t [Z_x,\nabla_x]\phi .
 \end{align*}
 Writing again $\nabla_v=G-t \nabla_x$, we obtain
 \begin{align*}
\big| \T_\phi (Z \varphi) \big| & \lesssim t  \big|\nabla_x Z_x \phi\big| \big(1+\big|\nabla_x \varphi \big| \big)+\big|\nabla_x Z_x \phi\big|\big|G \varphi \big| , \\
 \big| \T_\phi (\partial_{x^i} \varphi) \big| & \lesssim t \big|\nabla_x \partial_{x^i}\phi\big| \big(1+\big| \nabla_x \varphi \big|+\big|G \varphi \big| \big).
 \end{align*}
We then deduce from the pointwise decay estimates of Proposition \ref{proposition_estimate_phi} as well as from the bootstrap assumptions \eqref{eq:boot1}--\eqref{eq:boot2} for the derivatives of $\varphi$ that, for all $(t,x,v) \in [0,T_{\mathrm{boot}}) \times \R^3_x \times \R^3_v$ we have
\begin{align*}
 \left| \T_\phi (Z \varphi) \right|(t,x,v) & \lesssim \epsilon  \langle t \rangle^{-1}\left( 1+C  \epsilon  + C  \epsilon \langle t \rangle^{-1} \log  \langle t \rangle  \right) \lesssim \epsilon \langle t \rangle^{-1}(1+C \epsilon) , \\
  \left| \T_\phi (\partial_{x^i} \varphi) \right|(t,x,v)& \lesssim \epsilon \langle t \rangle^{-2}\left(1+  \epsilon \log  \langle t \rangle  \right) \lesssim \epsilon \langle t \rangle^{-\frac{3}{2}} (1+C\epsilon) .
  \end{align*}
Hence, if $C$ is chosen large enough and if $\epsilon$ is small enough, we have 
$$ \left| \T_\phi (Z \varphi) \right|(t,x,v) \leq \frac{C}{2}\epsilon \, \langle t \rangle^{-1}  , \qquad \left| \T_\phi (\partial_{x^i}  \varphi) \right|(t,x,v) \leq \frac{C}{2}\epsilon \langle t \rangle^{-\frac{3}{2}} .$$
Using that $\varphi$ initially vanishes and Duhamel's formula, we improve \eqref{eq:boot1}--\eqref{eq:boot2} on $[0,T_{\mathrm{boot}})$. As a result, we have $T_{\mathrm{boot}}=T$ as well as the stated estimate for $\nabla_v \varphi$.
\end{proof}

We then introduce the \emph{modified velocity weight function}.

\begin{definition}
Let $w=(w^1,w^2,w^3)\colon [0,T) \times \R^3_x \times \R^3_v \to \R^3$ be the unique solution to
$$\mathbf{T}_\phi (w^i)=- \mathbf{T}_\phi (v^i), \qquad w(0,x , v) =0.$$
We define the \emph{modified velocity weight function} $\mathbf{v}_{\mathrm{mod}}(t,x,v) := \langle v+w(t,x,v) \rangle$.
\end{definition}

The next result implies that $v$ is almost conserved by the nonlinear flow, that is $|v-\mathbf{v}_{\mathrm{mod}}| \lesssim \epsilon$.

\begin{lemma}\label{Lemvelociweight}
We have $\T_\phi (\mathbf{v}_{\mathrm{mod}} )=0$ and $|w|(t,x,v) \lesssim \epsilon$ for all $(t,x,v) \in [0,T) \times \R^3_x \times \R^3_v$.
\end{lemma}

\begin{proof}
The first relation ensues from $\T_\phi (v+w)=0$. For the second one, we remark that for $ 1 \leq i \leq 3$ we have
\begin{equation}\label{eq:estiv}
\left| \T_{\phi} (w^i) \right|=\left| \T_{\phi} (v^i) \right|=| \partial_{x^i} \phi |(t,x) \lesssim \epsilon \langle t \rangle^{-2}
\end{equation}
according to the decay estimate for the force field in Proposition \ref{proposition_estimate_phi}. The result then follows from Duhamel formula as well as $w(0, \cdot , \cdot ) =0$.
\end{proof}

\subsection{Pointwise estimates for the distribution function and its derivatives}

We are now able to prove upper bounds for the weighted derivatives $\langle v \rangle^{M_v} \mathbf{z}_{\mathrm{mod}}^{M_z} Z^{\beta}f$. Recall the initial norm $\mathbb{E}_N^{N_x,N_v}[f_{0}]$ introduced in \eqref{eq:defininorm} and that $\mathbb{E}_N^{8,7}[f_{0}] \leq \epsilon$. For simplicity, we will consider $N\geq 2$. We sketch the proof of the case $N=1$ in Remark \ref{RqNequal1}.

\begin{proposition}\label{prop_point_bound_deriv_distrib}
Let $N_x \geq 0$, $N_v \geq 0$, $N \geq 2$, and assume that $\mathbb{E}_N^{N_x,N_v}[f_{0}]<\infty$. Then, there exists a constant $\varepsilon_N>0$ depending only on $N$, such that the following statement holds. If $\epsilon \leq \varepsilon_N$, then, for all $(t,x,v)\in [0,T)\times \R^3_x\times \R^3_v $, we have
\begin{alignat}{2}
\label{machintruc} & \color{white} \square \square \, \color{black} \mathbf{v}_{\mathrm{mod}}^{N_v} \,  \mathbf{z}_{\mathrm{mod}}^{N_x} \big| \partial_x^{\beta_x} G^{\beta_v} f \big|(t,x,v) \leq 3\mathbb{E}_N^{N_x,N_v}[f_{0}] \log^{|\beta_v|}(2+t), \qquad &&\textrm{ if } |\beta_x|+|\beta_v| \leq N-1, \\
\nonumber & \color{white} \square \square \qquad \color{black} \mathbf{v}_{\mathrm{mod}}^{N_v} \,  \mathbf{z}_{\mathrm{mod}}^{N_x} \big| Z^{\beta}f \big|(t,x,v) \leq 3\mathbb{E}_N^{N_x,N_v}[f_{0}] \langle t \rangle^{\frac{1}{4}}, \qquad && \textrm{ if } |\beta| = N.
\end{alignat}
\end{proposition}
\begin{remark}
By Lemma \ref{Lemvelociweight} and the inequality $(a+b)^z\leq 2^{z-1}(a^z+b^z)$, there exists $c>0$ such that $\langle v \rangle^{N_v} \leq 2^{\frac{N_v}{2}-1} \, \mathbf{v}_{\mathrm{mod}}^{N_v}+ 2^{\frac{N_v}{2}-1}| c \epsilon|^{N_v}$. 
\end{remark}
\begin{proof}
By continuity, there exists $0 < T_0 \leq T$ such that \eqref{machintruc} holds on $[0,T_0)$, and let us prove that we can improve these two estimates if $\epsilon$ is small enough, which would imply that they in fact hold on $[0,T)$. The starting point of the analysis consists in writing, for a fixed $|\beta| \leq N$,
\begin{align}
\T_{\phi}\left(\mathbf{v}_{\mathrm{mod}}^{N_v} \, \mathbf{z}_{\mathrm{mod}}^{N_x}  Z^{\beta}f  \right)&=\mathbf{v}_{\mathrm{mod}} ^{N_v} \, \mathbf{z}_{\mathrm{mod}}^{N_x}\T_{\phi}(Z^{\beta}f), \label{formula_transport_applied_weighted_distribution}
\end{align}
where we used $\T_\phi(\mathbf{v}_{\mathrm{mod}})=\T_\phi(\mathbf{z}_{\mathrm{mod}})=0$. By the commutation formula in Lemma \ref{lemma_commuted_nonliner_Vlasov} and $\nabla_v=G-t\nabla_x$, we have 
\begin{align} \nonumber
\big|\T_{\phi} \big(Z^{\beta}f \big) \big| & \lesssim  \sup_{|\alpha| \leq |\beta|-1,} \, \sup_{\gamma= \beta - \alpha}  |\nabla_x Z^{\gamma}_x\phi\cdot  \nabla_v Z^{\alpha}f| \\
& \leq   \sup_{|\alpha| \leq |\beta|-1,} \, \sup_{\gamma= \beta - \alpha} t|\nabla_x Z^{\gamma}_x\phi||\nabla_x Z^\alpha f |+|\nabla_x Z^{\gamma}_x\phi||G Z^\alpha f |. \label{eq:TZbeta}
 \end{align}
Fix $\gamma$ and $\alpha$ such that $|\alpha| \leq |\beta|-1$ and $\gamma+\alpha = \beta$. Then, we have
 $$Z^\beta = \partial_x^{\beta_x} G^{\beta_v}, \qquad Z^\gamma_x = t^{|\gamma_v|}\partial_x^{\gamma_x} \partial_x^{\gamma_v}, \qquad Z^\alpha = \partial_x^{\alpha_x} G^{\alpha_v},$$ where $\beta_x=\gamma_x+\alpha_x$, and $\beta_v=\gamma_v+\alpha_v$. We deal first with the case $|\beta| \leq N-1$. The goal consists in identyfing a hierarchy, indexed by $\beta_v$, in the different derivatives $Z^\beta f$. We start by controlling the easiest term in \eqref{eq:TZbeta}. According to Proposition \ref{proposition_estimate_phi}, we have
 \begin{equation}\label{eq:TZbetabis}
\big|\T_{\phi} \big(Z^{\beta}f \big) \big|  \lesssim \frac{\epsilon}{\langle t \rangle^2} \sup_{\substack{|\xi_x|+|\xi_v| \leq |\beta| \\ |\xi_v| \leq |\beta_v|+1}} |\partial_x^{\xi_x} G^{\xi_v}f|+  \sup_{|\alpha| \leq |\beta|-1,} \, \sup_{\gamma= \beta - \alpha} t|\nabla_x Z^{\gamma}_x\phi||\nabla_x Z^\alpha f |. 
 \end{equation}
 For the first terms on the RHS, even if the number of Galilean vector fields can be larger than $|\beta_v|$, the decay rate $t^{-2}$ is strong enough in order to handle them. For the most problematic terms, we will crucially exploit the condition $\gamma+\alpha = \beta$. Then,
 \begin{itemize}
 \item either $|\gamma_x| \geq 1$, in which case $|\nabla_x Z^\gamma_x \phi|(t,x) \lesssim \epsilon \, \langle t \rangle^{-3}$ by Proposition \ref{proposition_estimate_phi},
 \item or $|\alpha_v| \leq |\beta_v|-1$. Indeed, as $\gamma+\alpha=\beta$ and $|\alpha| \leq |\beta|-1$, we have $|\gamma| \geq 1$. So $|\gamma_x|=0$ implies that $|\gamma_v| \geq 1$ and it remains to use $\gamma_v+\alpha_v=\beta_v$.
 \end{itemize}
 Consequently, there exists a constant $C(N)>0$, depending only on $N$, such that
  \begin{equation*}
\big|\T_{\phi} \big(Z^{\beta}f \big) \big|  \leq \frac{\epsilon C(N)}{\langle t \rangle^2} \sup_{\substack{|\xi_x|+|\xi_v| \leq |\beta| \\ |\xi_v| \leq |\beta_v|+1}} \big|\partial_x^{\xi_x} G^{\xi_v}f \big|+ \frac{\epsilon C(N)}{\langle t \rangle} \sup_{\substack{|\alpha_x|+|\alpha_v| \leq |\beta| \\ |\alpha_v| \leq |\beta_v|-1 }} \big|\partial_x^{\alpha_x} G^{\alpha_v}f \big|. 
 \end{equation*}
 We then deduce from \eqref{formula_transport_applied_weighted_distribution} and the bootstrap assumption \eqref{machintruc} that there exists $\overline{C}(N)>0$ such that, for all $(t,x,v) \in [0,T_0)\times \R^3_x \times \R^3_v$ we have
 $$ \Big| \T_{\phi}\left(\mathbf{v}_{\mathrm{mod}} ^{N_v} \, \mathbf{z}_{\mathrm{mod}}^{N_x}  Z^{\beta}f  \right) \Big| (t,x,v) \leq  \epsilon \overline{C}(N) \mathbb{E}_N^{N_x,N_v}[f_{0}] \bigg(\frac{\log^{|\beta_v|+1}(2+t)}{\langle t \rangle^2} +\frac{\log^{|\beta_v|-1}(2+t)}{\langle t \rangle} \bigg).$$
 We then deduce from Duhamel's principle that, for all $(t,x,v) \in [0,T_0)\times \R^3_x \times \R^3_v$
 $$  \mathbf{v}_{\mathrm{mod}}^{N_v} \, \mathbf{z}_{\mathrm{mod}}^{N_x} \big| Z^\beta f \big|(t,x,v) \leq \mathbb{E}_N^{N_x,N_v}[f_{0}]+\epsilon \widetilde{C}(N) \mathbb{E}_N^{N_x,N_v}[f_{0}] \log^{|\beta_v|}(2+t),$$
 where $\widetilde{C}(N) >0$ depends only on $N$. If $\epsilon$ is small enough, this improves the bootstrap assumption \eqref{machintruc}. 
 
 We now focus on the top order derivatives $|\beta|=N$. The difference here is that the case $|\gamma|=N$ is allowed and that the top order derivatives of the force field decay at a slightly slower decay rate, 
 $$ | \nabla_x Z^\gamma_x \phi| (t,x) = t^{|\gamma_v|}|\nabla_x\partial_x^{\gamma_x} \partial_x^{\gamma_v}\phi|(t,x) \lesssim t^{-\frac{7}{4}-|\gamma_x|}.$$ Note however that $|\gamma|=N$ implies $|\alpha| =0$ in \eqref{eq:TZbeta}, and $\nabla_x f$ is uniformly bounded. We then get according to \eqref{machintruc} and Proposition \ref{proposition_estimate_phi} that
\begin{align*}
  \sup_{|\beta| = N} \big|\T_{\phi}\big( \mathbf{v}_{\mathrm{mod}}^{N_v} \, \mathbf{z}_{\mathrm{mod}}^{N_x}  Z^{\beta}f  \big)\big|(t,x,v)& \lesssim  \frac{\epsilon \log^{N-1}(2+t)}{\langle t \rangle^{\frac{7}{4}}}\mathbb{E}_N^{N_x,N_v}[f_{0}]+\frac{\epsilon}{\langle t \rangle^{\frac{3}{4}}}\mathbb{E}_N^{N_x,N_v}[f_{0}] \\
  & \quad +\frac{\epsilon}{ \langle t \rangle} \sup_{|\kappa| =N}  \mathbf{v}_{\mathrm{mod}}^{N_v} \, \mathbf{z}_{\mathrm{mod}}^{N_x} |Z^\kappa f|(t,x,v).
  \end{align*}
The estimate for the top order derivatives then ensues from Duhamel's principle and the Gronwall lemma for $\epsilon$ small enough.
\end{proof}
\begin{remark}\label{RqNequal1}
The case $N=1$ could be handled by slightly modifying the previous proof. We would have to prove first that $|\nabla_x f |$ remains uniformly bounded using $|\T_\phi (\nabla_x f)|(t,x,v) \lesssim \langle t \rangle^{-3/2}$. Then, $Gf$ can be controlled as we bounded $Z^\beta f$ for $|\beta|=N$.
\end{remark} 

Recall that $g_0(t,x,v):=f(t,x+tv,v)$, so that $\partial_x^{\kappa_x} \partial_v^{\kappa_v} g_0(t,x,v)=[\partial_x^{\kappa_x} G^{\kappa_v} f](t,x+tv,v)$. We then obtain the next result by applying Proposition \ref{prop_point_bound_deriv_distrib} as well as Lemma \ref{lemma_varphi}. We also use $\mathbf{z}_{\mathrm{mod}}(t,x+tv,v)-\langle x \rangle \leq |\varphi|(t,x,v) \lesssim \epsilon \log \langle t \rangle$.

\begin{corollary}\label{Corboundg0}
For any $|\kappa_x|+|\kappa_v| \leq N$ and all $(t,x,v)\in [0,T)\times \R^3_x\times \R^3_v$, we have 
$$
  \langle x \rangle^8 \langle v \rangle^7 \big| \partial_{x}^{\kappa_x}  \partial_v^{\kappa_v} g_0 \big|(t,x,v) \lesssim \left\{ 
	\begin{array}{ll}
        \epsilon \log^{|\kappa_v|+8}(2+t) & \mbox{if $|\kappa_x|+|\kappa_v| \leq N-1$, } \\
       \epsilon \, \langle t \rangle^{\frac{1}{4}} \log^{8}(2+t) & \mbox{if $|\kappa_x|+|\kappa_v|=N$}.
    \end{array} 
\right. .$$
\end{corollary}

\subsection{Uniform boundedness of spatial averages}
We proceed to show uniform boundedness in time of the spatial averages 
\begin{equation}\label{introduce_stable_averages_for_decay}
\int _{\R^3_x} G^{\beta}f(t,x,v)\mathrm{d}x.
\end{equation} 

\begin{remark}
We restrict the analysis to the derivatives $G^\beta$ since the spatial average of $\partial_x^\kappa G^\beta f$ vanishes as soon as $|\kappa| \geq 1$. This follows from integration by parts, the spatial decay properties of $f$, and the bounds for its derivatives given in Proposition \ref{prop_point_bound_deriv_distrib}.
\end{remark}

\begin{proposition}\label{prop_derivative_time_x_average}
Let $|\beta|\leq N-1$. Then, for all $(t,v)\in [0,T)\times \R^3_v$, we have 
$$ \bigg| \partial_t\int_{\R^3_x}  Z^{\beta}f(t,x,v) \mathrm{d}x \bigg|\lesssim \frac{\e}{ \langle t \rangle^{\frac{7}{4}}}\sup_{|\kappa|\leq |\beta|+1}\sup_{x\in\R^3_x }\big|\mathbf{z}^{4}_{\mathrm{mod}}Z^{\kappa}f(t,x,v)\big|.$$
If $|\beta| \leq N-2$, then the factor $\langle t \rangle^{-\frac{7}{4}}$ can be upgraded to $\langle t \rangle^{-2}$.
\end{proposition}

\begin{proof}
Fix $(t,v)\in  [0,T)\times \R^3_v$ and $|\beta|\leq N-1$. Integrating the commutation formula of Proposition \ref{lemma_commuted_nonliner_Vlasov} for $Z^{\beta}f$ and performing integration by parts in $x$, we have,
\begin{align*}
 \partial_t\int _{\R^3_x}  Z^{\beta}f \mathrm{d}x &=\int _{\R^3_x} \big(  \partial_t+v \cdot \nabla_x \big)  Z^{\beta}f(t,x,v) \mathrm{d}x\\
&= \mu \int_{\R^3_x} \nabla_x \phi \cdot \nabla_v Z^\beta f \dr x+\mu \sum_{|\alpha| \leq |\beta|-1} \, \sum_{\gamma = \beta-\alpha} C^{\beta}_{\alpha \gamma} \int_{\R^3_x} \nabla_x Z^\gamma_x \phi \cdot \nabla_v Z^\alpha f \dr x.
\end{align*} 
Using $\nabla_v=G-t\nabla_x$ and performing again integration by parts in $x$, we get
$$
\bigg| \partial_t\int _{\R^3_x}  Z^{\beta}f  \mathrm{d}x \bigg| \lesssim \sum_{|\gamma|+|\kappa| \leq |\beta|} \int _{\R^3_x}  \big| \nabla_x Z^{\gamma}_x\phi \big| \big| GZ^\kappa f \big|\mathrm{d}x+t\int _{\R^3_x}  \big| \nabla_x \nabla_x Z^{\gamma}_x\phi \big| \big| Z^\kappa f \big|\mathrm{d}x .
$$

We next use the time decay of the force field in Proposition \ref{proposition_estimate_phi} to obtain that for all $ (t,v) \in [0,T) \times \R^3_v$, we have
$$ 
 \bigg| \partial_t\int_{\R^3_x}  Z^{\beta}f(t,x,v) \mathrm{d}x \bigg|(t,v) \lesssim \frac{\epsilon}{ \langle t \rangle^{\frac{7}{4}}} \sup_{|\kappa|\leq |\beta|+1,}\sup_{x\in\R^3_x }  \big| \mathbf{z}_{\mathrm{mod}}^4 Z^{\kappa}f(t,x,v)\big|  \int_{\R^3_y} \frac{\dr y}{\mathbf{z}_{\mathrm{mod}}^4(t,y,v)}.
$$ 
It remains us to prove that
\begin{equation}\label{keva:zmodint}
\int_{\R^3_y} \frac{\dr y}{\mathbf{z}_{\mathrm{mod}}^4(t,y,v)} \lesssim \int_{\R^3_z} \frac{\dr z}{\langle z \rangle^4} \leq 4.
\end{equation}
For this, we recall the estimate $|\nabla_x \varphi| \lesssim \epsilon$, that allows us to perform the change of variables $z=y-tv+\varphi(t,y,v)$ for $\epsilon$ small enough. The better estimate in the case $|\beta| \leq N-2$ is related to the better decay properties verified by the lower order derivatives of the force field $\nabla_x Z^\xi \phi$ for $|\xi| \leq N-1$. 
\end{proof}

Combining Proposition \ref{prop_derivative_time_x_average} with the pointwise estimates in Proposition \ref{prop_point_bound_deriv_distrib}, we obtain the following.

\begin{corollary}\label{cor_bound_small_integral_x}
For all $(t,v)\in [0,T)\times \R_v^3$ and every $|\beta|\leq N-1$, we have
$$ \langle v \rangle^{3}\bigg| \int_{\R^3_x}Z^{\beta}f(t,x,v)\mathrm{d}x \bigg|\lesssim \epsilon.$$
\end{corollary}

\subsection{Pointwise decay estimates for velocity averages}

In this subsection, we show that the decay rate of $\rho(Z^{\beta}f)$ for $|\beta| \leq N -1$, coincides with the one of the linearised system. In particular, we improve the bootstrap assumption \ref{boot2}. The starting point of the analysis consists in performing the change of variables $y=x-tv$.
\begin{proposition}\label{Prodecayvelocitav}
For any $|\beta| \leq N-1$, we have that for all $(t,x) \in [2,T) \times \R^3_x$,
$$  \bigg| t^3 \int_{\R^3_v} Z^\beta f(t,x,v) \dr v- \int_{\R^3_y} Z^\beta f \Big(t,y,\frac{x}{t} \Big) \dr y \bigg| \lesssim   \left\{ 
	\begin{array}{ll}
        \epsilon \, \langle x/ t \rangle^{-3} \, t^{-1} \log^{N}(t) & \mbox{if $|\beta| \leq N-2$, } \\[2pt]
        \epsilon \, \langle x/ t \rangle^{-3} \, t^{-2/3} & \mbox{if $|\beta|=N-1$}.
    \end{array} 
\right.
$$
\end{proposition}
\begin{proof}
Let $|\beta| \leq N-1$ and recall that $\partial_{x,v}^\beta g_0(t,x,v)= Z^\beta f(t,x+tv,v)$. Fix $(t,x) \in [0,T) \times \R^3_x$ and perform the change of variables $y=x-tv$ to get
\begin{align*}
 t^3 \! \int_{\R^3_v} Z^\beta f(t,x,v) \dr v & = \int_{\R^3_v} \big[\partial_{x,v}^\beta g_0 \big](t,x-tv,v) \dr v = \int_{\R^3_y} \big[\partial_{x,v}^\beta g_0 \big] \Big( t,y, \frac{x-y}{t} \Big) \dr y \\
 & =  \int_{\R^3_y} \big[\partial_{x,v}^\beta g_0 \big] \Big( t,y, \frac{x}{t} \Big) \dr y+\int_{\tau=0}^1 \int_{\R^3_y} \frac{y}{t} \cdot \big[ \nabla_v \partial_{x,v}^\beta g_0 \big] \Big( t,y, \frac{x-\tau y}{t} \Big) \dr y \dr \tau.
 \end{align*}
We remark that the spatial average of $\partial_{x,v}^\beta g_0 $ is equal to the one of $Z^\beta f$. It only remains to deal with the last term on the RHS of the previous equality. For this, note that
\begin{equation*}
\forall \, \tau \in [0,1 ], \qquad 3\Big\langle \frac{x-\tau y}{t} \Big\rangle+ \big\langle y+\varphi(t,y+tv,v) \big\rangle = \Big\langle \frac{x-\tau y}{t} \Big\rangle+ \mathbf{z}_{\mathrm{mod}}(t,y+tv,v)  \geq  \Big\langle \frac{x}{t} \Big\rangle .
\end{equation*}
Indeed, either  $|y+\varphi(t,y+tv,v)| \geq \frac{1}{2}|x|$, in which case there is nothing to prove. Otherwise, it suffices to use that
$$2|x-\tau y|-|x| \geq -2\tau |\varphi(t,y+tv,v)| \gtrsim -\epsilon \log(1+t) \gtrsim -\epsilon t.$$
As $|y| \leq \mathbf{z}_{\mathrm{mod}}(t,y+tv,v)+\epsilon \log(t)$ and $y \mapsto \mathbf{z}_{\mathrm{mod}}^{-4}(t,y+tv,v) \in L^1(\R^3_y)$ by \eqref{keva:zmodint}, we then deduce
$$ \bigg| \int_{\tau=0}^1 \int_{\R^3_y} \frac{y}{t} \cdot \nabla_v \partial_{x,v}^\beta g_0 \Big( t,y, \frac{x-\tau y}{t} \Big) \dr y  \dr \tau \bigg| \lesssim \frac{\log(t)}{t\langle x/ t \rangle^3} \! \sup_{(z,v) \in \R^3\times\R^3} \!\langle v \rangle^3  \big|\mathbf{z}_{\mathrm{mod}}^8 G Z^\beta f \big|(t,z,v).$$
We finally bound the RHS by applying Proposition \ref{prop_point_bound_deriv_distrib} for $(N_x,N_v)=(8,3)$.
\end{proof}

Proposition \ref{Prodecayvelocitav} allows us to deduce the following pointwise decay estimates.

\begin{corollary}\label{Cordecayaveragev}
For any $|\beta| \leq N-1$, we have
$$ \forall \, (t,x) \in [0,T) \times \R^3_x, \qquad  \bigg|\int_{\R^3_v} Z^\beta f(t,x,v) \dr v \bigg| \lesssim \frac{\epsilon}{\langle t+|x| \rangle^3}.$$
\end{corollary}

\begin{proof}
Let $(t,x) \in [0,T) \times \R^3_x$ and remark that if $t \leq 2$, one only has to use 
$$\langle x \rangle^3 \, \langle v \rangle^4 \big |Z^\beta f \big|(t,x,v)\lesssim \mathbf{z}_{\mathrm{mod}}^3(t,x,v) \, \langle v \rangle^4 \big|Z^\beta f \big|(t,x,v)\lesssim \epsilon  .$$ Assume now that $t \geq 2$. According to Corollary \ref{cor_bound_small_integral_x}, we have
$$ \bigg| \int_{\R^3_x}Z^{\beta}f \Big( t,x, \frac{x}{t} \Big) \mathrm{d}x \bigg|\lesssim \e  \langle x/t \rangle^{-3}.$$
The result ensues from Proposition \ref{Prodecayvelocitav} and $2t \langle x/t \rangle \geq t+|x|$.
\end{proof}

We now improve the bootstrap assumption \ref{boot3}. For this, we will use the following lemma.

\begin{lemma}\label{Lemdecay1}
Let $h:[0,T) \times \R^3_x \times \R^3_v \to \R$ be a regular function. Then, for all $(t,x) \in [0,T) \times \R^3_x$, we have
$$  \int_{\R^3_v} |h|(t,x,v) \mathrm{d} v \lesssim \frac{1}{\langle t+|x| \rangle^3}\sup_{(y,v) \in \R^3_y \times \R^3_v} \langle v \rangle^7\big|  \mathbf{z}_{\mathrm{mod}}^7 \, h \big|(t,y,v).$$
\end{lemma}
\begin{proof}
The case $t \leq 1$ and $|x| \leq 1$ is straightforward since the map $v \mapsto \langle v \rangle^{-4}$ belongs to $L^1(\R^3_v)$. Let $t \geq 1$. We observe that
$$ \int_{\R^3_v} |h(t,x,v)| \mathrm{d} v \leq \sup_{(y,v) \in \R^3_y \times \R^3_v} \big| \mathbf{z}_{\mathrm{mod}}^4 \, h \big|(t,y,v) \int_{ \R^3_v} \frac{\mathrm{d} v}{\langle x-tv+\varphi(t,x,v) \rangle^4} .$$
The change of variables $y= x-tv$ yields 
$$
t^3\int_{\R^3_v} |h(t,x,v)| \mathrm{d} v \lesssim \sup_{(y,v) \in \R^3_y \times \R^3_v} \big| \mathbf{z}_{\mathrm{mod}}^4 \,  h \big|(t,y,v) \int_{\R^3_y} \frac{\mathrm{d} y}{\langle y+\varphi (t,x, (x-y)t^{-1} ) \rangle^4}.
$$
Let $\phi : y \mapsto y+\varphi (t,x, (x-y)t^{-1} )$ and $\psi = \mathrm{Id}-\phi$. In order to perform the change of variables $w= \phi(y)$, we will prove that $\frac{1}{2 } \leq |\det( \mathrm{d} \phi(y))| \leq 2$ for all $y \in \R^3$, provided $\epsilon$ is small enough. This property is satisfied since 
$$ \forall \, (t,x,y) \in [0,T) \times \R^3_x\times\R^3_y , \qquad |\mathrm{d} \psi (y) | \leq t^{-1} |\nabla_v \varphi | \big(t,x, (x-y)t^{-1} \big) \lesssim \epsilon,$$
according to Lemma \ref{lemma_varphi}. We then deduce $$ \int_{\R^3_y} \frac{\mathrm{d} y}{\langle y+\varphi (t,x, (x-y)t^{-1} ) \rangle^{4}} \leq  2\int_{\R^3_w} \frac{\mathrm{d} w}{\langle w \rangle^{4}} \leq 4.$$ 
We have then proved that 
\begin{equation}\label{eq:htimedecay}
\int_{\R^3_v} |h|(t,x,v) \mathrm{d} v \lesssim  \langle t \rangle^{-3}\sup_{(y,v) \in \R^3_y \times \R^3_v} \big|  \mathbf{z}_{\mathrm{mod}}^4 \, h \big|(t,y,v)+\langle v \rangle^4 |h|(t,y,v) ,
\end{equation}
for all $(t,x) \in [0,T) \times \R^3_x$. To obtain the spatial decay and conclude the proof, we remark that
$$
 |x| \leq  |x-tv|+t|v| \leq 2 \log(3+t) \mathbf{z}_{\mathrm{mod}}(t,x,v)+t \langle v \rangle \leq 2\langle t \rangle \big( \mathbf{z}_{\mathrm{mod}}(t,x,v)+ \langle v \rangle  \big)
$$
in view of Lemma \ref{lemma_varphi}. We finally apply \eqref{eq:htimedecay} to $\langle v\rangle^3 h$ as well as to $\mathbf{z}_{\mathrm{mod}}^3h$.
\end{proof}
\begin{remark}\label{rklineardecay}
Note that a simpler proof provides 
$$ \forall \, (t,x) \in [0,T) \times \R^3_x, \qquad  \int_{\R^3_v} |h|(t,x,v) \mathrm{d} v \lesssim \frac{1}{\langle t+|x| \rangle^3}\sup_{(y,v) \in \R^3_y \times \R^3_v} \langle y-tv \rangle^7 \langle v \rangle^7 |h|(t,y,v).$$
\end{remark}
Applying Lemma \ref{Lemdecay1} to $h(t,x,v)=  Z^{\beta}f(t,x,v)$ and then Proposition \ref{prop_point_bound_deriv_distrib}, we obtain the following estimates.

\begin{corollary}\label{cor_decay_losing_spatial_density}
For any $|\beta|\leq N$ and all $(t,x)\in [0,T)\times \R^3_x$, we have 
$$ \int _{\R^3_v} \big| Z^{\beta}f \big|(t,x,v)\mathrm{d}v\lesssim \dfrac{\e \, \langle t \rangle^{\frac{1}{4}}}{\langle t+|x| \rangle^{3}} .$$
\end{corollary}

If $C_{\mathrm{boot}}$ is chosen large enough, Corollary \ref{cor_decay_losing_spatial_density} improves the bootstrap assumption \ref{boot3} and concludes the proof of the global existence part of Theorem \ref{thm_detail_modified_scattering}.

\section{Modified scattering}\label{SecModiscat}

We follow here the strategy outlined in Section \ref{Subsecstrategymodifiedscatt}. We consider a global solution $f$ to \eqref{eq:VP} arising from data satisfying the assumptions of Theorem \ref{thm_detail_modified_scattering}. Although all the results of this section could be adapted to the derivatives of order $N-1$, we restrict our analysis to the lower order derivatives for convenience. We assume here that $N \geq 2$.  

\subsection{Convergence of the spatial averages}

Recall $g_0(t,x,v):=f(t,x+tv,v)$. The first step consists in proving the next result.

\begin{proposition}\label{ProconvQ}
There exists $Q_\infty  \in C^{N-2}(\R^3_v)$ such that, for any $|\beta| \leq N-2$, we have
$$ \forall \, (t,v) \in [2,\infty) \times \R^3_v,  \qquad \langle v \rangle^{4} \bigg| \int_{\R^3_x} \partial_v^\beta g_0 (t,x,v) \dr x - \partial_v^\beta Q_\infty (v) \bigg| \lesssim  
        \epsilon \langle t \rangle^{-1}\log^{N-1}(t) .$$
\end{proposition}

\begin{proof}
As $[G^\beta f](t,x+tv,v)=\partial_v^\beta g_0(t,x,v)$, the decay estimate in Proposition \ref{prop_derivative_time_x_average} implies
$$ \langle v \rangle^{4} \bigg| \partial_t\int_{\R^3_x}  \partial_v^\beta g_0 (t,x,v) \mathrm{d}x \bigg|\lesssim \epsilon \langle t \rangle^{-2}\sup_{|\kappa|\leq |\beta|+1}\sup_{x\in\R^3_x\times \R^3_v } \, \langle v \rangle^{4}\big|\mathbf{z}^{4}_{\mathrm{mod}}Z^{\kappa}f(t,x,v)\big| \lesssim \epsilon  \langle t \rangle^{-2} \log^{N-1}(t),$$
where we used Proposition \ref{prop_point_bound_deriv_distrib} in the last step. This implies the result.
\end{proof}

\subsection{Leading contribution of the spatial density}

Next, we prove that $t^3 \rho (G^\beta f)$ has an asymptotic self-similar profile for any $|\beta| \leq N-2$. For this, we first apply Propositions \ref{Prodecayvelocitav} and \ref{ProconvQ}, and then remark that $t^3 \, \langle x/t \rangle^3 \gtrsim \langle t+|x| \rangle^3$ for $t \geq 2$.

\begin{proposition}\label{Prorhosimilar}
Let $|\beta| \leq N-2$. Then, for all $(t,x) \in [2,\infty) \times \R^3_x$, we have
$$ \Big| t^3 \rho \big( G^\beta f \big) (t,x) - \partial_v^\beta Q_\infty \Big( \frac{x}{t} \Big) \Big| \lesssim  \frac{\epsilon \, t^2 \log^{N}(t)}{\langle t+|x| \rangle^3} .$$
\end{proposition}
\begin{remark}
Proposition \ref{Prorhosimilar} identifies the leading order contribution of the spatial derivatives of $\rho(f)$ since $t^{3+|\beta|} \partial_x^\beta \rho (f) =t^3\rho (G^\beta f)$.
\end{remark}

\subsection{Asymptotic behaviour of the force field}\label{subsecasym_force}

We are then able to derive the asymptotic self-similar behaviour of the force field and its derivatives. For this, we exploit the precise asymptotic behaviour of $\rho (f)$. 

\begin{proposition}\label{Proasympselfsiphi}
Let $\phi_{\infty}\in C^{N-2}(\R^3_v)$ be defined as
$$  \Delta_v \phi_{\infty} = Q_\infty   .$$
Then, for any $|\kappa| \leq N-2$, we have
$$ \forall \, (t,x) \in [2,\infty) \times \R^3_x, \qquad \Big| t^{2+|\kappa|} \nabla_x  \partial_x^\kappa \phi (t,x)- \nabla_v \partial_v^\kappa \phi_{\infty} \Big( \frac{x}{t} \Big) \Big| \lesssim  \epsilon \, \frac{\log^{N}(t)}{t}.$$
Moreover, we have $\|  \nabla_v \phi_{\infty}  \|_{W^{N-2,\infty}(\R^3_v)} \lesssim \epsilon$.
\end{proposition}

\begin{proof}
Let $|\kappa| \leq N-2$, and set $\psi^\kappa (t,x):= \partial_v^\kappa \phi_{\infty} (\frac{x}{t})$. Then, $t^2\Delta_x \psi^\kappa = \partial_v^\kappa Q_\infty (\frac{x}{t})$, so 
$$ t^2 \Delta_x \big[ t^{1+|\kappa|} \partial_x^\kappa \phi -  \psi^\kappa \big](t,x) = t^3 \rho \big( G^\kappa f \big)(t,x)-\partial_v^\kappa Q_\infty \Big( \frac{x}{t} \Big).$$
According to Proposition \ref{Prorhosimilar}, this quantity is bounded by $\epsilon \, t^{2} \log^{N}(t) \langle t+|x| \rangle^{-3}$. It remains to use the estimate \eqref{remark_uniform_integral_bound_kernel_convolution_duan}. 
\end{proof}

We are then able to describe the asymptotic behaviour of $t^2 \nabla_x \phi $ along the spatial characteristics of the linearised system $t \mapsto x+tv$.

\begin{corollary}\label{Coralonglin}
Let $|\kappa| \leq N-2$. Then, for all $(t,x,v) \in [2,\infty) \times \R^3_x \times \R^3_v$, we have
$$  \big| t^{2+|\kappa|} \nabla_x \partial_x^\kappa \phi (t,x+tv) - \nabla_v \partial_v^{\kappa} \phi_\infty (v) \big| \lesssim \epsilon \log^{N}(t) t^{-1}+\epsilon|x| t^{-1} .$$
\end{corollary}
\begin{proof}
By the mean value theorem and the estimates in Proposition \ref{proposition_estimate_phi}, we have
$$ t^{2+|\kappa|}\big|  \nabla_x \partial_x^\kappa \phi (t,x+tv)-\nabla_x \partial_x^\kappa \phi (t,tv) \big| \lesssim  \epsilon |x|  t^{-1} .$$
It remains to apply Proposition \ref{Proasympselfsiphi}. 
\end{proof}

\subsection{Modified scattering for the distribution function}\label{subsec_mod_scatt}

Now that we have identified the term responsible for the long-range effect of the force field, that is, the term preventing linear scattering to hold, we define the spatial modification of the characteristics as
$$ \X_1(t,x,v) := x+ \mu \log(t) \nabla_v \phi_\infty(v).$$
The modification $\X_1$ is a good approximation of the spatial nonlinear characteristics, as it can be observed in Lemma \ref{lemmaesti_f} below. Then, as $\T_\phi (f)=0$, we have
$$ \partial_t \big[f(t,\X_1+tv,v) \big] = \partial_t \big[ \X_1 \big] \cdot \nabla_x f(t,\X_1+tv,v)+\mu \nabla_x \phi (t,\X_1+tv) \cdot \nabla_v f(t,\X_1+tv,v),$$
where we have dropped the dependence in $(t,x,v)$ of $\X_1$. Writing $\nabla_v=G-t\nabla_x$ yields
\begin{align*}
 \partial_t \big[f(t,\X_1+tv,v) \big] & = \frac{\mu}{t}\Big( \! \nabla_v \phi_\infty(v)-t^2 \nabla_x \phi (t,\X_1+tv) \! \Big) \cdot \big[\nabla_x f \big](t,\X_1+tv,v)
 \\ & \quad +\mu \nabla_x \phi (t,\X_1+tv) \cdot \big[G f \big](t,\X_1+tv,v).
 \end{align*}
Then, we need to determine the asymptotic behaviour of the force field along the modified spatial characteristics. Since the corrections are lower order, it sufficies to apply Corollary \ref{Coralonglin} for $x=\X_1$. As $|\nabla_v \phi_\infty(v)| \lesssim \epsilon$, we obtain $|\X_1| \leq |x|+\log(t)$, so that
\begin{equation}\label{eq:toremind}
 \big| \partial_t \big[f(t,\X_1+tv,v) \big] \big| \lesssim \epsilon \, t^{-2}\big(\log^{N}(t)+\langle x \rangle\big)\big|\nabla_x f \big|(t,\X_1+tv,v)+\epsilon t^{-2} \big|G f \big|(t,\X_1+tv,v) .
 \end{equation}
We remark now that according to Corollary \ref{Corboundg0}, for any $Z \in \boldsymbol{\lambda}$ we have
\begin{equation}\label{eq:auxi}
 \langle x \rangle^{8} \, \langle v \rangle^{7}  \big|Z f \big| (t,\X_1+tv,v) \lesssim \big(\langle \X_1 \rangle^{8}+\log^{8}(t) \big) \, \langle v \rangle^{7}  \big|\nabla_{x,v} g_0 \big| (t,\X_1,v) \lesssim \epsilon \log^9(t).
 \end{equation}
Thus, we have $$\langle x \rangle^7 \, \langle v \rangle^{7}|\partial_t[f(t,\X_1+tv,v)] | \lesssim \epsilon^2 t^{-2}\log^{N+9}(t),$$ from which we deduce that $f$ enjoys a modified scattering dynamic.
\begin{proposition}\label{prop_mod_scattering_statem_high_reg_proof}
There exists a distribution function $\widetilde{f}_{\infty}\in C^{N-2}(\R^3_x\times\R^3_v)$ such that for every $|\kappa|\leq N-2$ and all $(t,x,v) \in [2,\infty ) \times \R^3_x \times \R^3_v$, we have $$ \langle x \rangle^7  \langle v \rangle^{7}\Big| \partial_{x,v}^{\kappa}\big[ f(t,\X_1(t,x,v)+tv,v) \big]-\partial_{x,v}^{\kappa} \widetilde{f}_{\infty}(x,v)\Big|\lesssim \e t^{-1}\log^{2N+8}(t).$$ 
\end{proposition}
\begin{remark}
By a slightly more precise analysis, one could obtain from \eqref{eq:toremind} and Corollary \ref{Corboundg0} that $| f(t,\X_1(t,x,v)+tv,v) -\widetilde{f}_{\infty}|\lesssim \e t^{-1}\log^{2}(t)$ as stated in Theorem \ref{thm_main_result_first_version}.  
\end{remark}

So far, we have only proved the convergence result in a weighted $L^\infty_{x,v}$ norm. For the purpose of deriving late-time asymptotics, we will require a similar result but for different modified characteristics. In order to avoid repetitions, we postpone the proof of the high order regularity statement, which can be obtained, either by following the proof of Proposition \ref{prop_mod_scatteringg_1} and by formally replacing $\V_1$ by $v$, or from \eqref{eq:partialtg1} and $\|\nabla_v  \phi_\infty \|_{W^{N-2,\infty}_v} \lesssim \epsilon$.

\section{Improved expansion for the spatial density}\label{SecRhoOrder2}

The purpose of this section consists in deriving a first order expansion in powers of $t^{-1}$ for the normalised spatial density $t^3\rho(f)$. We will also prove useful preparatory results for the proof of Theorem \ref{ThLatetimeFull}, our main result.

In view of the discussion carried out in Section \ref{Subsecsecondorde}, we introduce the modification
$$ \mathbf{V}_1(t,z,v) := v+\mu t^{-1} \nabla_v \phi_\infty (v)$$ 
and we set $g_1(t,z,v):= f(t,\X_1+t\V_1,\V_1)$, that is
$$g_1 (t,z,v) := f \big(t,z+tv+\mu \log(t) \nabla_v \phi_\infty(v)+\mu  \nabla_v \phi_\infty(v),v+\mu t^{-1} \nabla_v \phi_\infty(v) \big).$$
We first prove that $g_1$ converges to $$f_\infty(x,v)=\widetilde{f}_\infty(x+\mu \nabla_v \phi_\infty (v),v)$$ in a strong topology. In other words, we prove a modified scattering statement for $f$. Then, we prove an enhanced convergence estimate for the spatial average of $g_1$ to the one of $f_\infty$. This will allow us to express $\rho(f)$ in terms of $g_1$, and then obtain a first order expansion for $t^3 \rho (f)$.

\subsection{Convergence properties of $g_1$}\label{Subsubsecg1}

Let us assume that $N \geq 2$. All the results derived here will require to compute the time derivative of $g_1(t,x,v)=f(t,\X_1+t\V_1,\V_1)$. Applying the chain rule, we have
\begin{align*}
\partial_t g_1(t,x,v) &= \big[\partial_t f\big]\big(t,\X_1+t\V_1,\V_1\big)+\big(\V_1+t\partial_t \V_1+\partial_t \X_1 \big)\cdot \big[ \nabla_x f \big]\big(t,\X_1+t\V_1,\V_1\big) \\
& \quad +\partial_t\V_1 \cdot \big[ \nabla_v f \big]\big(t,\X_1+t\V_1,\V_1\big).
\end{align*}
We now rewrite $[\partial_t f](t,\X_1+t\V_1,\V_1)+\V_1 \cdot [ \nabla_x f ](t,\X_1+t\V_1,\V_1)$ using the Vlasov equation, and $\nabla_v=G -t\nabla_x$. We then obtain
\begin{align}
\partial_t g_1(t,x,v) &= \big(\partial_t \X_1-t\mu \nabla_x \phi (t,\X_1+t\V_1) \big)\cdot \big[ \nabla_x f \big]\big(t,\X_1+t\V_1,\V_1\big)\label{timederivfirstexptechlemm}\\
& \quad +\big( \mu \nabla_x \phi (t,\X_1+t\V_1)+\partial_t \V_1 \big) \cdot \big[ G f \big]\big(t,\X_1+t\V_1,\V_1\big).\nonumber
\end{align}

Then, we determine the asymptotic behaviour of the force field along the modified spatial characteristics. Since the corrections are lower order, a result similar to Corollary \ref{Coralonglin} holds.

\begin{corollary}\label{Coralongmod}
Let $|\kappa| \leq N-2$. Then, for all $(t,x,v) \in [2,\infty) \times \R^3_x \times \R^3_v$, we have
$$  \big| t^{2+|\kappa|} \big[\nabla_x \partial_x^\kappa \phi \big] \big(t,x+tv+\mu [ \log(t)+1] \nabla_v \phi_\infty (v) \big) - \nabla_v \partial_v^{\kappa} \phi_\infty (v) \big| \lesssim \epsilon \, \langle x \rangle \,\log^{N}(t) t^{-1}.$$
\end{corollary}
\begin{proof}
We have $|x+\mu [\log(t)+1] \nabla_v \phi_\infty (v)| \lesssim |x|+\log(t)$, so it suffices to apply Corollary \ref{Coralonglin}. 
\end{proof}

More work is required to show that $g_1$ converges in $W^{N-2,\infty}_{x,v}$. We next compute and control the derivatives of the error terms in \eqref{timederivfirstexptechlemm}.

\begin{lemma}\label{Lemformodscattderiv}
Let $|\kappa| \leq N-2$. For all $(t,x,v) \in \R_+^* \times \R^3_x \times \R^3_v$, the derivative $\partial_t \partial_{x,v}^\kappa \big[f(t,\X_1+t\V_1,\V_1) \big]$ can be written as a linear combination of terms of the form
\begin{align*}
&\frac{1}{t}  \partial_{x,v}^\gamma \Big( \! \nabla_v \phi_\infty(v)-t^2 \nabla_x \phi (t,\X_1+t\V_1) \! \Big) \cdot \big[\nabla_x \partial_x^\alpha G^\beta f \big] \big(t,\X_1+t\V_1,\V_1 \big) \frac{\log^k(t)}{t^{p-k}} \prod_{1 \leq i \leq p}  \partial_v^{\xi_i} \phi_\infty^{k_i}(v), \\
& \frac{1}{t^2} \partial_{x,v}^\gamma \Big( \! \nabla_v \phi_\infty(v) -t^2\nabla_x \phi (t,\X_1+t\V_1)\! \Big) \cdot \big[G \partial_x^\alpha G^\beta f \big] \big(t,\X_1+t\V_1,\V_1 \big)\frac{\log^k(t)}{t^{p-k}} \prod_{1 \leq i \leq p}  \partial_v^{\xi_i} \phi_\infty^{k_i}(v), 
\end{align*}
where $ k \leq  |\alpha|$, $0  \leq k \leq  p \leq |\kappa|$, $|\gamma|+|\alpha|+|\beta| \leq |\kappa|$, $2 \leq |\xi_i| \leq |\kappa|+1$, and $k_i \in \llbracket 1 , 3 \rrbracket$.
\end{lemma}

\begin{remark}
The condition $k \leq |\alpha|$ is key to determine the asymptotic behaviour of the spatial average of the derivatives of $g_1$.
\end{remark}

\begin{proof}
We obtain the result from \eqref{timederivfirstexptechlemm}, and by iterating the relations
\begin{align*}
\partial_{x^i} \big[ h(t,\X_1+t\V_1,\V_1) \big] & = \big[\partial_{x^i} h \big] (t,\X_1+t\V_1,\V_1), \\
 \partial_{v^i} \big[ h(t,\X_1+t\V_1,\V_1) \big] & = \big[G_i h \big] (t,\X_1+t\V_1,\V_1)+\frac{\mu}{t}  \nabla_v \partial_{v^i}\phi_\infty(v) \cdot \big[G h \big] (t,\X_1+t\V_1,\V_1) \\
 &\qquad  \quad +\mu \log(t) \nabla_v \partial_{v^i}\phi_\infty(v) \cdot \big[\nabla_x h \big] (t,\X_1+t\V_1,\V_1),
\end{align*}
which hold for any distribution function $h : \R_+ \! \times \R^3_x \times \R^3_v \to \R$.
\end{proof}

To show the convergence of the derivatives of $g_1$, we need the following result to control the terms in Lemma \ref{Lemformodscattderiv}.

\begin{corollary}\label{Coralongmod2}
Let $|\kappa_x| \leq N-1$ and $|\kappa_v| \leq N-2$. For all $(t,x,v) \in [2,\infty) \times \R^3_x \times \R^3_v$, we have
\begin{align*}
\big| \partial_x^{\kappa_x} \big[ \nabla_x  \phi \big(t,x+tv+\mu [\log(t)+1] \nabla_v \phi_\infty  \big)\big] \big| & \lesssim \epsilon \, t^{-2-|\kappa_x|}, \\
  \big| t^{2} \partial_v^{\kappa_v} \big[ \nabla_x  \phi \big(t,x+tv+\mu [\log(t)+1] \nabla_v \phi_\infty  \big)\big] - \nabla_v \partial_v^{\kappa_v} \phi_\infty  \big| & \lesssim \epsilon \, \langle x \rangle\log^{N}(t) t^{-1}.
  \end{align*}
\end{corollary}

\begin{proof}
The first estimate follows from $\partial_{x^i} (\X_1^j)= \delta_i^j$ and Proposition \ref{proposition_estimate_phi}. For the second estimate, we note that $\partial_{v^i} (\X_1^j+t\V_1^j)=t\delta_{i,j}+\mu [\log(t) +1]\partial_{v^i} \partial_{v^j} \phi_\infty(v) $. Then, we write the difference
$$ \partial_v^{\kappa_v} \big[ \nabla_x  \phi \big(t,\X_1 +t\V_1\big)\big]-t^{|\kappa_v|}\big[\nabla_x \partial_x^{\kappa_v} \phi \big] (t,\X_1+t\V_1)$$
as a linear combination of terms of the form
$$ t^{|\alpha|-p}\big[\nabla_x \partial_x^{\alpha} \phi \big] (t,\X_1+t\V_1) \! \prod_{1 \leq i \leq p} (\log(t)+1) \partial_{v}^{\xi_i} \phi_\infty(v),$$ with $1 \leq p=|\alpha| \leq |\kappa_v|,$ and $2 \leq |\xi_i| \leq |\kappa_v|+1$. These last terms are bounded by $\epsilon^{1+p} t^{-3}  \log^p(t)$, so the result follows from Corollary \ref{Coralongmod}.
\end{proof}

Finally, we prove that $\X_1+tv$, and then $\X_1+t\V_1$, are good approximations to the nonlinear spatial characteristics. The following lemma allows to save $N_x$ powers of $\log(t)$ in the error terms.
 
\begin{lemma}\label{lemmaesti_f}
For all $(t,x,v) \in [2,\infty) \times \R^3_x \times \R^3_v$, we have
$$\big| \z (t,x,v)-\big\langle x-tv- \mu \log(t) \nabla_v \phi_\infty(v) \big\rangle \big| \leq 4 \, \z (t,x,v).$$
Then, for any $|\kappa_x|+|\kappa_v| \leq N-1$ and all $(t,x,v) \in [2,\infty) \times \R^3_x \times \R^3_v$, we have
\begin{align}
 \langle x \rangle^{N_x} \langle v \rangle^{N_v} \big| \partial_x^{\kappa_x} G^{\kappa_v} f \big| \big(t,\X_1(t,x,v)+tv,v \big) &\lesssim \log^{|\kappa_v|}(t),\nonumber\\
  \langle x \rangle^{N_x} \langle v \rangle^{N_v} \big| \partial_x^{\kappa_x} G^{\kappa_v} f \big| \big(t,\X_1(t,x,v)+t\V_1(t,x,v),\V_1(t,x,v) \big) &\lesssim \log^{|\kappa_v|}(t).\nonumber
  \end{align}
\end{lemma}

\begin{proof}
Let $z_{\mathrm{approx}}(t,x,v):=x-tv-\mu \log(t) \nabla_v \phi_\infty(v)$. Then,
\begin{align*}
\big| \T_\phi (z_{\mathrm{approx}}) \big|(t,x,v) & \leq \big| t  \nabla_x \phi(t,x)-t\nabla_x \phi (t,tv)+t\nabla_x \phi (t,tv)-t^{-1}\nabla_v \phi_\infty(v) \big| \\
& \qquad +\log(t) \big| \nabla_x \phi(t,x) \cdot \nabla_v^2 \phi_\infty (v) \big| .
\end{align*}
Thus, by the mean value theorem, and Propositions \ref{proposition_estimate_phi} and \ref{Proasympselfsiphi}, we have
\begin{align*}
\big| \T_\phi (z_{\mathrm{approx}}) \big|(t,x,v) & \lesssim \epsilon t^{-2} |x-tv|+\epsilon t^{-2} \log^N(t) \lesssim \epsilon t^{-2} \big|z_{\mathrm{approx}}\big|(t,x,v)+\epsilon t^{-2} \log^N(t).
\end{align*}
Denote by $\mathcal{F}(t,s,x,v)$ the nonlinear characteristic flow induced by $\T_\phi$. If $\epsilon$ is small enough, for all $t \geq 2$ we have
$$  \langle z_{\mathrm{approx}} \rangle (t,\mathcal{F}(t,2,x,v)) \leq 2  \langle z_{\mathrm{approx}} \rangle \,(2,x,v) \leq 3\langle x-2v \rangle \leq 4 \z (2,x,v) = 4\z (t,\mathcal{F}(t,2,x,v)) ,$$
where, in the last step, we used that $\z$ is conserved along $\mathcal{F}$. We then obtain the desired estimate by applying Duhamel formula, since $$|\T_\phi (\z-\langle z_{\mathrm{approx}} \rangle)| \leq |\T_\phi ( z_{\mathrm{approx}} )|.$$

The first estimate for $f$ then follows from Proposition \ref{prop_point_bound_deriv_distrib}. Moreover, the second estimate holds since $|v-\V_1(t,x,v)| \lesssim \epsilon t^{-1}$.
\end{proof}

We are now able to prove a modified scattering result for $f$ in a strong topology. Recall for this that $g_1(t,x,v)=f(t,\X_1+t\V_1,\V_1)$.

\begin{proposition}\label{prop_mod_scatteringg_1}
There exists a scattering state $f_{\infty}\in C^{N-2}(\R^3_x\times\R^3_v)$ such that the modified profile $g_1$ converges to $f_\infty$. More precisely, for any $|\kappa|\leq N-2$ and all $(t,x,v) \in [2,\infty ) \times \R^3_x \times \R^3_v$, we have $$ \langle x \rangle^{N_x-1} \, \langle v \rangle^{N_v}\Big| \partial_{x,v}^{\kappa} g_1(t,x,v) - \partial_{x,v}^\kappa f_{\infty}(x,v)\Big|\lesssim \frac{\log^{2N}(t)}{t}.$$ 
\end{proposition}

 \begin{remark}
If $(N_x,N_v)=(8,7)$, then the RHS could be improved to $\epsilon t^{-1} \log^{2N}(t)$.
 \end{remark}
 
\begin{proof}
Lemma \ref{Lemformodscattderiv} and Corollary \ref{Coralongmod2} imply that for any $|\kappa| \leq N-2$ and all $(t,x,v) \in [2,\infty) \times \R^3_x \times \R^3_v$, we have
\begin{equation*}
 \big| \partial_t \partial_{x,v}^\kappa \big[ f(t,\X_1+t\V_1,\V_1) \big] \big| \lesssim  \epsilon \log^{N+|\alpha|}(t) t^{-2} \, \langle x \rangle \sup_{|\alpha|+|\beta| \leq N-1}   \big|\partial_x^\alpha G^\beta f \big| (t,\X_1+t\V_1,\V_1).
 \end{equation*}
We then deduce from Lemma \ref{lemmaesti_f} that, for any $|\kappa| \leq N-2$ and all $(t,x,v) \in [2,\infty) \times \R^3_x \times \R^3_v$, 
\begin{equation}\label{eq:partialtg1}
\langle x \rangle^{N_x-1} \, \langle v \rangle^{N_v}  \big| \partial_t \partial_{x,v}^\kappa \big[ f(t,\X_1+t\V_1,\V_1) \big] \big|  \lesssim  \log^{2N-1}(t) t^{-2},
 \end{equation}
This estimate implies that $g_1$ converges to a $C^{N-2} (\R^3_x \times \R^3_v)$ function with the stated rate of convergence. 
\end{proof}  

\begin{remark}
As $|\V_1-v| = \mu t^{-1} \nabla_v \phi_\infty (v)$, we have $f_\infty (x,v)=\widetilde{f}_{\infty} (x+\mu \nabla_v \phi_\infty (v),v)$ by the mean value theorem. Here, $\widetilde{f}_\infty$ is the limit function in Proposition \ref{prop_mod_scattering_statem_high_reg_proof}.
 \end{remark}

We now derive some direct consequences of Proposition \ref{prop_mod_scatteringg_1} that we will use below. In particular, we emphasise that the weighted spatial averages of $g_1$ converge, contrary to the ones of $g_0$.

\begin{corollary}\label{Corestig1}
Let $|\kappa| \leq N-2$. There exists $C>0$ such that for all $(x,v) \in \R^3_x \times \R^3_v$ and all $t \in [2,\infty]$, we have
$$\langle x \rangle^{N_x-1} \langle v \rangle^{N_v} | \partial_{x,v}^\kappa g_1 | (t,x,v) \leq C,$$ where we set $g_1(\infty,x,v):=f_\infty(x,v)$. As a result, for all $v \in  \R^3_v$ and all $t \in [2,\infty]$, we have
$$ \langle v \rangle^3 \int_{\R^3_z} \langle z \rangle^{N_x-5} \, | \partial_{x,v}^\kappa g_1 | (t,z,v) \dr z \lesssim C.$$ Finally, for any $|\xi| \leq N_x-5$, we have
$$ \forall \, (t,v) \in [2,\infty) \times \R^3_v , \qquad \langle v \rangle^3 \bigg|\int_{\R^3_z} z^\xi \partial_{x,v}^\kappa g_1  (t,z,v) \dr z-\int_{\R^3_z} z^\xi \partial_{x,v}^\kappa f_\infty (z,v) \dr z \bigg| \lesssim  \frac{\log^{2N}(t)}{t}.$$
\end{corollary}

\begin{proof}
The first part of the statement is obtained by integrating \eqref{eq:partialtg1} between $t=2$ and $t \in [2,\infty]$, together with Lemma \ref{lemmaesti_f}, since
$$ |\partial_{x,v}^\kappa g_1 (2,x,v) | \lesssim \sup_{|\alpha| \leq |\kappa|} \big| Z^\alpha f| \Big(2,x+2v+\mu[\log(2)+1]\nabla_v \phi_\infty,v+\frac{\mu}{2}\nabla_v \phi_\infty \Big).$$ We deduce the second estimate using $z \mapsto \langle z \rangle^{-4} \in L^1(\R^3_z)$. Finally, the third estimate follows similarly from Proposition \ref{prop_mod_scatteringg_1}.
\end{proof}

Before proving the expansion for $t^3 \rho (f)$, we need to improve the rate of convergence of the spatial average of $\partial_v^\kappa g_1$. For this, we now require $N \geq 3$.

\begin{proposition}\label{Proconvsecondorder}
For all $(t,v) \in [2,\infty) \times \R^3_v$, we have
 $$ \langle v \rangle^3 \bigg| \int_{\R^3_z} g_1 (t,z,v) \dr z- \int_{\R^3_z}\Big[ 1-  \mu\frac{1}{t} \Delta_v \phi_\infty (v) \Big] f_\infty (z,v) \dr z \bigg| \lesssim \epsilon\frac{\log^{2N}(t)}{t^{2}}.$$ Moreover, for any $|\kappa| \leq N-3$ there exist $n_{\beta,\gamma}^\kappa \in \mathbb{Z}$ such that
 \begin{align*}
  \langle v \rangle^3 \bigg| \int_{\R^3_z}\! \partial_v^\kappa g_1 (t,z,v) \dr z~- \int_{\R^3_z}\! \partial_v^\kappa f_\infty (z,v) \dr z~&-\sum_{|\gamma|+|\beta|\leq|\kappa|}\! \frac{n_{\beta,\gamma}^\kappa}{t} \!\int_{\R^3_z}\!  \Delta_v  \partial_v^\gamma \phi_\infty (v)   \partial_v^\beta f_\infty (z,v) \dr z  \bigg|  \lesssim  \epsilon t^{-2}\log^{2N}(t).
  \end{align*}
\end{proposition}

\begin{proof}
 Let $|\kappa| \leq N-3$ and $(t,v) \in [2,\infty) \times \R^3_v$. Combining Lemma \ref{Lemformodscattderiv}, Corollary \ref{Coralongmod2}, and Lemma \ref{lemmaesti_f}, we have
 \begin{align*}
 & \langle v \rangle^3 \bigg| \partial_t \partial_v^\kappa  \bigg( \int_{\R^3_z}\partial_v^\kappa g_1 (t,z,v) \dr z \bigg)- \sum_{|\gamma|+|\beta| \leq |\kappa|} n_{\beta,\gamma}^\kappa \mathfrak{I}^\kappa_{\gamma,0,\beta}(t,v)\bigg| \\
 & \qquad \qquad \qquad \qquad  \lesssim \epsilon \frac{ \log^{2N}(t)}{t^{3}} +\sum_{1 \leq |\alpha| \leq |\kappa|} \; \sum_{|\gamma|+|\alpha|+|\beta| \leq |\kappa|} \log^{|\alpha|}(t) \langle v \rangle^3\big|\mathfrak{I}^\kappa_{\gamma,\alpha,\beta}\big|(t,v),
 \end{align*}
 where
 $$ \mathfrak{I}^\kappa_{\gamma,\alpha, \beta} (t,v) := \int_{\R^3_x} \frac{1}{t}  \partial_{v}^\gamma \Big( \! \nabla_v \phi_\infty(v)-t^2 \nabla_x \phi (t,\X_1+t\V_1) \! \Big) \cdot \big[\nabla_x \partial_x^\alpha G^\beta f \big] \big(t,\X_1+t\V_1,\V_1 \big) \dr x.$$
Since $\X_1-x$, $\V_1$, and $\partial_t \X_1$, do not depend on $x$, we show by integration by parts that
 $$ \mathfrak{I}^\kappa_{\gamma,\alpha, \beta} (t,v) = t\int_{\R^3_x}  \partial_{v}^\gamma \partial_x^\alpha \Big(  \Delta_x \phi (t,\X_1+t\V_1)  \Big) \big[ G^\beta f \big] \big(t,\X_1+t\V_1,\V_1 \big) \dr x$$ do not depend on $x$ either. Thus, we have $$ \log^{|\alpha|}(t) \langle v \rangle^3\big|\mathfrak{I}^\kappa_{\gamma,\alpha,\beta}\big|(t,v) \lesssim \epsilon \log^{2N}(t) t^{-2-|\alpha|},$$ by Lemma \ref{lemmaesti_f} and Corollary \ref{Coralongmod2}. However, for $|\alpha|=0$ the term $\mathfrak{I}^\kappa_{\gamma,\alpha, \beta}$ does not decay faster than $t^{-2}$. Nonetheless, we have identified the leading order contribution of $\Delta_x \phi $, so we can find a correction to the spatial average of $g_1$ allowing for a stronger rate of convergence. 
 
We note first that by performing similar computations as in the proof of Corollary \ref{Coralongmod2}, we obtain
$$ \Big| \partial_{v}^\gamma \Big(  \Delta_x \phi \big(t,\X_1+t\V_1 \big)  \Big) -t^{|\gamma|} \big[ \Delta_x \partial_x^\gamma  \phi \big] \big(t,\X_1+t\V_1 \big) \Big| \lesssim \epsilon \log^{|\gamma|}(t) t^{-4} .$$ We then deduce, by peforming first the change of variables $z=x+tv+\mu[\log(t)+1] \nabla_v \phi_\infty (v)$ and using Lemma \ref{lemmaesti_f} that
 \begin{align*}
  \langle v \rangle^3& \bigg|\mathfrak{I}^\kappa_{\gamma,0,\beta}- \frac{1}{t^2}\Delta_v \partial_v^\gamma \phi_\infty(v) \int_{\R^3_z}  \big[ G^\beta f \big] \big(t,z,\V_1 \big) \dr z \bigg| \\
  & \; \; \lesssim \epsilon\frac{\log^{2N}(t)}{t^{3}}+ t \langle v \rangle^3 \! \int_{\R^3_x} \Big|\big[\Delta_x \partial_v^\gamma \big] \phi (t,\X_1+t\V_1)-\frac{1}{t^3}\Delta_v \partial_v^\gamma \phi_\infty(v) \Big| \big| Z^\beta f \big| \big(t,\X_1+t\V_1,\V_1 \big)  \dr x .
  \end{align*}
Finally, we bound the RHS of the previous inequality by $\epsilon  \log^{2N}(t) t^{-3}$ by applying Corollary \ref{Coralongmod} and Lemma \ref{lemmaesti_f}, as $|\gamma| \leq N-2$. One can then derive the result by using Proposition \ref{ProconvQ}. For the case $|\kappa|=0$, the coefficient $n^\kappa_{\gamma,\beta}$ is given by \eqref{timederivfirstexptechlemm}.
 \end{proof}

\subsection{First order expansion of $t^3 \rho (f)$}\label{Subsec72}

In this subsection, we assume $N \geq 4$. As we need to express $\rho (f)$ using the distribution $g_1$, we have to study the maps
$$ \Psi_t \colon(z,v) \mapsto \big(z+\mu \log(t) \nabla_v \phi_\infty(v),v+\mu t^{-1} \nabla_v \phi_\infty(v) \big), \qquad \Upsilon_t \colon (a,b) \mapsto (a+tb,b),$$ for $ t \geq 2$. In terms of these maps, we have $$g_1(t,\cdot , \cdot)= f(t, \Upsilon_t \circ \Psi_t).$$ For our purposes, we need to show that $\Psi_t$ is invertible and consider 
\begin{equation}\label{keyident_backwardsdistrib}
f(t,z,v) =  g_1 \big(t, \Psi_t^{-1} (z-tv,v ) \big)=g_1(t,z-tv +\mathfrak{X}_{t}(v),v+\mathfrak{V}_{t}(v) \big),
\end{equation}
where $(\mathfrak{X}_t,\mathfrak{V}_t):= \Psi_t^{-1}-\mathrm{id}$. In the following, we denote the space of matrices of size $d\times d$ with real coefficients by $\mathfrak{M}_d(\R) $. We begin proving a preparatory result.

\begin{lemma}\label{Lemdiffeo}
Let $d \geq 1$ and $k \geq 1$. Let $h \in C^k(\R^d,\R^d)$ be a map verifying $\| h \|_{W^{k,\infty}}\! < \!\infty$, and $\| \dr h \|_{L^\infty}  < 1$. Then, the map $H:=\mathrm{id}-h$ is a $C^k$-diffeomorphism of $\R^d$, and
\begin{equation}\label{eq:invH}
 \big[ \dr H \big]^{-1} = \mathrm{id}+\sum_{i \geq 1} \big[ \dr h \big]^k \quad \text{in $C^{k-1}\big(\R^d,\mathfrak{M}_d(\R ) \big)$},
 \end{equation}
where $[\dr h\big]$ is the Jacobian matrix of $ h$, and $[\dr h]_{ij}:=\partial_{x^j} h^i$ for all $1 \leq i, \, j \leq d$. 
\end{lemma}

\begin{proof} For the first part of the statement, we have three steps:
\begin{itemize}
\item $H$ is injective. This follows from the inequality
$$ \color{white} \square \color{black} \qquad \big||x-y|-|H(x)-H(y)|\big|\!\leq \big|x-y-H(x)+H (y) \big| \leq |h(x)-h(y)| \leq \| \dr h \|_{L^\infty} |x-y|<|x-y|,$$ which relies on the mean value theorem.

\item $\mathrm{d}H$ is a local $C^{k}$ diffeomorphism. As $\|\dr h \|_{L^\infty} < 1$, the map $\dr H=\mathrm{id}-\dr h$ is invertible for all $x \in \R^d$, so \eqref{eq:invH} holds in $C^0(\R^d,\mathfrak{M}_d(\R))$. \eqref{eq:invH} holds as well in high regularity as $\| \dr h \|_{W^{k-1,\infty}} <\infty$.

\item $H$ is surjective. Let $(x_n)_{n \geq 0}$ be a sequence in $\R^d$ such that $H(x_n) \to y_\infty$ as $n \to  \infty$, with $y_\infty \in \R^d$. Since $H - \mathrm{id} \in L^\infty(\R^d,\R^d)$, we obtain that $(x_n)_{n \geq 0}$ is bounded. By the Bolzano-Weierstrass theorem and by continuity, there exists $x_\infty \in \R^d$ such that $H(x_\infty)=y_\infty$. As $H(\R^d)$ is open by the previous step, it yields $H(\R^d)=\R^d$.
\end{itemize}
The second part of the statement is a direct consequence of the first one.
\end{proof}

We are now able to study $\Psi_t$. In particular, we study fine properties of the components of $\Psi_t^{-1} - \mathrm{id}=(\mathfrak{X}_t,\mathfrak{V}_t)$, which will allow us to estimate $f$ by \eqref{keyident_backwardsdistrib}.

\begin{lemma}\label{LemPhiminus1}
For all $t \geq 2$, the map $\Psi_t$ is a $C^{N-2}$-diffeomorphism of $\R^3_x \times \R^3_v$. The spatial and velocity components of $\Psi_t^{-1} - \mathrm{id}=(\mathfrak{X}_t,\mathfrak{V}_t)$ are independent of the spatial variable. Moreover, these components verify for all $v \in \R^3_v$ that
\begin{align*}
 \mathfrak{X}_t(v) &=-\mu \log(t)\nabla_v \phi_\infty(v) +\sum_{1 \leq k \leq N-3} t^{-k}\log(t)\phi_k^X(v)+ \epsilon^{N-1} \, O \big( t^{-N+2} \log(t) \big), \\ 
\mathfrak{V}_t(v) &= -\mu t^{-1} \nabla_v \phi_\infty (v)+\sum_{2 \leq k \leq N-2} t^{-k} \phi^V_k(v) + \epsilon^{N-1} O\big(t^{-N+1} \big),
\end{align*}
where $\phi_{k-1}^X, \, \phi_{k}^V \in C^{N-1-k}(\R^3,\R^3)$, and their components are linear combination of terms of the form
$$ \prod_{1 \leq i \leq k} \partial_v^{\gamma_i} \phi_\infty^{k_i}(v), \qquad 1 \leq |\gamma_i| \leq k, \quad k_i \in \llbracket 1 , 3 \rrbracket .$$
Moreover, for any $|\kappa| \leq N-3$, we have
\begin{align*}
 \partial_v^\kappa \mathfrak{X}_t(v) &=-\mu \log(t)\nabla_v \partial_v^\kappa \phi_\infty(v) +\sum_{1 \leq k \leq N-3-|\kappa|}  t^{-k} \log(t)\partial_v^\kappa\phi_k^X(v)+ \epsilon^{N-1-|\kappa|} \, O \big( t^{-N+2+|\kappa|} \log(t) \big), \\ 
\partial_v^\kappa \mathfrak{V}_t(v) &= -\mu t^{-1} \nabla_v \partial_v^\kappa\phi_\infty (v)+\sum_{2 \leq k \leq N-2-|\kappa|} t^{-k} \partial_v^\kappa\phi^V_k(v) + \epsilon^{N-1-|\kappa|} O\big(t^{-N+1+|\kappa|} \big).
\end{align*}
Finally, $\|\mathfrak{X}\|_{\dot{W}^{N-2,\infty}} \lesssim \epsilon \log(t)$ and $\|\mathfrak{V}\|_{\dot{W}^{N-2,\infty}} \lesssim t^{-1} \epsilon$.
\end{lemma}

 \begin{proof}
We fix $t \geq 2$ and decompose $$\Psi_t= \Phi_{t,1} \circ \Xi_t,$$ where
\begin{equation*}
\Phi_{t,1} (y,w)= \big(y,w+\mu t^{-1} \nabla_v \phi_\infty(w) \big), \qquad  \Xi_t (z,v)=\big( z+\mu \log(t) \nabla_v \phi_\infty(v),v \big).
 \end{equation*}
Clearly, $\Xi_t$ is a $C^{N-2}$ diffeomorphism of $\R^3 \times \R^3$ and 
$$ \Xi_t^{-1} (z,v) = \big( z-\mu \log(t) \nabla_v \phi_\infty(v),v \big).$$
Let us now prove that $\Phi_{t,1}$ is a $C^{N-2}$-diffeomorphism. For this, it suffices to study its velocity component 
$$\Omega_t :w \mapsto w +\mu t^{-1} \nabla_v \phi_\infty(w).$$ As $\| \nabla_v \phi_\infty \|_{W^{N-2,\infty}} \lesssim \epsilon$ by Proposition \ref{Proasympselfsiphi}, we show by applying Lemma \ref{Lemdiffeo} that $\Omega_t$ is a $C^{N-2}$-diffeomorphism, and 
\begin{equation}\label{eq:defandexpaOmega}
[\dr \Omega_t ]^{-1}=\mathrm{id}+\sum_{k \geq 1} \, [\dr h_t]^k \quad \text{in $C^{N-3} \big(\R^3_v,\mathfrak{M}_3(\R) \big)$}, \qquad [\dr h_t]_{ij}=-\mu t^{-1} \partial_{v^i} \partial_{v^j} \phi_\infty ,
\end{equation}
where $h_t(w):= - \mu t^{-1}\nabla_v \phi_\infty (w)$. Thus, $\Psi_t$ is a $C^{N-2}$-diffeomorphism, and
\begin{equation}\label{eq:expressionPhiinv}
 \Psi_t^{-1}(z,v)= \big( z-\mu \log(t) \nabla_v \phi_\infty \big(\Omega_t^{-1}(v) \big) ,\Omega_t^{-1}(v) \big).
 \end{equation}
The next step consists in deriving an expansion for $\Omega_t^{-1}$. Let $v \in \R^3$ and $w=\Omega_t^{-1}(v)$, so that $v=w+\mu t^{-1} \nabla_v \phi_\infty (w)$. Then, by applying the mean value theorem we have
$$ |v-w| \lesssim \epsilon t^{-1} , \qquad  \big|v-\mu t^{-1} \nabla_v \phi_\infty (v)-w \big| \lesssim t^{-1} \big\|\nabla_v^2 \phi_\infty \big\|_{L^\infty_v} |v-w| \lesssim \epsilon^2 t^{-2},$$ as $\| \nabla_v \phi_\infty\|_{W^{N-2,\infty}} \lesssim \epsilon$. Iterating the above, and using Taylor's theorem instead of the mean value theorem, we obtain
\begin{equation}\label{eq:defexpanOmeg00}
 \mathfrak{V}_t(v)= \Omega_t^{-1}(v)-v=-\mu t^{-1} \nabla_v \phi_\infty (v)+\sum_{2 \leq k \leq N-2} t^{-k} \phi^V_k(v) + \epsilon^{N-1} O\big(t^{-N+1} \big),
 \end{equation}
where $ \phi^V_k \in C^{N-1-k}(\R^3,\R^3)$. In fact, $\phi_k^V$ is a product of $k$ derivatives of $\partial_{v^i} \phi_\infty$. Together with the estimate $\|\nabla_v \phi_\infty \|_{W^{N-2,\infty}} \lesssim \epsilon$ and \eqref{eq:expressionPhiinv}, we obtain the expansion for $\mathfrak{X}_t$. We next remark that
$$ \dr \Psi_t^{-1}(z,v)=\dr \Xi_t^{-1} \big(z, \Omega_t^{-1}(v) \big) \circ \dr \Phi_{t,1}^{-1} (z,v), \qquad \dr \Phi^{-1}_{t,1}(z,v)\cdot (h,k)= \big(h, \big[\dr \Omega_t\big]^{-1}\big(\Omega_t^{-1}(v) \big) \cdot k \big),$$
where $(h,k)$ is a tangent vector in $\R^3_z \times \R^3_v$. The existence of the expansions for the derivatives of $(\mathfrak{X}_t,\mathfrak{V}_t)$ then ensue from an induction, \eqref{eq:defandexpaOmega}, \eqref{eq:defexpanOmeg00}, Leibniz rule, and Taylor's theorem. Lemma \ref{LemForexp} finally shows that the coefficients of these time expansions are the derivatives of the ones for $(\mathfrak{X}_t,\mathfrak{V}_t)$. 
 \end{proof}

The next step consists in performing the change of variables $y = x-tv+\mathfrak{X}_t(v)$ for fixed $x$. For this purpose, we study the map $v \mapsto z-tv+\mathfrak{X}(v)$, and its inverse.

 \begin{lemma}\label{LemYorder2}
Let $t\geq 2$ and $z\in \R^3_z$. The map $\Y_{t,z} : v \mapsto z-tv+\mathfrak{X}(v)$ is a $C^{N-2}(\R^3_v,\R^3)$-diffeomorphism. Moreover, the inverse $\Y^{-1}_{t,z}$ satisfies that for all $y \in \R^3$, we have
\begin{equation}\label{est11}
\Y^{-1}_{t,z}(y)  =   \frac{z-y}{t} -\mu \frac{\log(t)}{t} \nabla_v \phi_\infty \Big( \frac{z}{t}  \Big)+ \!  \sum_{\substack{2 \leq q \leq N-2 \\ |\alpha|+p \leq q}}\frac{y^\alpha \log^{p}(t)}{t^q} \mathbf{A}_{q,\alpha,p}\Big( \frac{z}{t}  \Big) + \epsilon \, O \bigg( \frac{ \langle y \rangle^{N-2}  \log^{N-1} (t) }{ t^{N-1}} \bigg) ,
 \end{equation}
where $\mathbf{A}_{q,\alpha,p} \in C^0(\R^3,\R^3)$ and $\| \mathbf{A}_{q,\alpha,p} \|_{L^\infty} \lesssim \epsilon^{1+p}$. The Jacobian determinant satisfies
\begin{equation}\label{estdet}
 t^{3} \big|\det \mathrm{d} \Y_{t,z}^{-1} \big|(y) =1-  \mu\frac{\log(t)}{t}  \Delta_v \phi_\infty\Big( \frac{z}{t} \Big) + \!  \sum_{\substack{2 \leq q \leq N-3 \\ |\alpha|+p \leq q}}\frac{y^\alpha \log^{p}(t)}{t^{q}} \mathbf{B}_{q,\alpha,p}\Big( \frac{z}{t}  \Big) + \epsilon \,  O \bigg( \frac{ \langle y \rangle^{N-3}  \log^{N-2}(t)}{ t^{N-2}} \bigg) ,
 \end{equation}
 where $\mathbf{B}_{q,\alpha,p} \in C^0(\R^3,\R^3)$ and $\| \mathbf{B}_{q,\alpha,p} \|_{L^\infty} \lesssim \epsilon^{1+p}$. Finally, we have
\begin{align}
\Y^{-1}_{t,z}(y)+\mathfrak{V}_t \circ \Y^{-1}_{t,z}(y)  &=   \frac{z}{t}-\frac{y}{t} -\mu \frac{\log(t)+1}{t} \nabla_v \phi_\infty \Big( \frac{z}{t}  \Big) \label{est22}
 \\ & \quad + \! \sum_{2 \leq q \leq N-2} \, \sum_{ |\alpha|+p \leq q}\frac{y^\alpha \log^{p}(t)}{t^q} \mathbf{\overline{A}}_{q,\alpha,p}\Big( \frac{z}{t}  \Big) + \epsilon \,  O \bigg( \frac{ \langle y \rangle^{N-2}  \log^{N-1}(t)}{ t^{N-1}} \bigg) ,\nonumber
 \end{align}
where $\mathbf{\overline{A}}_{q,\alpha,p} \in C^0(\R^3,\R^3)$ and $\| \mathbf{\overline{A}}_{q,\alpha,p} \|_{L^\infty} \lesssim \epsilon^{1+p}$. 
 \end{lemma}
 \begin{proof}
Let $t \geq 2$ and $z \in \R^3$. The analysis can be reduced to the study of 
\begin{equation}\label{eq:expandZminus1}
\mathfrak{Z}(v):=v-t^{-1}\mathfrak{X}_t(v), \qquad \det \dr \Y_{t,z}(v) = -t^3 \det \dr \mathfrak{Z}(v).\end{equation}
Note that $y=\Y_{t,z}(v) $ if and only if $$\mathfrak{Z}(v)=\frac{z-y}{t}.$$ 

We next show estimates for $\mathfrak{Z}$ from where the stated estimates \eqref{est11}--\eqref{est22} are obtained. Applying Lemma \ref{Lemdiffeo} with $(H,h)=(\mathfrak{Z},t^{-1} \mathfrak{X})$, together with Lemma \ref{LemPhiminus1}, we get that $\mathfrak{Z}$ is a $C^{N-2}$-diffeomorphism of $\R^3$ and
$$ \big[\dr \mathfrak{Z} \big]^{-1} = \mathrm{id} + \sum_{k \geq 1} t^{-k} \big[ \dr \mathfrak{X}_t \big]^k \quad \text{in $C^{N-3} \big( \R^3,\mathfrak{M}_3(\R^3) \big)$}, $$ where $[ \dr \mathfrak{X}_t ]_{ij}=\partial_{v^j} \mathfrak{X}_t^i$. We now have to derive an expansion for $\mathfrak{Z}$. The strategy is similar to the one for $\Omega_t^{-1}$. Let $v \in \R^3$ and $w = \mathfrak{Z}^{-1}(v)$, so that $v=w-t^{-1}\mathfrak{X}_t (w)$. Using $\| \mathfrak{X}\|_{W^{N-2,\infty}} \lesssim \epsilon \log(t)$, we then get $|v-w| \lesssim \epsilon t^{-1}\log(t)$. By an induction, Taylor's theorem, and the expansion for $\mathfrak{X}_t$ given by Lemma \ref{LemPhiminus1}, we have
\begin{equation}\label{eq:defexpanOmeg}
 \mathfrak{Z}^{-1}(v)=v-\mu \frac{\log(t)}{t} \nabla_v \phi_\infty (v)+\sum_{2 \leq k \leq N-2} \frac{\log^k(t)}{t^k} \psi_k(v) + \epsilon^{N-1} O\bigg(\frac{\log^{N-1}(t)}{t^{N-1}} \bigg),
 \end{equation}
where $ \psi_k \in C^{N-1-k}(\R^3,\R^3)$ and $\|  \psi_k \|_{L^\infty} \lesssim \epsilon^k$. Two other applications of Taylor's theorem, with one of them relying on \eqref{eq:defexpanOmeg} and the expansion for $\mathfrak{V}$ given in Lemma \ref{LemPhiminus1}, provide the estimates for $\mathfrak{Z}^{-1} (\frac{z-y}{t})$ and $\mathfrak{Z}^{-1} (\frac{z-y}{t})+\mathfrak{V}_t \circ \mathfrak{Z}^{-1} (\frac{z-y}{t})$, which allow to show \eqref{est11} and \eqref{est22}. The expansion \eqref{estdet} for the Jacobian determinant of $\Y_{t,z}$ follows from \eqref{eq:expandZminus1}--\eqref{eq:defexpanOmeg}, Lemma \ref{LemPhiminus1}, the multi-linearity of the determinant, and $$ \dr \mathfrak{Z}^{-1}(y)= \big[ \dr \mathfrak{Z} \big]^{-1} \big( \mathfrak{Z}^{-1}(y) \big), \qquad   \det(I_3+M)=1+\mathrm{Tr}(M)+O(|M|^2).$$
 \end{proof}
 
We are now able to derive an improved expansion for the spatial density. The improved expansion for derivatives of the spatial density will be handled in the next section. We apply the previous results with $N=4$.

\begin{proposition}\label{ProfirstorderexpansionRho}
For all $(t,x) \in [2,\infty) \times \R^3_x$, there holds
\begin{align*}
t^3 \rho (f) (t,x) =& \bigg[1-\mu \frac{ \log(t)+1}{t} \Delta_v \phi_\infty \Big( \frac{x}{t} \Big) \bigg] \int_{\R^3_y} f_\infty \Big( y,\frac{x}{t}   \Big) \dr y \\ 
& \quad-\frac{1}{t}\int_{\R^3_y}\Big[ y +\mu \big[\log(t)+1\big] \nabla_v \phi_\infty \Big( \frac{x}{t}  \Big) \Big] \cdot \big[ \nabla_v f_\infty  \big] \Big( y,\frac{x}{t}  \Big) \dr y+ \epsilon \, O \bigg( \frac{\log^{8}(t)}{ t^{2}} \bigg).
\end{align*}
\end{proposition}
\begin{proof}
 Let $(t,x) \in [2,\infty) \times \R^3_x$. Performing the change of variables $y=\Y_{t,x}(v)$, allowed by Lemma \ref{LemYorder2}, we have
\begin{align*}
 \rho (f) (t,x) & =  \int_{\R^3_v} g_1 \big(t, x-tv+\mathfrak{X}_t(v),v+\mathfrak{V}_t(v) \big) \dr v =  \int_{\R^3_v} g_1 \Big(t, \Y_{t,x}(v), v+\mathfrak{V}_t(v) \Big) \dr v \\
& =  \int_{\R^3_y} g_1 \Big(t,y, \Y_{t,x}^{-1} (y)+\mathfrak{V}_t \big( \Y_{t,x}^{-1}(y) \big) \Big) \big| \det \dr \Y_{t,x}^{-1} \big|(y) \dr y.
\end{align*}
We now apply Taylor's theorem by using the expansion of Lemma \ref{LemYorder2} up to first order. This argument yields, by using the estimates of Corollary \ref{Corestig1} for $\partial_v^{\kappa_v}g_1$ with $|\kappa_v | \leq 2$, that
\begin{align*}
t^3 \rho (f) (t,x) & = \int_{\R^3_y} g_1 \Big( t,y,\frac{x}{t}   \Big) \dr y-\mu\frac{\log(t)}{t} \Delta_v \phi_\infty \Big( \frac{x}{t} \Big)  \int_{\R^3_y} g_1 \Big( t,y,\frac{x}{t}   \Big) \dr y  -\frac{1}{t}\int_{\R^3_y} y \cdot \big[ \nabla_v g_1  \big] \Big( t,y,\frac{x}{t}  \Big) \dr y
\\ & \qquad  -\mu \frac{\log(t)+1}{t} \nabla_v \phi_\infty \Big( \frac{x}{t}  \Big) \cdot \int_{\R^3_y}  \big[ \nabla_v g_1  \big] \Big( t,y,\frac{x}{t}  \Big) \dr y+ \epsilon \, O \bigg( \frac{\log^2(t)}{t^2} \bigg).
\end{align*}
Finally, we treat the terms on the RHS by applying Proposition \ref{Proconvsecondorder} for the first term, and Corollary \ref{Corestig1} for the remaining ones.
\end{proof}

\section{Late-time asymptotics}\label{SEcLatetime}

During this section, we set $N \geq 3$, $N_v \geq 7$ and $N_x \geq 2\sqrt{2N}+5$. We note that $N_x \geq 2n+8$ if $n$ satisfies $r_{n+1} \leq N$. We consider further a solution $f$ to the Vlasov-Poisson system verifying $\mathbb{E}_N^{8,7}[f_0] \leq \epsilon$ and $\mathbb{E}_N^{N_x,N_v}[f_0]<\infty$.

\subsection{Enhanced modified characteristics and asymptotics}

The purpose of this section consists in proving the properties stated in Proposition \ref{Proinduction} by an induction argument. The base case $n=0$ has been treated in Sections \ref{SecModiscat}--\ref{SecRhoOrder2}. In the course of the proof, we will derive the asymptotic self-similar polyhomogeneous expansions satisfied by the normalised spatial density $t^3 \!\rho (G^\beta \! f)$ and normalised force field $t^{2+|\beta|}\nabla_x \partial_x^\beta \phi$. 

For convenience, we will use two sequences $(r_n)_{n \geq 1}$ and $(S_n)_{n \geq 1}$ which satisfy the following conditions
\begin{alignat}{2}
 \nonumber r_{n+1}&=r_n+n+1, \qquad \qquad &&r_1=2, \\
 \nonumber S_{n+1} &\geq S_n+N+1-r_{n+1},  \qquad \qquad  &&S_1\geq 2N.
\end{alignat}
The sequence $(r_n)$, already introduced in \eqref{kev:defSn}, will express the number of derivatives required to derive expansions of order n for the solutions of the system. The sequence $(S_n)$ will express the logarithmical growth in the error terms at order $n$. One can check that a possible choice for $(S_n)$ is $S_n=N(n+2)$.

\begin{proposition}\label{Proinduction}
Let $n \in \mathbb{N}$ such that $r_n \leq N$, $ \X_0(t,x,v) := x,$ and $\V_0(t,x,v):=v$. There exist modifications of the linear characteristic flow
\begin{align}
\X_{n}(t,x,v) & =x+ \mu \nabla_v \phi_{\infty}(v) \log(t)+\sum_{1 \leq q \leq n-1} \, \sum_{|\alpha|+ p \leq q} \frac{x^\alpha \log^p(t)}{t^{q}} \mathbb{X}_{q,\alpha,p}(v), \label{eq:defCX }\\
\V_{n}(t,x,v) &=v+  \frac{\mu}{t}\nabla_v \phi_{\infty}(v)+\sum_{1 \leq q \leq n-1} \, \sum_{|\alpha|+ p \leq q} \frac{x^\alpha \log^p(t)}{t^{q+1}} \mathbb{V}_{q,\alpha,p}(v), \label{eq:defCV}
\end{align}
where $\mathbb{X}_{1,\alpha,0}=-\nabla_v \partial_v^\alpha \phi_\infty$ for $|\alpha|=1$, and $\mathbb{X}_{q,\alpha,p}, \; \mathbb{V}_{q,\alpha,p}\in C^{N-r_{q+1}}\cap W^{N-r_{q+1}} (\R^3_v)$, such that the following properties hold. For any $n \in \mathbb{N}$, if $r_{n+1} \leq N$ and $|\beta| \leq N-r_{n+1}$: 
\begin{enumerate}[label = (\alph*)] 
\item Modified scattering holds with an enhanced rate of convergence. Let $g_{n+1}:\R_+^*\times \R^3_x\times\R^3_v\to \R$ be defined by
$$ g_{n+1}(t,x,v) := f \big(t,\X_{n+1}(t,x,v)+t\V_{n+1}(t,x,v), \V_{n+1}(t,x,v) \big).$$ 
For any $ |\kappa| \leq N-r_{n+1}$ and all $(t,x,v) \in [2,\infty) \times \R^3_x \times \R^3_v$ such that $|x| \leq t$, we have
\begin{equation}\label{equa:induModscat}
  \langle x \rangle^{N_x-1-n} \, \langle v \rangle^{N_v} \big|\partial_{x,v}^\kappa \big( g_{n+1}(t,x,v)-f_\infty (x,v) \big) \big| \lesssim  t^{-n-1}\log^{S_{n+1}}(t).
  \end{equation}
\item The modified spatial average verifies an enhanced convergence estimate to the spatial average of $f_\infty$. For any $|\kappa| \leq N-1-r_{n+1}$, 
there exists $\mathbf{Q}_{p,\xi}^{\kappa,\beta} \in C^0 \cap L^\infty (\R^3_v) $ such that
\begin{align}
 \nonumber \color{white} \square \qquad \quad \color{black} \bigg| \int_{|x|<t } \partial_v^\kappa g_{n+1}(t,x,v) \dr x -\int_{\R^3_x} \partial_v^\kappa f_\infty (x,v) \dr x- \sum_{|\beta| \leq |\kappa|} &\sum_{p +|\xi| \leq n} \frac{\log^p(t)}{t^{n+1}} \mathbf{Q}_{p,\xi}^{\kappa,\beta} (v)\int_{\R^3_x} x^\xi \partial_v^\beta f_\infty (x,v) \dr x \bigg| \\
  \lesssim \frac{ \log^{S_{n+1}}(t)}{t^{n+2}}&, \qquad  \forall (t,v) \in [2,\infty) \times \R^3_v.\label{equa:induQn}
  \end{align}  
\end{enumerate}
\end{proposition}

\begin{remark}\label{Rkdomain}
The domain of integration of the spatial average in \eqref{equa:induQn} is restricted to $|x| \leq t$ since $x \mapsto \X_n(t,x,v)+t\V_n(t,x,v)$ may not be injective on $\R^3_x$ for $n \geq 2$. Note that $f(t,x+tv,v)$ and its derivatives enjoy strong decay estimates on the complementary region $|x| \geq t$ by Corollary \ref{Corboundg0}.
\end{remark}

\textbf{Structure of the proof of late-time asymptotics.}
The strategy of the proof relies on an induction argument. Fix an integer $n$ such that $r_{n+1} \leq N$. The base case $n=0$ has been treated in Sections \ref{SecModiscat}--\ref{SecRhoOrder2}. See in particular Propositions \ref{Prorhosimilar}, \ref{Proasympselfsiphi}, \ref{prop_mod_scatteringg_1} and \ref{Proconvsecondorder}. We then assume $n \geq 1$ and that the statement holds at any order $k \in \llbracket 0, n-1 \rrbracket$.

\begin{enumerate}
\item First, we derive the improved self-similar asymptotic expansion for $\rho(f)$ and its derivatives in Section \ref{Subsecrho}. For this, we make use of the convergence properties of $g_{n}$, $Q_{n}$, and their derivatives.
\item The first step allows us to show an improved polyhomogeneous asymptotic expansion of the force field along the $n^{\mathrm{th}}$-order modifications of the linear characteristics in Section \ref{SubsecforcefieldAxpan}.
\item In Section \ref{SubsecImproscatmod}, we define $\X_{n+1}$ and $\V_{n+1}$ from where the improved modified scattering statement hold. Introducing $\V_{n+1}$ is not required in this step but we do it to reduce the number of computations in the perspective of addressing part $(b)$ of Proposition \ref{Proinduction}.
\item Finally, we prove that the spatial average of $g_{n+1}$ satisfies a strong convergence estimate in Section \ref{Subsecspatialgnplus1}. With this estimate, we complete the induction.
\end{enumerate}

\subsection{Late-time asymptotics for the spatial density}\label{Subsecrho}

We begin the proof of the late-time asymptotics for the spatial density by writing the distribution $G^\beta f$ in terms of the modified profile $g_n$.

\subsubsection{Step 1: $G^\beta f$ in terms of derivatives of $g_n$}

The asymptotic self-similar polyhomogeneous expansion for $t^3\rho (G^\beta f)$ will be obtained through the convergence estimates satisfied by the weighted and non-weighted spatial averages of $\partial_v^\kappa g_n$. We then express $G^\beta f$ in terms of $g_n$ and its derivatives. For this, we recall that
$$g_n(t,x,v)= f \big( \X_n(t,x,v)+t \V_n(t,x,v),  \V_n(t,x,v) \big),$$
so we will have to invert the map $(x,v) \mapsto (\X_n+t\V_n,\V_n)$ for all $t \geq 2$. As mentioned earlier, we will have to restrict our analysis to a subdomain of $\R^3_x \times \R^3_v$.

To facilitate its study, we will write $(\X_n,\V_n)$ as the composition of three maps. Recall the definitions \eqref{eq:defCX }--\eqref{eq:defCV} of $(\X_n,\V_n)$. We first rewrite \eqref{eq:defCX }--\eqref{eq:defCV} as
\begin{align*}
\X_{n}(t,x,v) & = x+\mu \nabla_v \phi_{\infty}(v) \log(t)+\sum_{1 \leq q \leq n-1} \, \sum_{|\alpha|+ p \leq q} \frac{[ x+\mu \nabla_v \phi_{\infty}(v) \log(t) ]^\alpha \log^p(t)}{t^{q}} \overline{\mathbb{X}}_{q,\alpha,p}(v), \\
\V_{n}(t,x,v) &= v+ \frac{\mu}{t}\nabla_v \phi_{\infty}(v)+\sum_{1 \leq q \leq n-1} \, \sum_{|\alpha|+ p \leq q} \frac{[ x+\mu \nabla_v \phi_{\infty}(v) \log(t) ]^\alpha \log^p(t)}{t^{q+1}} \overline{\mathbb{V}}_{q,\alpha,p}(v),
\end{align*}
where $\overline{\mathbb{X}}_{1,\alpha,0}=-\nabla_v \partial_v^\alpha \phi_\infty$ for $|\alpha|=1$, and $\overline{\mathbb{X}}_{q,\alpha,p}, \; \overline{\mathbb{V}}_{q,\alpha,p} \in C^{N-r_{q+1}}\cap W^{N-r_{q+1}} (\R^3_v)$. It is the condition $|\alpha|+p \leq q$ which allows $(\X_n,\V_n)$ to be written in this alternative form.

This leads us to consider the decomposition
$$ (\X_n +t\V_n, \V_n )= \Upsilon_t \circ \Phi_{t,n} \circ \Xi_t,$$
where we recall
$$ \Upsilon_t (a,b):=(a+tb,b), \qquad \Xi_t (x,v):= \big( x+\mu \log(t) \nabla_v \phi_\infty(v),v \big),$$
and we define $\Phi_{t,n}$ by
$$ \Phi_{t,n}(y,w):= \mathrm{id}(y,w)+ \big(\mathcal{X}^n_t(y,w),\mathcal{V}_t^n(y,w)\big),$$
with
$$ \mathcal{X}^n_t (x,v) := \sum_{\substack{1 \leq q \leq n-1 \\ |\alpha|+ p \leq q }} \frac{x^\alpha \log^p(t)}{t^q}\overline{\mathbb{X}}_{q,\alpha,p}(v),  \qquad  \mathcal{V}^n_t (x,v) :=\frac{\mu}{t} \nabla_v \phi_\infty (v)+ \sum_{\substack{1 \leq q \leq n-1 \\ |\alpha|+ p \leq q }} \frac{x^\alpha \log^p(t)}{t^{q+1}}\overline{\mathbb{V}}_{q,\alpha,p}(v).$$
Since $\Upsilon_t$ and $\Xi_t$ are diffeomorphisms of $\R^3 \times \R^3$ and that their inverse can be easily computed, we will focus on $\Phi_{t,n}$. We note that for all $(t,x,v) \in [2,\infty)  \times \R^3_x \times \R^3_v$ such that $|x| \leq 2t$, we have
\begin{equation}\label{keva:estiXV}
  \big| \mathcal{X}_t^n \big|(t,x,v) \lesssim  \frac{\epsilon |x|+\log(t)}{t}+\frac{|x|^2}{t^2}, \qquad   \big| \mathcal{V}_t^n \big|(t,x,v) \lesssim \frac{1}{t}.
  \end{equation}
  The constant $\epsilon>0$ in the estimate for $\mathcal{X}_t^n$ comes from $\overline{\mathbb{X}}_{1,\alpha,0}=-\nabla_v \partial_v^\alpha \phi_\infty$. 
  
\begin{lemma}\label{Lemdiffeoordern}
There exists $T_n \! \geq \! 2$ and $0  < \! c_n \! <  1$ such that for all $t \geq T_n$, there exists an open set $\mathcal{U}_{t,n} \subset \{ |x|<2c_nt \} \times \R^3_v$ such that the restriction of $\Phi_{t,n}$ to $\mathcal{U}_{t,n}$ is a $C^{N-r_n}$-diffeomorphism from $\mathcal{U}_{t,n}$ to $\{ |x|<c_nt \} \times \R^3_v$.
\end{lemma}

\begin{proof} We introduce the family of sets
$$ \mathcal{A}_t^c := \big\{x \in \R^3_x \; : \; |x|<ct \big\} \times \R^3_v \subset \R^3_x \times \R^3_v ,$$ where $0< c <2$. Then, by \eqref{keva:estiXV} there holds
\begin{equation*}
\big\| \big( \mathcal{X}_t^n,\mathcal{V}_t^n \big) \big\|_{W^{N-r_n,\infty}(\mathcal{A}^{2c}_t )} \lesssim  \epsilon c+\log(t)t^{-1}+c^2,
\end{equation*}
where $  (\mathcal{X}_t^n,\mathcal{V}_t^n) \in C^{N-r_n}(\R^3_x \times \R^3_v,\R^3_x \times \R^3_v)$. Thus, for $c<1$ small enough and $t$ large enough so that 
\begin{equation}\label{eq:smallPsicondi}
\big\| \big( \mathcal{X}_t^n,\mathcal{V}_t^n \big) \big\|_{W^{N-r_n,\infty}(\mathcal{A}^{2c}_t )} \leq \frac{1}{2}c,
\end{equation}
the next properties hold by proceeding as in the proof of Lemma \ref{Lemdiffeo}:
\begin{itemize}
\item $\Phi_{t,n}$ is injective on $\mathcal{A}^{2c}_t$.
\item $\Phi_{t,n}$ is a local $C^{N-r_n}$-diffeomorhism on $\mathcal{A}^{2c}_t$. Note that $N-r_n \geq 1$.
\item $\mathcal{A}^c_t \subset \Phi_{t,n} (\mathcal{A}_t^{2c} )$ holds. We already know that $\mathcal{A}^c_t \cap \Phi_{t,n} (\mathcal{A}_t^{2c} )$ is open in the connected set $\mathcal{A}^c_t$. Let $(y_n,w_n)_{n \geq 0}$ be a sequence in $\mathcal{A}^c_t \cap \Phi_{t,n} (\mathcal{A}_t^{2c} )$ converging to $(y,w) \in \mathcal{A}_t^c$. Let $(z_n,v_n) \in \mathcal{A}^{2c}_t$ be such that $\Phi_{t,n}(z_n,v_n)=(y_n,w_n)$. By \eqref{eq:smallPsicondi}, we have
$$ \forall \, n \in \mathbb{N}, \qquad |(y_n-z_n,w_n-v_n)| =\big| \big( \mathcal{X}_t^n,\mathcal{V}_t^n \big) \big |(z_n,w_n)  \leq \frac{1}{2}c,$$
so that $(z_n,v_n) \in \mathcal{A}^{\frac{3}{2}c}_t$ and $(z_n,v_n)_{n \geq 0}$ is bounded. Consequently, there exists a subsequence that converges to $(z,v) \in \mathcal{A}^{2c}_t$, from which we get that $\mathcal{A}^c_t \cap \Phi_{t,n} (\mathcal{A}_t^{2c} )$ is closed in $\mathcal{A}^c_t$.
\end{itemize}
We then deduce the result with $\mathcal{U}_{t,n} := \Phi_{t,n}^{-1} \big(\mathcal{A}^c_t \cap \Phi_{t,n} ( \mathcal{A}^{2c}_t) \big)$. By \eqref{eq:smallPsicondi}, we have $\mathcal{U}_{t,n} \subset \mathcal{A}^{2c}_t$.
\end{proof}

\begin{remark}\label{Rkforipp}
In order to integrate by parts over $\{ |x| \leq c_nt \}$, we will require $\Phi^{-1}_{t,n}-\mathrm{id}$ to be bounded on $\{ |x| \leq c_nt \} \times \R^3_v$. Since the statement in Lemma \ref{Lemdiffeoordern} holds in fact for any $0<c'_n \leq c_n$, we can assume it is indeed the case by considering a smaller $c_n$ if necessary.
\end{remark}

Thus, for $(t,x,v) \in [T_n,\infty) \times \R^3_x \times \R^3_v$ in the domain of invertibility, we can write
$$ f(t,x,v)=g_n \big( t, \Xi_t^{-1} \circ \Phi^{-1}_{t,n} (x-tv,v) \big)  . $$
For our purposes, we need to derive an expansion for $\Xi_t^{-1} \circ \Phi^{-1}_{t,n}$ and its derivatives. We begin with the next intermediary result.
\begin{lemma}\label{LemexpaPhitn}
Let $t \geq T_n$ and $(\mathfrak{Z}_t^n,\mathfrak{W}_t^n)$ be the spatial and velocity components of $\Phi_{t,n}^{-1}-\mathrm{id}$. For any $\kappa=(\kappa_x,\kappa_v)$ such that $|\kappa| \leq N-1-r_n$, there holds
\begin{align*}
 \bigg|t^{|\kappa_x|} \partial_x^{\kappa_x} \partial_v^{\kappa_v} \mathfrak{Z}_t^n(z,v)- \sum_{1 \leq q \leq N+n-2-r_n-|\kappa| } \, \sum_{|\alpha|+p\leq q}  \frac{z^\alpha \log^p(t)}{t^{q}} \mathbf{Z}^{\kappa}_{\alpha,q,p}(v)  \bigg| & \lesssim \frac{\langle z ,\log(t)\rangle^{N+n-1-r_n-|\kappa|}}{t^{N+n-1-r_n-|\kappa|}}, \\
   \bigg| t^{|\kappa_x|} \partial_x^{\kappa_x} \partial_v^{\kappa_v} \mathfrak{W}_t^n(z,v)- \sum_{ q \leq N+n-2-r_n-|\kappa| } \, \sum_{|\alpha|+p\leq q}  \frac{z^\alpha \log^p(t)}{t^{q+1}}   \mathbf{W}^{\kappa}_{\alpha,q,p}(v)  \bigg| & \lesssim \frac{\langle  z ,\log(t) \rangle^{N+n-1-r_n-|\kappa|}}{t^{N+n-r_n-|\kappa|}},
   \end{align*}
for all $|z| <c_nt$, and $v \in\R^3_v$. Here, $\mathbf{Z}^{\kappa}_{\alpha,q,p}, \, \mathbf{W}^{\kappa}_{\alpha,q,p} \in C^{0}\cap L^{\infty}(\R^3_v)$. 
\end{lemma}

\begin{proof}
Let $|z|<c_nt$, $v \in \R^3_v$, and $(y,w)=\Phi^{-1}_{t,n}(z,v)$, so that
$$z=y+\mathcal{X}_t^n(y,w), \quad v=w+\mathcal{V}_t^n(y,w),$$ and  
$$y=z+\mathfrak{Z}_t^n(z,v), \quad w=v+\mathfrak{W}_t^n(z,v).$$ 
We then get from \eqref{keva:estiXV} that 
$$|z-y| \lesssim |y| t^{-1} +t^{-1}\log(t) , \qquad |v-w| \lesssim t^{-1}.$$ 
This allows us to show, according to the mean value theorem and $|y|<2c_nt$, that
$$ \big|z-\mathcal{X}_t^2(z,v)-y \big| \lesssim  \langle z ,\log(t) \rangle^2 \,t^{-2}, \qquad  \big|v-\mu t^{-1}\nabla_v \phi_\infty (v)-w \big| \lesssim  \langle z ,\log(t) \rangle \,t^{-2}.$$
Recall that $\mathbb{X}_{q,\alpha,p}$ and $\mathbb{V}_{q,\alpha,p}$ are of class $C^{N-r_{q+1}}$ for any $1 \leq q \leq n-1$. Then, iterating the above by using Taylor's theorem instead of the mean value theorem, we obtain the stated expansions for $\mathfrak{Z}_t^n$ and $\mathfrak{W}_t^n$. The cases with derivatives are handled as in the proof of Lemma \ref{LemPhiminus1}. The key identities are
\begin{equation*}
\big[\dr \Phi_{t,n} \big]^{-1}=\mathrm{id}+\sum_{k \geq 1} \,(-1)^k \big[\dr \big(\mathcal{X}_t^n,\mathcal{V}_t^n \big) \big]^k \quad \text{in $C^{N-1-r_n} \Big(\{|x|<c_nt\} \times\R^3_v,~\mathfrak{M}_6(\R) \Big)$},
\end{equation*}
where $[\dr(\mathcal{X}_t^n,\mathcal{V}_t^n)]$ is the Jacobian matrix of $(\mathcal{X}_t^n,\mathcal{V}_t^n)=\Phi_{t,n}-\mathrm{id}$. We also use that
$$ \dr \Phi_{t,n}^{-1} (z,v) = \big[ \dr \Phi_{t,n} \big]^{-1} \big( \Phi_{t,n}^{-1}(z,v) \big) .$$
The existence of the expansions for derivatives of $(\mathfrak{Z}_t^n,\mathfrak{W}_t^n)$ ensues from an induction, Leibniz rule, Taylor's theorem, and the expansions for $(\mathfrak{Z}_t^n,\mathfrak{W}_t^n)$. 
\end{proof}

Using $\Xi_t^{-1} (x,v)=(x-\mu \log(t) \nabla_v \phi_\infty (v),v)$, the chain rule, and Taylor's theorem, the previous Lemma \ref{Lemdiffeoordern} allows us to deduce the next corollary.

\begin{corollary}\label{CorexpanXVn}
Let $t \geq T_n$ and $(\mathfrak{X}_t^n,\mathfrak{V}_t^n)$ be the spatial and velocity components of $\Xi_t^{-1} \circ \Phi_{t,n}^{-1}-\mathrm{id}$. For any $\kappa=(\kappa_x,\kappa_v)$ such that $|\kappa| \leq N-1-r_n$, there holds
\begin{align*}
 \bigg|t^{|\kappa_x|} \partial_x^{\kappa_x} \partial_v^{\kappa_v} \mathfrak{X}_t^n(z,v)+\delta_{0,|\kappa_x|} &\mu \log(t)\nabla_v  \partial_v^{\kappa_v}  \phi_\infty(v)- \sum_{1 \leq q \leq N+n-2-r_n-|\kappa| } \, \sum_{|\alpha|+p\leq q}  \frac{z^\alpha \log^p(t)}{t^{q}} \mathbf{R}^\kappa_{\alpha,q,p}(v)  \bigg| \\
 & \qquad \qquad  \lesssim \frac{\langle z ,\log(t)\rangle^{N+n-1-r_n-|\kappa|}}{t^{N+n-1-r_n-|\kappa|}},
 \end{align*}
 and
 \begin{align*}
   \bigg| t^{|\kappa_x|} \partial_x^{\kappa_x} \partial_v^{\kappa_v} \mathfrak{V}_t^n(z,v)- \sum_{ q \leq N+n-2-r_n-|\kappa| } \, \sum_{|\alpha|+p\leq q}  \frac{z^\alpha \log^p(t)}{t^{q+1}}   \mathbf{S}^\kappa_{\alpha,q,p}(v)  \bigg| & \lesssim \frac{\langle  z ,\log(t) \rangle^{N+n-1-r_n-|\kappa|}}{t^{N+n-r_n-|\kappa|}},
   \end{align*}
for all $|z| <c_nt$ and $v \in\R^3_v$. Here, $\mathbf{R}^\kappa_{\alpha,q,p}, \, \mathbf{S}^\kappa_{\alpha,q,p} \in C^{0}\cap L^{\infty}(\R^3_v)$.
\end{corollary}
\begin{remark}\label{lastRk}
Note also that $N-1-r_n \geq n$ and Lemma \ref{LemForexp} gives 
$$\forall \, q \in \llbracket 1, n-1 \rrbracket, \qquad \mathbf{R}^0_{\alpha,q,p}\in C^{n}\cap W^{n,\infty}(\R^3_v).$$
More generally, if $|\kappa| \leq N-r_{n+1}$ and $q \in \llbracket 1, n \rrbracket$, we have $\mathbf{R}^\kappa_{\alpha,q,p}, \, \mathbf{S}^\kappa_{\alpha,q-1,p} \in C^{n+1-q}\cap W^{n+1-q,\infty}(\R^3_v)$. These properties will be used in order to apply Taylor's theorem below.
\end{remark}

\subsubsection{Step 2: The high order change of variables}

We now know that if $t \geq T_n$ and $|x-tv| < c_nt$, then
\begin{equation}\label{ident_distrib_terms_gn}
  f(t,x,v)=g_n \big( t,x-tv+\mathfrak{X}_t^n(x-tv,v),v+\mathfrak{V}_t^n(x-tv,v) \big),
\end{equation}
and we have determined the asymptotic behaviour of $(\mathfrak{X}_t^n,\mathfrak{V}_t^n)$. So, we wish to perform the change of variables $y=\Y_{t,z}^n(v)$, where the map $\Y_{t,x}^n \colon   \mathcal{D}_{t,x}^{c_n}   \to  \R^3$ defined on $$\mathcal{D}_{t,x}^{c_n}:= \big\{ v \in \R^3_v \; : \; |x-tv|<c_nt \big\} $$ is given by
$$\Y_{t,x}^n  (v) :=x-tv+\mathfrak{X}_t^n(x-tv,v).$$
This change of variables would allow us to obtain an expansion in terms of the spatial averages of $g_n$. 

The analysis of $\Y_{t,x}^n$ can be reduced to the study of
\begin{equation*}
\mathfrak{Z}_{t,x}(v):=v-t^{-1}\mathfrak{X}_t^n(x-tv,v), \quad \qquad \det \dr \Y_{t,x}^n(v) = -t^3 \det \dr \mathfrak{Z}_{t,x}(v).
\end{equation*}
Note that 
\begin{equation}\label{linkYZ}
\Y_{t,x}^n(v)=y  \qquad \textrm{if and only if} \qquad \mathfrak{Z}_{t,x}(v)=\frac{x-y}{t}.
\end{equation}

\begin{lemma}\label{LemExpZn}
Let $t\geq T_n$ and $x\in \R^3_x$. Then, $\mathfrak{Z}_{t,x}$ is a $C^{1}$-diffeomorphism from $\mathcal{D}_{t,x}^{c_n}$ to $\mathfrak{Z}_{t,x}(\mathcal{D}^{c_n}_{t,x} )$. Moreover, for all $v \in \mathfrak{Z}_{t,x}(\mathcal{D}^{c_n}_{t,x} )$, we have
\begin{equation*}
\bigg| \mathfrak{Z}^{-1}_{t,x}(v)-v+\mu \frac{\log(t)}{t} \nabla_v \phi_\infty (v)-\sum_{2 \leq q \leq n} \; \sum_{|\alpha|+p \leq q} \frac{(x-tv)^\alpha \log^p(t)}{t^q} \Psi_{q,\alpha,p}(v)\bigg| \lesssim \frac{\langle x-tv,\log(t) \rangle^{n+1}}{t^{n+1}},
 \end{equation*}
where $ \Psi_{q,\alpha,p} \in C^{n+2-q}\cap W^{n+2-q,\infty} \big( \mathfrak{Z}_{t,x}(\mathcal{D}_{t,x}^{c_n} ) \big)$. 
\end{lemma}
\begin{proof}
Note from Corollary \ref{CorexpanXVn} that, for any $|\kappa_v| \leq 1$, we have
\begin{equation}\label{equatio:justpourla}
 \forall \, v \in \mathcal{D}_{t,x}^{c_n}, \qquad t^{-1}\big| \partial_v^{\kappa_v} \big[ \mathfrak{X}_t^n(x-tv,v) \big] \big| \lesssim t^{-1}\log(t).
 \end{equation}
Consequently, if $T_n$ is chosen large enough, we get as in the proof of Lemma \ref{Lemdiffeo} that
\begin{itemize}
\item $\mathfrak{Z}_{t,x}$ is injective on $\mathcal{D}_{t,x}^{c_n}$,
\item $\mathfrak{Z}_{t,x}$ is a local $C^{1}$-diffeomorphism on $\mathcal{D}_{t,x}^{c_n}$,
\end{itemize}
These properties imply the first part of the result. 

We now focus on the expansion, so we consider $w \in \mathcal{D}_{t,x}^{c_n}$ and $v = \mathfrak{Z}_{t,x}(w)$. By definition $v=w-t^{-1}\mathfrak{X}_t^{n} (x-tw,w)$, so that $|v-w| \lesssim  t^{-1}\log(t)$. By an induction, Taylor's theorem, and the expansion for $\mathfrak{X}_t^n$ given by Corollary \ref{CorexpanXVn}, we obtain the stated expansion for $\mathfrak{Z}^{-1}_{t,x}$. We have used Remark \ref{lastRk} in order to apply Taylor's theorem at the required orders.

Finally, the regularity of $ \Psi_{q,\alpha,p}$ is obtained by Remark \ref{lastRk}. Indeed, one can check that $\Psi_{q,\alpha,p}$ is a linear combination of product of quantities of the form $\partial_v^{\xi}\mathbf{R}^0_{q',\alpha',p'}$, where $q' \geq 1$ and $|\xi|+q'\leq q-1$, or $\partial_v^\kappa \phi_\infty$ with $1 \leq |\kappa| \leq q$.
\end{proof}

We are now able to derive expansions for the two quantities related to $\Y_{t,x}^{n,-1}$ that we are interested in.

\begin{corollary}\label{CorYordern}
Let $t\geq T_n$ and $x \! \in \R^3_x$. The map $\Y_{t,x}^n \! : v \mapsto x-tv+\mathfrak{X}_t^n(x-tv,v)$ is a $C^{1}$-diffeomorphism from $\mathcal{D}_{t,x}^{c_n}$ to $\Y_{t,x}^n (\mathcal{D}_{t,x}^{c_n})$, which verifies 
\begin{equation}\label{recalltruc}
\Big\{ y \in \R^3 \, : \,  |y|<\frac{1}{2}c_nt \Big\} \subset \Y_{t,x}^n (\mathcal{D}_{t,x}^{c_n}) \subset \Big\{ y \in \R^3 \, : \,  |y|< 2c_nt \Big\}.
\end{equation}
 Moreover, the inverse $\Y^{n,-1}_{t,x}$ satisfies that for all $y \in \Y_{t,x}^n (\mathcal{D}_{t,x}^{c_n})$,
\begin{equation*}
\bigg|\Y^{n,-1}_{t,x}(y) -   \frac{x}{t}- \sum_{1 \leq q \leq n} \, \sum_{ |\alpha|+p \leq q}\frac{y^\alpha \log^{p}(t)}{t^q} \mathbf{D}_{q,\alpha,p}\Big( \frac{x}{t}  \Big) \bigg| \lesssim  \frac{ \langle y, \log (t) \rangle^{n+1}}{ t^{n+1}} ,
\end{equation*}
where $\mathbf{D}_{q,\alpha,p} \in C^0\cap L^\infty(\mathcal{D}_{t,x}^{c_n})$, and
\begin{equation*}
\bigg|\Y^{n,-1}_{t,x}(y)+\mathfrak{V}_t^n \big( x-t\Y^{n,-1}_{t,x}(y) ,\Y^{n,-1}_{t,x}(y)\big)  -   \frac{x}{t}- \sum_{1 \leq q \leq n} \, \sum_{ |\alpha|+p \leq q}\frac{y^\alpha \log^{p}(t)}{t^q} \overline{\mathbf{D}}_{q,\alpha,p}\Big( \frac{x}{t}  \Big) \bigg| \lesssim  \frac{ \langle y, \log (t) \rangle^{n+1}}{ t^{n+1}} ,
\end{equation*}
where $\overline{\mathbf{D}}_{q,\alpha,p} \in C^0\cap L^\infty(\mathcal{D}_{t,x}^{c_n})$. Moreover, the Jacobian determinant satisfies
\begin{equation*}
\bigg| t^{3} \big|\det \mathrm{d} \Y_{t,x}^{n,-1} \big|(y) -1-   \sum_{1 \leq q \leq n} \, \sum_{ |\alpha|+p \leq q}\frac{y^\alpha \log^{p}(t)}{t^{q}} \mathbf{G}_{q,\alpha,p}\Big( \frac{x}{t}  \Big)\bigg| \lesssim \frac{ \langle y,  \log(t) \rangle^{n+1}}{ t^{n+1}} ,
 \end{equation*}
 where $\mathbf{G}_{q,\alpha,p} \in C^0\cap L^\infty(\mathcal{D}_{t,x}^{c_n})$.
 \end{corollary}
 
 \begin{proof}
 For the first part of the statement, we use \eqref{linkYZ} and the previous Lemma \ref{LemExpZn}. The inclusions \eqref{recalltruc} are implied by \eqref{equatio:justpourla}. For the first expansion, we exploit the one satisfied by $\mathfrak{Z}_{t,x}^{-1}$ in Lemma \ref{LemExpZn} and we use \eqref{linkYZ} as well as Taylor's theorem. For the second one, we use additionally the expansion verified by $\mathfrak{V}_t^n$ in Corollary \ref{CorexpanXVn} and $N-1-r_n \geq n$. Finally, note that \eqref{equatio:justpourla} implies
\begin{equation*}
 \big[\dr \mathfrak{Z}_{t,x} \big]^{-1}(v) = \mathrm{id} + \sum_{k \geq 1} t^{-k} \big[ \dr \big( \mathfrak{X}_t^n(x-tv,v) \big) \big]^k, 
 \end{equation*}
where $[ \dr ( \mathfrak{X}_t^n(x-tv,v) ) ]_{ij}=-t\partial_{x^j} \mathfrak{X}_t^{n,i}(x-tv,v)+\partial_{v^j} \mathfrak{X}_t^{n,i}(x-tv,v)$. We then obtain the expansion for the Jacobian determinant using the multi-linearity of the determinant, Corollary \ref{CorexpanXVn}, Taylor's theorem and
$$ t^3 \det \dr \Y_{t,x}^{n,-1}(y)= - \det  \big[\dr \mathfrak{Z}_{t,x} \big]^{-1} \big( \Y_{t,x}^{n,-1}(y)\big) .$$
 \end{proof}

\subsubsection{Step 3: Expansion of $t^3 \rho (G^\beta f)$ in terms of the modified spatial averages} In order to lighten the presentation, we introduce the following terminology.

\begin{definition}
Let $h_1, \; h_2 : [T_n,\infty) \times \R^3_x \to \R^d$ with $d \in \mathbb{N}^*$. We say that $h_1$ and $h_2$ are \emph{$n$--equivalent} if there exists $C>0$ such that
\begin{equation}\label{equat:kequiv}
 \forall \, (t,x) \in [T_n,\infty) \times \R_x^3, \qquad  | h_1-h_2 |(t,x) \leq C \frac{\log^{S_{n+1}}(t)}{\langle t+|x| \rangle^3 \, t^{n-2}}.
\end{equation}
\end{definition}

In the rest of this subsection $\mathbf{K}$ will always denote a general function belonging to $C^0 \cap L^\infty (\R^3_v)$, which may be different from line to line.

Our goal consists in proving the next result by exploiting the previous two steps.

\begin{proposition}\label{Profromg1togn}
Let $(t,x) \in [T_n,\infty) \times \R^3_x$ and $|\beta| \leq N-r_{n+1}$. Then, $t^3 \rho (G^\beta f)$ is $n$--equivalent to a linear combination of terms of the form
\begin{equation}\label{equat:forRk2}
 \frac{\log^p(t)}{t^{q}} \mathbf{K} \Big(\frac{x}{t} \Big) \!\int_{|z| < t}  z^\alpha \big[ \partial_{v}^{\gamma_v} g_n \big] \Big( t,z,\frac{x}{t} \Big) \dr z  , 
 \end{equation}
with $|\gamma_v| \leq N-1-r_n$, and $p+|\alpha| \leq q \leq n$.
\end{proposition}

A preparatory result is required to treat the cases where $|\beta| \geq 1$.

 \begin{lemma}\label{LemZbetafg1}
Let $|\beta| \leq N-r_n$ and $(t,x,v) \in [T_n,\infty) \times \R^3_x \times \R^3_v$ be such that $|x-tv|<c_nt$. Then, $G^\beta f(t,x,v)$ can be written as a linear combination of terms of the form
$$ \prod_{1 \leq i \leq p_x} \! \big[\partial_v^{\kappa_i} \mathfrak{X}^{n,k_i}_t\big](x-tv,v) \! \prod_{1 \leq j \leq p_v} \! \big[\partial_v^{\gamma_j} \mathfrak{V}^{n,\ell_j}_t\big](x-tv,v) \big[ \partial_x^{\alpha_x} \partial_v^{\alpha_v} g_n \big] \big(t, x-tv+\mathfrak{X}_t^n(x-tv,v),v+\mathfrak{V}_t^n(x-tv,v) \big) ,$$
where $k_i, \, \ell_j \in \llbracket 1,3 \rrbracket$, $\mathfrak{X}_t^{n,k_i}$ and $\mathfrak{V}_t^{n,\ell_j}$ are the $k_i$ and $\ell_j$ cartesian components of $\mathfrak{X}_t^n$ and $\mathfrak{V}_t^n$,
$$|\alpha_x|+|\alpha_v| \leq |\beta|, \qquad 0 \leq p_x = |\alpha_x|, \qquad 0 \leq p_v \leq |\alpha_v| , \qquad  \sum_{1 \leq i \leq p_x} \sum_{1 \leq j \leq p_v} |\kappa_i|+|\gamma_j| \leq |\beta|.$$
 \end{lemma}
 \begin{remark}\label{RkZbetafg1}
 Recall from Section \ref{Subsec72} that if $n=1$, then $(\mathfrak{X}_t^1,\mathfrak{V}_t^1)=(\mathfrak{X}_t,\mathfrak{V}_t)$ does not depend on the spatial variable. So, in this case the result holds for all $(t,x,v) \in [2,\infty) \times \R^3_x \times \R^3_v$.
 \end{remark}

\begin{proof}
Let $h: \R_+^* \times \R^3_x \times \R^3_v \to \R$ be a distribution function. Using, the relations
\begin{align*}
\partial_{v^i} \big[ h  \big(t, x-tv+\mathfrak{X}_t^n(x-tv,v),v+\mathfrak{V}_t^n(x-tv,v) \big) \big]  = -\big[t \partial_{x^i} h \big] \big(t, x-tv+\mathfrak{X}_t^n(x-tv,v),v+\mathfrak{V}_t^n(x-tv,v) \big)& \\
+ \, \big[ \partial_{v^i} h \big] \big(t, x-tv+\mathfrak{X}_t^n(x-tv,v),v+\mathfrak{V}_t^n(x-tv,v) \big) & \\
 + \, \partial_{v^i} \big[ \mathfrak{X}_t^n (x-tv,v)\big] \cdot \big[ \nabla_x h \big]\big(t, x-tv+\mathfrak{X}_t^n(x-tv,v),v+\mathfrak{V}_t^n(x-tv,v) \big)& \\
 + \, \partial_{v^i} \big[ \mathfrak{V}_t^n (x-tv,v)\big] \cdot \big[ \nabla_v h \big] \big(t, x-tv+\mathfrak{X}_t^n(x-tv,v),v+\mathfrak{V}_t^n(x-tv,v) \big)&
   \end{align*}
and 
$$\partial_{v^i} [ \mathfrak{X}_t^n (x-tv,v)]=-t[\partial_{x^i}  \mathfrak{X}_t^n] (x-tv,v)+[\partial_{v^i}  \mathfrak{X}_t^n] (x-tv,v),$$ 
we obtain in view of $G=t\nabla_x+\nabla_v$ the identity
\begin{align*}
 G_i\big[ h \big(t, x-tv+\mathfrak{X}_t^n(x-tv,v),v+\mathfrak{V}_t^n(x-tv,v) \big) \big]  = \big[ \partial_{v^i}h \big] \big(t, x-tv+\mathfrak{X}_t^n(x-tv,v),v+\mathfrak{V}_t^n(x-tv,v) \big) &  \\
 + \, \big[ \partial_{v^i} \mathfrak{X}_t^n \big](x-tv,v) \cdot \big[ \nabla_x h \big]\big(t, x-tv+\mathfrak{X}_t^n(x-tv,v),v+\mathfrak{V}_t^n(x-tv,v) \big)& \\
 + \, \big[ \partial_{v^i} \mathfrak{V}_t^n \big](x-tv,v) \cdot \big[ \nabla_v h \big] \big(t, x-tv+\mathfrak{X}_t^n(x-tv,v),v+\mathfrak{V}_t^n(x-tv,v) \big)&.
   \end{align*}
The result is proven by iterating this relation and recalling \eqref{ident_distrib_terms_gn}. 
\end{proof} 

We now exploit the first step.

\begin{proposition}\label{Profromg1tognbis}
Let $(t,x) \in [T_n,\infty) \times \R^3_x$ and $|\beta| \leq N-r_{n+1}$. Then, $t^3 \rho(G^\beta f)$ is $n$--equivalent to a linear combination of terms 
\begin{equation*}
 \frac{\log^{p_1+p_2}(t)}{t^{q-3}} \int_{|x-tv| < c_nt} \mathbf{L}(v) \big(x-tv+\mathfrak{X}_t^n(x-tv,v) \big)^\alpha \big[ \partial_x^{\kappa_x} \partial_{v}^{\kappa_v} g_n \big] \! \big(t, x-tv+\mathfrak{X}_t^n(x-tv,v),v+\mathfrak{V}_t^n(x-tv,v) \big) \dr v  , 
\end{equation*}
with $|\kappa_x| + |\kappa_v|\leq N-r_{n+1}$, $p_1+|\alpha| \leq q \leq n$, $p_2 \leq |\kappa_x|$ and $\mathbf{L} \in C^{n+1-q} \cap W^{n+1-q,\infty}(\R^3_v)$.
\end{proposition}
\begin{proof}
We remark first that by support considerations, Proposition \ref{prop_point_bound_deriv_distrib}, and Remark \ref{rklineardecay}, we have
\begin{align*}
&t^3 \bigg|  \rho(G^\beta f) (t,x) - \int_{|x-tv|<c_nt} G^\beta f(t,x,v) \dr v \bigg| \\
& \qquad \qquad \qquad  \lesssim \! \int_{|x-tv| \geq c_nt} \frac{t^3 \, \dr v}{t^{n+2} \langle x-tv \rangle^7 \langle v \rangle^7} \sup_{(y,w) \in \R^3 \times \R^3} \langle y-tw \rangle^{n+9} \langle w \rangle^7 \big| G^{\beta} f \big|(t,y,w)\\  
& \qquad \qquad \qquad \lesssim \frac{\log^{n+9+|\beta|}(t)}{\langle t+|x| \rangle^3 \, t^{n-1}}  \lesssim \frac{1}{\langle t+|x| \rangle^3 \, t^{n-2}}.
 \end{align*}
Since $N+n-2-r_n-N+r_{n+1}=2n-1 \geq n$, the previous Lemma \ref{LemZbetafg1} together with the expansions in Corollary \ref{CorexpanXVn} and Remark \ref{lastRk} imply that 
$$t^3\int_{|x-tv|<c_nt} G^\beta f(t,x,v) \dr v$$
can be written as a linear combination of terms of the form 
\begin{equation}\label{eq:forhere000}
\color{white} \square \qquad \color{black} \frac{t^3 \log^{p_1'+p_2'}(t)}{t^{q'}} \int_{|x-tv| < c_nt} \mathbf{L}(v) (x-tv)^{\alpha'} \big[ \partial_x^{\kappa_x} \partial_{v}^{\kappa_v} g_n \big]\big(t, x-tv+\mathfrak{X}_t^n(x-tv,v),v+\mathfrak{V}_t^n(x-tv,v) \big) \dr v  , 
 \end{equation}
with $|\kappa_x| + |\kappa_v|\leq N-r_{n+1}$, $p_1'+|\alpha'| \leq q' \leq n$, $p_2' \leq |\kappa_x|$ and $\mathbf{L} \in C^{n+1-q'} \cap W^{n+1-q',\infty}(\R^3_v)$, and terms bounded by
$$ \mathfrak{Q}^{\kappa_x,\kappa_v}_{t,x}:= \log^{|\kappa_x|}(t)t^3\int_{|x-tv|<c_nt} \frac{\langle x-tv , \log(t) \rangle^{n+1}}{t^{n+1}}  \big| \partial_x^{\kappa_x} \partial_{v}^{\kappa_v} g_n \big|\big(t, x-tv+\mathfrak{X}_t^n(x-tv,v),v+\mathfrak{V}_t^n(x-tv,v) \big) \dr v ,$$
with $|\kappa_x|+|\kappa_v| \leq N-r_{n+1}$. 

As $|\mathfrak{X}_t^n (z,w) |\lesssim \log(t)$ and $|\mathfrak{V}_t^n (z,w) \lesssim 1$, we get by \eqref{equa:induModscat} and $N_x \geq 2n+8$ that
$$ \langle x-tv \rangle^{n+8} \langle v \rangle^7 \big| \partial_x^{\kappa_x} \partial_{v}^{\kappa_v} g_n \big|\big(t, x-tv+\mathfrak{X}_t^n(x-tv,v),v+\mathfrak{V}_t^n(x-tv,v) \big) \lesssim \log^{n+8}(t).$$
By Remark \ref{rklineardecay}, we then have
\begin{equation}\label{boundmathfrakQ}
| \mathfrak{Q}^{\kappa_x,\kappa_v}_{t,x} |\lesssim \frac{\log^{N-r_{n+1}+n+8}(t)}{\langle t+|x| \rangle^3 \, t^{n-2}} \lesssim \frac{\log^{S_{n+1}}(t)}{\langle t+|x| \rangle^3 \, t^{n-2}}.
 \end{equation}
We now work with the terms of the form \eqref{eq:forhere000}. If $|\alpha'|=0$ they are already of the expected form. Otherwise, $(x-tv)^{\alpha'}$ can be written as a linear combination of terms
$$ \big(x-tv+\mathfrak{X}_t^n(x-tv,v) \big)^{\alpha_1}  \big[ \mathfrak{X}_t^n(x-tv,v) \big]^{\alpha_2}, $$
with $\alpha_1+\alpha_2=\alpha'$. Using the expansion for $\mathfrak{X}_t^n$ in Corollary \ref{CorexpanXVn} and \eqref{boundmathfrakQ} (in order to bound the strongly decaying terms), we obtain that \eqref{eq:forhere000} is $n$--equivalent to a linear combination of terms 
\begin{equation*}
 \frac{t^3 \log^{p_1+p_2'}(t)}{t^{q}} \int_{|x-tv| < c_nt} \mathbf{M}(v)  \big(x-tv+\mathfrak{X}_t^n \big)^{\alpha_1} (x-tv)^{\alpha_0} \big[ \partial_x^{\kappa_x} \partial_{v}^{\kappa_v} g_n \big]\big(t, x-tv+\mathfrak{X}_t^n,v+\mathfrak{V}_t^n \big) \dr v  , 
 \end{equation*}
 with $p_1+|\alpha_1|+|\alpha_0| \leq q \leq n$, $|\alpha_0| < |\alpha'|$ and $\mathbf{M} \in C^{n+1-q} \cap W^{n+1-q,\infty}(\R^3_v)$. Note that we have dropped the dependence in $(x-tv,v)$ of $(\mathfrak{X}_t^n,\mathfrak{V}_t^n)$. One can iterate the process to obtain the result since $|\alpha_0| < |\alpha'|$.
\end{proof}

Next, we perform the high order change of variables which generalises $y=x-tv$.

\begin{proposition}\label{ProCDV}
Let $(t,x) \in [T_n,\infty) \times \R^3_x$ and $|\beta| \leq N-r_{n+1}$. Then, $t^3 \rho(G^\beta f)$ is $n$--equivalent to a linear combination of terms of the form
\begin{equation*}
 \frac{ \log^{p_1+p_2}(t)}{t^{q}} \mathbf{K} \Big( \frac{x}{t} \Big) \int_{ y \in \Y_{t,x}^n(\mathcal{D}_{t,x}^{c_n})} y^\alpha \big[ \partial_x^{\gamma_x} \partial_{v}^{\gamma_v} g_n \big]\Big(t, y, \frac{x}{t} \Big) \dr y  , 
 \end{equation*}
with $|\gamma_x| + |\gamma_v| \leq N-1-r_{n}$, $p_1+|\alpha| \leq q \leq n$, and $p_2 \leq |\gamma_x|$.
\end{proposition}

\begin{proof}
We fix $(t,x) \in [T_n,\infty) \times \R^3_x$. Apply first Proposition \ref{Profromg1tognbis} and perform the change of variables $y=\Y_{t,x}^n(v)=x-tv+\mathfrak{X}_t^n(x-tv,v)$ allowed by Corollary \ref{CorYordern}. We then have that 
$t^3 \rho(G^\beta f)$ is $n$--equivalent to a linear combination of terms of the form
\begin{align*}
\frac{ \log^{p_1'+p_2}(t)}{t^{q'}}  \int_{y \in \mathcal{D}_{t,x}} &\mathbf{K}\big( \Y^{n,-1}_{t,x}(y) \big) y^{\alpha'} \\
& \times  \big[ \partial_x^{\kappa_x} \partial_{v}^{\kappa_v} g_n \big]\Big(t, y,\Y^{n,-1}_{t,x}(y)+\mathfrak{V}_t^n \big( x-t\Y^{n,-1}_{t,x}(y),\Y^{n,-1}_{t,x}(y)\big) \Big) t^3\big| \det \dr \Y_{t,x}^{n,-1} \big|(y) \dr y , 
 \end{align*}
 with $|\kappa_x|+|\kappa_v| \leq \! N-r_{n+1}$, $p_1'+|\alpha'| \leq q' \leq n$, $p_2 \leq |\kappa_x|$, and the domain of integration is $\mathcal{D}_{t,x} \! := \Y_{t,x}^n (\mathcal{D}_{t,x}^{c_n}) $. 
 
Recall now the expansions satisfied by $\Y^{n,-1}_{t,x}(y)$ and $\Y^{n,-1}_{t,x}(y)+\mathfrak{V}_t^n \big( x-t\Y^{n,-1}_{t,x}(y),\Y^{n,-1}_{t,x}(y)\big)$ in Corollary \ref{CorYordern}. By Taylor's theorem applied at order $n-q'$, we get that the last quantity is the sum of terms 
 $$  \frac{ \log^{p_1+p_2}(t)}{t^{q}} \mathbf{K} \Big( \frac{x}{t} \Big) \int_{y \in \mathcal{D}_{t,x}}  y^\alpha \big[ \partial_x^{\gamma_x} \partial_{v}^{\gamma_v} g_n \big]\Big(t, y, \frac{x}{t} \Big) \dr y  ,$$
 where $p_1+|\alpha| \leq q \leq n$ and $|\gamma_x|+|\gamma_v| \leq N-r_{n+1}+q \leq N-1-r_n$, and also terms bounded by
 \begin{equation*}
\sup_{0 \leq \tau \leq 1} \frac{ \log^{p_1'+p_2}(t)}{t^{q'}} \int_{y  \in \mathcal{D}_{t,x}} \frac{\langle y, \log(t) \rangle^{n-q'+1}}{t^{n-q'+1}} |y|^{\alpha'} \big| \partial_x^{\xi_x} \partial_{v}^{\xi_v} g_n \big|\Big(t, y,\frac{x}{t}+z_\tau(y) \Big) \dr y  , 
 \end{equation*}
 where $|\xi_x|+|\xi_v| \leq N-r_n$ and
 $$ z_\tau (y) := \tau \Big( \Y^{n,-1}_{t,x}(y)+\mathfrak{V}_t^n \Big( x-t\Y^{n,-1}_{t,x}(y),\Y^{n,-1}_{t,x}(y)\Big)-\frac{x}{t} \Big) , \qquad y \in \mathcal{D}_{t,x}, \; \tau \in [0,1] .$$
To conclude the proof, it remains us to prove that for any $|\xi_x|+|\xi_v| \leq N-r_n$ and all $ \tau \in [0,1]$, we have
\begin{equation}\label{eqeqeq:1}
  \int_{y \in \mathcal{D}_{t,x}}  \langle y \rangle^{n-q'+1+|\alpha'|}  \big| \partial_x^{\xi_x} \partial_{v}^{\xi_v} g_n \big| \Big(t, y,\frac{x}{t}+z_\tau(y) \Big) \dr y \lesssim \frac{1}{\langle x/t \rangle^3} \lesssim \frac{ t^3}{\langle t+|x| \rangle^3}.
  \end{equation}
Recall from Corollary \ref{CorYordern} that $\mathcal{D}_{t,x} \subset \{ |y| < 2c_n t \}$. Hence, using $c_n<1$ and the expansion for $z_\tau (y)$, still given by Corollary \ref{CorYordern}, we get
$$ \forall \, (\tau,y) \in [0,1] \times \mathcal{D}_{t,x}, \qquad \big| z_\tau (y) \big| \lesssim c_n+t^{-1} \log(t) \leq 2.$$
We then obtain \eqref{eqeqeq:1} by using
$$ \langle y \rangle^{n-q'+1+|\alpha'|}  \big| \partial_x^{\xi_x} \partial_{v}^{\xi_v} g_n \big| (t, y,v) \lesssim \langle y \rangle^{-4} \langle v \rangle^{-3},$$
 which holds for all $t \geq 2$, $|y| \leq t$ and $v\in \R^3_v$ according to \eqref{equa:induModscat} and $N_x-n \geq n+5$.
\end{proof}

 We are now ready to obtain Proposition \ref{Profromg1togn}. 

\begin{proof}[Proof of Proposition \ref{Profromg1togn}]
The proposition is a direct consequence of the previous result as well as the strong spatial decay estimates satisfied by $g_n$. We remark first that \eqref{equa:induModscat} implies that for all $t \geq T_n$
\begin{equation}\label{kevatalenn:0}
 \forall \, \frac{1}{2}c_nt \leq |z| \leq t, \; \forall \, v \in \R^3_v, \qquad \langle z \rangle^{|\alpha|} \langle v \rangle^3 \big| \partial_x^{\gamma_x} \partial_v^{\gamma_v} g_n \big|(t,z,v) \lesssim t^{-N_x+n+|\alpha|} \lesssim t^{-n-2+|\alpha|},
 \end{equation}
for any $|\gamma_x|+|\gamma_v| \leq N-1-r_n$, and $|\alpha| \leq n$. Moreover, we have $\{|y| < \frac{1}{2}c_nt \} \subset \Y_{t,x}^n(\mathcal{D}_{t,x}^{c_n})$ according to Corollary \ref{CorYordern}. By Proposition \ref{ProCDV} and the estimate \eqref{kevatalenn:0}, the normalised spatial average $t^3 \rho (G^\beta f)$ is then $n$--equivalent to a linear combination of terms of the form
\begin{equation*}
 \frac{ \log^{p_1+p_2}(t)}{t^{q}} \mathbf{K} \Big( \frac{x}{t} \Big) \int_{|y| < t} y^\alpha \big[ \partial_x^{\gamma_x} \partial_{v}^{\gamma_v} g_n \big]\Big(t, y, \frac{x}{t} \Big) \dr y  , 
 \end{equation*}
with $|\gamma_x| + |\gamma_v|\leq N-r_{n+1}$, $p_1+|\alpha| \leq q \leq n$, and $p_2 \leq |\gamma_x|$. We conclude the proof by performing integration by parts and by bounding the strongly decaying terms (which includes the boundary terms) through \eqref{kevatalenn:0}.
\end{proof}

\subsubsection{Step 4: Exploiting the convergence of the modified spatial averages}

We are finally able to conclude the analysis of $t^3\rho(G^\beta f)$.

\begin{proposition}\label{Proselfsimrho}
Let $|\beta| \leq N-r_{n+1}$. Then, for all $(t,x) \in [2,\infty) \times \R^3_x$, we have
$$ \bigg| t^3 \rho (G^\beta f)(t,x) - \sum_{p \leq q \leq n} \frac{\log^p(t)}{t^q} \partial_v^\beta \mathbf{F}_{p,q} \Big( \frac{x}{t} \Big) \bigg| \lesssim \frac{\log^{S_n+1}(t)}{\langle t+|x| \rangle^3 \, t^{n-2}} ,$$ 
where $\mathbf{F}_{p,q} \in C^{N-r_{q+1}}(\R^3_v)$, and $\| \langle v \rangle^3 \partial_v^\beta \mathbf{F}_{p,q} \|_{L^\infty} <\infty$.
\end{proposition} 

\begin{proof}
In view of the spatial (respectively velocity) decay satisfied by $\rho ( G^\beta f)$ (respectively $\partial_v^\beta \mathbf{F}_{p,q}$), and the compactness of $[2,T_n]$, it suffices to establish the result for $t \geq T_n$. We remark further that the existence of the asymptotic self-similar polyhomogeneous expansion (recall Definition \ref{Defselsimexp}) implies by Lemma \ref{LemForexpbis} that
$$  \big[t^3 \rho (G^\beta f) \big]_{p,q}=\partial_v^\beta  \big[t^3 \rho (f) \big]_{p,q} \qquad \textrm{and} \qquad   \mathbf{F}_{p,q} \in C^{N-r_{n+1}}(\R^3_v).$$
Then, as such expansions are unique, we get from the induction hypothesis at order $q \in \llbracket 0,n-1 \rrbracket$ that $\mathbf{F}_{p,q} \in C^{N-r_{q+1}}(\R^3_v)$.

We then fix $(t,x) \in [T_n,\infty) \times \R^3_x$. We first apply Proposition \ref{Profromg1togn} to reduce the analysis to the study of terms of the form 
\begin{equation}
\color{white} \square \qquad \quad \color{black} \frac{\log^p(t)}{t^{q}} \mathbf{K} \Big(\frac{x}{t} \Big) \!\int_{|z| < t}  z^\alpha \big[ \partial_{v}^{\gamma_v} g_n \big] \Big( t,z,\frac{x}{t} \Big) \dr z  ,
 \end{equation}
with $ |\gamma_v| \leq N-1-r_{n}$, and $p+|\alpha| \leq q \leq n$. We start by dealing with the terms for which $q \geq 1$. Note that $N_x-n \geq 4+n$, so \eqref{equa:induModscat} implies
$$   \frac{\log(t)}{t}\bigg| \int_{|z|<t} z^\alpha \big[ \partial_{v}^{\gamma_v} g_n \big] \Big(t,z,\frac{x}{t} \Big) \dr z - \int_{|z|<t} z^\alpha \big[ \partial_{v}^{\gamma_v} f_\infty \big] \Big(z,\frac{x}{t} \Big) \dr z \bigg| \lesssim \frac{\log^{S_n+1}(t)}{\langle x/t \rangle^3 \, t^{n+1}}.$$
Furthermore, we get from Corollary \ref{Corestig1} that $f_\infty$ enjoys strong spatial decay, so 
$$\frac{\log(t)}{t}\bigg| \int_{|z|\geq t} z^\alpha \big[ \partial_{v}^{\gamma_v} f_\infty \big] \Big(z,\frac{x}{t} \Big) \dr z \bigg| \lesssim \frac{\log(t)}{\langle x/t \rangle^3 \,t^{1+N_x-5-|\alpha|}} \lesssim \frac{\log(t)}{\langle x/t \rangle^3 \,t^{n+1}},$$
since $N_x \geq 2n+5$. It remains to treat the case $q=0$, for which we have $|\alpha|=0$. It then suffices to use the strong convergence estimate for the modified non-weighted spatial averages \eqref{equa:induQn} provided by the induction hypothesis. Finally, $\| \langle v \rangle^3 \partial_v^\beta \mathbf{F}_{p,q} \|_{L^\infty} <\infty$ ensues from the velocity decay of $f_\infty$ and its derivatives.
\end{proof}

\subsection{Late-time asymptotics for the force field}\label{SubsecforcefieldAxpan}

In this section we will prove that, for any $|\gamma| \leq N-r_{n+1}$ the normalised force field $t^2\nabla_x G^\gamma \phi$ admits a polyhomogeneous expansion of order $n$ along the modified spatial characteristics $t \mapsto (t,\X_{n}(t,x,v)+t\V_n(t,x,v))$. For this, we start by proving the next result. 
\begin{proposition}\label{Proexpa}
Let $ \Phi_{p,q} \in C^{N-r_{q+1}}(\R^3_v)$ be defined as
\begin{equation}\label{asymp_poiss_eqn}
 \Delta_v \Phi_{p,q}= \mathbf{F}_{p,q},
\end{equation}
for any $p \leq q \leq n$. Then, we have $ \Phi_{p,q} \in W^{N-r_{q+1},\infty} (\R^3_v)$, and for any $|\gamma| \leq N-r_{n+1}$
$$ \forall \, (t,x) \in [2,\infty) \times \R^3_x , \qquad  \bigg|t^{2+|\gamma|} \nabla_x \partial_x^\gamma \phi (t,x)-\sum_{p \leq q \leq n} \frac{\log^p(t)}{t^q} \big[ \nabla_v \partial_v^\gamma \Phi_{p,q} \big] \Big( \frac{x}{t} \Big)\bigg| \lesssim \frac{\log^{1+S_n}(t)}{t^{n+1}}  .$$
\end{proposition}
\begin{proof}
Let $|\gamma| \leq N-r_{n+1}$, and set
$$ \psi (t,x) := \sum_{p \leq q \leq n} \frac{\log^p(t)}{t^{q}} \partial_v^\gamma \Phi_{p,q} \Big( \frac{x}{t} \Big).$$ We note that $$\nabla_x \psi (t,x) = \sum_{p \leq q \leq n} \frac{\log^p(t) }{t^{q+1}} [ \nabla_v \partial_v^\gamma \Phi_{p,q} ] \Big( \frac{x}{t} \Big).$$
Then, as $t^2 \Delta_x [ \partial_v^\gamma \Phi_{p,q} (\frac{x}{t}) ]=[\Delta_v  \partial_v^\gamma \Phi_{p,q}](\frac{x}{t})$, we have
$$ t^2 \Delta_x  \big( t^{1+|\gamma|} \partial_x^\gamma \phi - \psi \big) (t,x) = t^3 \rho \big( G^\gamma f \big)(t,x)- \sum_{p \leq q \leq n} \frac{\log^p(t)}{t^q} \partial_v^\gamma \mathbf{F}_{p,q} \Big( \frac{x}{t} \Big),$$
so that 
$$ t^2 \big| \Delta_x  \big(  t^{1+|\gamma|} \partial_x^\gamma \phi - \psi \big) \big| (t,x) \lesssim  \langle t+|x| \rangle^{-3} \, t^{-n+2}\log^{1+S_n}(t).$$
It remains to estimate $\nabla_x (t^{1+|\gamma|} \partial_x^\gamma \phi - \psi)$ by using \eqref{remark_uniform_integral_bound_kernel_convolution_duan}. The boundedness of $\nabla_v \partial_v^\gamma \Phi_{p,q}$ is given by the estimates of $\partial_v^\gamma \mathbf{F}_{p,q}$ in Proposition \ref{Proselfsimrho}, and Lemma \ref{lemma_uniform_integral_bound_kernel_convolution_duan}.
\end{proof}

We next show that the force field satisfies an asymptotic polyhomogeneous expansion in the sense of Definition \ref{defpolhomoexpforce}.

\begin{proposition}\label{Proexpanphi}
For any $|\gamma| \leq N-r_{n+1}$, there exists $\overline{\Phi}_{q,\alpha,p}^{\,\gamma} \in C^{N-r_{q+1}}\cap W^{N-r_{q+1},\infty}(\R^3_v)$ such that
$$ \bigg|t^2 \nabla_x \partial_x^\gamma \phi \big(t,\X_n+t\V_n\big) - \sum_{q \leq n } \,  \sum_{p+|\alpha| \leq q} \frac{x^\alpha  \log^p(t)}{t^{q}} \overline{\Phi}^{\,\gamma}_{q,\alpha,p}(v)  \bigg| \lesssim \frac{ \langle x \rangle^{n+1}\log^{1+S_n}(t)}{ t^{n+1}},$$ for all $t \geq 2$, all $ |x| \leq t$, and all $v \in \R^3_v$.
\end{proposition}

\begin{proof}
Fix now $(t,x,v) \in [2,\infty) \times \R^3_x \times \R^3_v$ such that $|x| \leq t$ and apply a $n^{\mathrm{th}}$-order Taylor expansion to get
$$ \bigg| t^{2} \nabla_x \partial_x^\gamma \phi \big(t,\X_n+t\V_n \big)-t^2\sum_{|\alpha| \leq n} \frac{\big(\X_n+t(\V_n-v)\big)^\alpha}{\alpha ! } \nabla_x \partial_x^\alpha \partial_x^\gamma \phi (t,tv) \bigg| \lesssim \frac{\epsilon \,  \langle \X_n+t(\V_n-v) \rangle^{n+1}}{t^{n+1+|\gamma|}}   .$$
The RHS can be bounded for $t \geq |x|$, by using $|\X_n-x| \lesssim \log(t)$ and $t|\V_n-v| \lesssim 1$. Then, to derive the result:
\begin{enumerate}
\item We use the expansion obtained for $\nabla_x \partial_x^\gamma \phi $ in Proposition \ref{Proexpa}. If $|\alpha| \geq 1$, we appeal to the induction hypothesis at order $n-|\alpha|$ in order to get the corresponding expansion for $\nabla_x \partial_x^\alpha \partial_x^\gamma \phi$. Note for this that $|\alpha|+|\gamma| \leq N-r_{n-|\alpha|+1}$.
\item We expand $(\X_n+t(\V_n-v))^\alpha$. 
\item Finally, some terms thus obtained are lower order, so they have to be included in the error term. For this, we use that $\mathbb{X}_{q,\alpha,p}, \, \mathbb{V}_{q,\alpha,p}, \, \nabla_v \Phi_{p,q} \in W^{N-r_{q+1},\infty}_v$, and that $|x| \leq t$.
\end{enumerate} 
\end{proof}

\begin{remark}
One can check that $\overline{\Phi}^{ \, \gamma}_{1,\alpha,0}=\nabla_v \partial_v^\alpha \phi_\infty$ if $|\gamma|=0$ and $|\alpha|=1$. Since $\X_{n+1}$ will verify \eqref{defXnplus1} below, we will have that $\mathbb{X}_{1,\alpha,0}=-\nabla_v \partial_v^\alpha \phi_\infty$.
\end{remark}

\subsection{Improved modified scattering}\label{SubsecImproscatmod}

Let us begin by assuming that we have constructed $\X_{n+1}$ and $\V_{n+1}$. Then, by the same computations as in \eqref{timederivfirstexptechlemm}, we have
\begin{align}
\partial_t g_{n+1}(t,x,v) &= \big(\partial_t \X_{n+1}-t\mu \nabla_x \phi (t,\X_{n+1}+t\V_{n+1}) \big)\cdot \big[ \nabla_x f \big]\big(t,\X_{n+1}+t\V_{n+1},\V_{n+1}\big)\label{timederivnexptechlemm}\\
& \quad +\big( \mu \nabla_x \phi (t,\X_{n+1}+t\V_{n+1})+\partial_t \V_{n+1} \big) \cdot \big[ G f \big]\big(t,\X_{n+1}+t\V_{n+1},\V_{n+1}\big).\nonumber
\end{align}
Now, in view of the required form of $\X_{n+1}$ and $\X_n$, we have
\begin{equation}\label{eq:estXnplus1Xn}
 \forall \, (t,x,v) \in [2,\infty) \times \R^3_x \times \R^3_v,  \qquad |\X_{n+1}- \X_{n}|+t|\V_{n+1}-\V_n| \lesssim  \frac{\langle x \rangle^{n}\log^{n}(t)}{t^{n}} .
\end{equation}
Thus, we expect the difference $$\nabla_x \phi \big(t,\X_{n+1}+t\V_{n+1} \big)-\nabla_x \phi \big(t,\X_{n}+t\V_n \big)$$ to be strongly decaying. 

We then define $(\X_{n+1},\V_{n+1})$ as the \emph{unique polyhomogeneous map} such that, for all $t \geq 2$, $|x| \leq t$ and $v \in \R^3_v$, we have
\begin{align}
\big| t\mu \nabla_x \phi (t,\X_{n}(t,x,v)+t\V_{n}(t,x,v) \big) - \partial_t \X_{n+1}(t,x,v) \big| & \lesssim  \frac{\langle x \rangle^{n+1}\log^{1+S_n}(t)}{t^{n+2}}  , \label{defXnplus1} \\
  \big|\mu \nabla_x \phi \big(t,\X_n(t,x,v) +t\V_n(t,x,v)\big)+\partial_t \V_{n+1} (t,x,v)\big| &\lesssim   \frac{\langle x \rangle^{n+1}\log^{1+S_n}(t)}{t^{n+3}}   
  \end{align}
and $(\X_{n+1},\V_{n+1})$ has the constant in time term $(x,v)$. Recall Lemma \ref{LemForexp} for the uniqueness of the expansions. In view of Proposition \ref{Proexpanphi}, the maps $\X_{n+1}$ and $\V_{n+1}$ are indeed of the form stated in \eqref{eq:defCX }--\eqref{eq:defCV}. 

The next step consists in computing the derivatives of $\partial_t g_{n+1}$. For convenience, we introduce a notation.
  
\begin{definition}\label{DefPolyXV}
Let $a \in \mathbb{N}^*$ and $b \in \llbracket 1, N-r_{n+1} \rrbracket$. We denote by $P_{a,b}^{n+1}(\X)$ any quantity of the form
$$ \prod_{1 \leq i \leq a } \partial_{x,v}^{\gamma_i} \X^{k_i}_{n+1}, \qquad \quad k_i \in \llbracket 1 , 3 \rrbracket, \quad \sum_{1 \leq i \leq a } |\gamma_i| = b, \quad |\gamma_i| \geq 1.$$
 We set $P_{0,b}^{n+1}(\X)=1$. And we define $P_{a,b}^{n+1}(\V)$ similarly.
\end{definition}

We can easily bound these quantities.

\begin{lemma}\label{BoundPolyXV}
Let $a \in \mathbb{N}$ and $b \in \llbracket 0, N-r_{n+1} \rrbracket$. Then, for all $t \geq 2$,
$$ \forall \, |x| \leq t, \quad \forall \, v \in \R^3_v, \qquad  \big| P_{a,b}^{n+1}(\X) \big|(t,x,v) \lesssim \epsilon^a \log^a(t)+1, \qquad \big| P_{a,b}^{n+1}(\V) \big|(t,x,v) \lesssim 1.$$
\end{lemma}

As expected, these terms do not generate strongly divergent factors in the error terms.

\begin{lemma}\label{LemderivPhialongXnplus}
Let $|\kappa| \leq N-r_{n+1}$ and $(t,x,v) \in \R_+^* \times \R^3_x \times \R^3_v$. Then, $\partial_t \partial_{x,v}^\kappa g_{n+1}(t,x,v)$ can be written as a linear combination of terms of the first kind
\begin{align*}
\mathbb{T}^{1,\kappa}_{\gamma,\alpha,\beta,b,c}:=~& \partial_{x,v}^\gamma \Big( \! \partial_t\X_{n+1}\!-t \nabla_x \phi (t,\X_{n+1}+t\V_{n+1}) \! \Big) \\
&\qquad \qquad \cdot \big[\nabla_x \partial_x^\alpha G^\beta \! f \big] \big(t,\X_{n+1}\!+t\V_{n+1},\V_{n+1} \big) P_{|\alpha|, \, b}^{n+1}(\X)P_{|\beta|, \,c}^{n+1}(\V), 
\end{align*}
or of the second kind
\begin{align*}
\mathbb{T}^{2,\kappa}_{\gamma,\alpha,\beta,b,c}:=~& \partial_{x,v}^\gamma \Big( \! \nabla_x \phi (t,\X_{n+1}+t\V_{n+1})-\partial_t\V_{n+1} \! \Big) \\
&\qquad \qquad \cdot  \big[G \partial_x^\alpha G^\beta \! f \big] \big(t,\X_{n+1}+t\V_{n+1},\V_{n+1} \big) P_{|\alpha|, \, b}^{n+1}(\X)P_{|\beta|, \, c}^{n+1}(\V), 
\end{align*}
where $|\gamma|+|\alpha|+|\beta| \leq |\kappa|$, and $b+c \leq |\kappa|$. 
\end{lemma}

\begin{remark}\label{Rk:easy}
If $\partial_{x,v}^\kappa = \partial_v^\kappa$ contains only velocity derivatives, then so does $\partial_{x,v}^\gamma$.
\end{remark}

\begin{proof}
Recall that $G=t\nabla_x +\nabla_v$, so that with $\partial_i$ denoting either $\partial_ {x^i}$ or $\partial_{v^i}$, we have
\begin{equation}\label{relation_first_second_kind}
 \partial_{i} \big[ h(t,\X+t\V,\V) \big] = \partial_{i} \X \cdot \big[ \nabla_x h \big] (t,\X+t\V,\V) +\partial_{i} \V \cdot \big[G h \big] (t,\X+t\V,\V),
\end{equation}
for any $C^1$ function $h:\R_+ \times \R^3_x\times \R^3_v \to \R$, and maps $\X,\, \V:\R_+^* \times \R^3_x\times \R^3_v \to \R^3$. We then obtain the result from an induction, where the base case is treated in \eqref{timederivnexptechlemm}, and the induction step is treated by using the relation \eqref{relation_first_second_kind} for $(\X,\V)=(\X_{n+1},\V_{n+1})$.
\end{proof}
Next, we control the first factor in these two types of error terms.
\begin{lemma}\label{LemderivModcoeff}
For any $|\alpha_x|+|\alpha_v| \leq N-r_{n+1}$ and for all $(t,x,v) \in [2,\infty) \times \R^3_x \times \R^3_v$ with $t \geq |x|$, we have
\begin{align*} 
  \Big|\partial_x^{\alpha_x} \partial_v^{\alpha_v} \Big( \partial_t \X_{n+1}- t\mu \nabla_x \phi \big(t,\X_{n+1}+t\V_{n+1}\big) \Big) \Big| & \lesssim   \langle x \rangle^{n+1}\log^{1+S_n}(t) t^{-n-2} , \\
  \Big|\partial_x^{\alpha_x} \partial_v^{\alpha_v} \Big(\mu \nabla_x \phi \big(t,\X_{n+1}+t \V_{n+1} \big)+\partial_t \V_{n+1} \Big) \Big| & \lesssim   \langle x \rangle^{n+1} \log^{1+S_n}(t) t^{-n-3} .
\end{align*}
\end{lemma}
\begin{proof}
According to Proposition \ref{proposition_estimate_phi}, the mean value theorem, and \eqref{eq:estXnplus1Xn}, we have
$$ t^{|\gamma|}\big| \big[\nabla_x \partial_x^{\gamma}  \phi\big] \big(t,\X_{n+1}+t\V_{n+1}\big)- \big[\nabla_x \partial_x^{\gamma}  \phi\big] \big(t,\X_{n}+t\V_{n}\big) \big| \lesssim \epsilon t^{-3-n} \langle x , \log(t)\rangle^n,$$ for any $|\gamma| \leq N-r_{n+1} $. In view of the definition of $(\X_{n+1},\V_{n+1})$, it implies that: 
\begin{itemize}
\item The two estimates in the statement hold for $|\alpha_x|+|\alpha_v|=0$.
\item  $\partial_x^{\alpha_x} \partial_v^{\alpha_v} \big[ \nabla_x \phi \big(t,\X_{n+1}+t \V_{n+1} \big)\big]$ admits an asymptotic polyhomogeneous expansion of order $n+2+|\alpha_x|$ (with an error term decaying at least as $\langle x \rangle^{n+1} \log^{1+S_n}(t) t^{-n-3-|\alpha_x|}$), by using the chain rule, and Proposition \ref{Proexpanphi}. Thus, so do the LHS in the second estimate of the statement. The LHS in the first estimate admits an expansion of order $n+1+|\alpha_x|$ because of the additional factor $t$.
\end{itemize}
These two properties together with Lemma \ref{LemForexp}, imply the result. Note that because of the last part of Lemma \ref{LemForexp}, the spatial derivatives do not provide an improved estimate.
\end{proof}
We are finally able to conclude this subsection. We recall that $S_{n+1}\geq S_n+N+1-r_{n+1}$.

\begin{proposition}\label{ProModscattn1}
For any $|\kappa| \leq N-r_{n+1}$, there holds
$$ \forall \, t \geq 2, \; \forall |x| \leq t, \; \forall v \in \R^3_v,  \qquad   \langle x \rangle^{N_x-n-1} \langle v \rangle^{N_v} \big| \partial_{x,v}^\kappa g_{n+1}(t,x,v)-\partial_{x,v}^\kappa f_\infty (x,v) \big| \lesssim \frac{\log^{S_{n+1}}(t)}{t^{n+1}} .$$
\end{proposition}

\begin{proof}
Since $|\X_1-\X_{n+1}| \lesssim 1$ and $|v-\V_{n+1}| \lesssim t^{-1}$ on the set $\{|x| \leq t \} \times \R^3_v$, we have by Lemma \ref{lemmaesti_f} that
\begin{equation}\label{toproveafter2}
\forall \, t \geq 2, \; \forall |x| \leq t, \; \forall v \in \R^3_v,  \qquad \langle x \rangle^{N_x} \langle v \rangle^{N_v} \big| \partial_x^{\gamma_x}G^{\gamma_v} f \big| \big(t,\X_{n+1}+t\V_{n+1},\V_{n+1} \big) \lesssim  \log^{|\gamma_v|}(t),
\end{equation}
for any $|\gamma_x|+|\gamma_v| \leq N-1$. Combining this estimate with the Lemmata \ref{BoundPolyXV}--\ref{LemderivPhialongXnplus}--\ref{LemderivModcoeff}, we have
$$ \forall \, t \geq 2, \; \forall |x| \leq t, \; \forall v \in \R^3_v,  \qquad \langle x \rangle^{N_x-n-1} \, \langle v \rangle^{N_v} \big| \partial_t \partial_{x,v}^\kappa g_{n+1} \big|(t,x,v) \lesssim \frac{\log^{1+S_n+N-r_{n+1}}(t)}{t^{n+2}} .$$
To conclude the proof, it remains to show that $f_\infty$ is indeed the limit of $g_{n+1}$. For this, we use that $f_\infty$ is the limit of $g_n$, and that 
$$|g_{n+1}(t,x,v)-g_n(t,x,v)| \leq \|\nabla_{x,v} g_n(t,\cdot,\cdot) \|_{L^\infty_{x,v}} \big( |\X_{n+1}-\X_n|+t|\V_{n+1}-\V_n|),$$ according to the mean value theorem. By \eqref{eq:estXnplus1Xn} and \eqref{equa:induModscat}, the RHS converges in $L^\infty_{\mathrm{loc}}(\R^3_x \times \R^3_v)$ to zero.
\end{proof}

\subsection{Strong convergence estimates for the spatial averages of $g_{n+1}$}\label{Subsecspatialgnplus1}

A direct application of Proposition \ref{ProModscattn1} provides that the spatial average of $g_{n+1}$ converges to the one of $f_\infty$ as $t^{-n-1} \log^{S_{n+1}}t$. As we performed in Section \ref{Subsubsecg1} for the analysis of $g_{1}$, we are in fact able to establish an enhanced convergence estimate for the spatial average of $g_{n+1}$. For this, we can decompose $\partial_t \partial_{x,v}^\kappa g_{n+1}$ as in Lemma \ref{LemderivPhialongXnplus} into the sum of terms of the first and second kind. The terms of the first kind decay as $t^{-n-2}\log^{1+S_n}(t)$ and they govern the rate of convergence in Proposition \ref{ProModscattn1}. In contrast, the terms of the second kind decay as $t^{-n-3}\log^{1+S_n}(t)$. 

In order to obtain the improved convergence estimate, we need to carefully study the spatial average of the terms of the first kind. The key idea will be to exploit that these terms carry a factor of the form $\nabla_x Z^\xi f$ allowing for integration by parts. 

\begin{proposition}
Let $|\kappa| \leq N-1-r_{n+1}$. Then, there exist $\mathbf{Q}_{p,\xi}^{\kappa,\beta} \in C^0 \cap L^\infty (\R^3_v) $ such that
\begin{align*}
 \bigg| \int_{|z|<t } \partial_v^\kappa g_{n+1}(t,z,v) \dr z -\int_{\R^3_z} \partial_v^\kappa f_\infty (z,v) \dr z- \sum_{|\beta| \leq |\kappa|} \sum_{p +|\xi| \leq n} &\frac{\log^p(t)}{t^{n+1}} \mathbf{Q}_{p,\xi}^{\kappa,\beta} (v)\int_{\R^3_z} z^\xi \partial_v^\beta f_\infty (z,v) \dr z \bigg| \\
  \lesssim t^{-n-2} \log^{S_{n+1}}(t)&,
  \end{align*}
for all $(t,v) \in [2,\infty) \times \R^3_v$.
\end{proposition}
\begin{proof}
Fix $(t,v) \in [2,\infty) \times \R^3_v$. We start by applying Lemma \ref{LemderivPhialongXnplus}. There exists $N^\kappa_{\gamma,\alpha,\beta,b,c} \in \R$ such that
\begin{align}
\nonumber \langle v \rangle^3\bigg| \partial_t \! \int_{|z| <t } &\partial_v^\kappa g_{n+1}(t,z,v) \dr z-\sum_{|\gamma|+|\alpha|+|\beta| \leq |\kappa|} \, \sum_{b+c \leq |\kappa|}N_{\gamma,\alpha,\beta,b,c}^\kappa \int_{|z|<t}\mathbb{T}^{1,\kappa}_{\gamma,\alpha,\beta,b,c}(t,z,v) \dr z \bigg| \\
  &\lesssim  \int_{\mathbb{S}^2_\omega}\langle v \rangle^3 \big| \partial_v^\kappa g_{n+1}(t,t\omega,v)  \big| t^2 \dr \mu_{\mathbb{S}^2_\omega} + \sup_{|\gamma|+|\alpha|+|\beta| \leq |\kappa|} \, \sup_{b+c \leq |\kappa|} \langle v \rangle^3 \bigg|\int_{|z|<t}\mathbb{T}^{2,\kappa}_{\gamma,\alpha,\beta,b,c}(t,z,v) \dr z \bigg|, \label{kevatal:error}
\end{align}
where $\dr \mu_{\mathbb{S}^2}$ is the volume form on the unit sphere $\mathbb{S}^2$. The goal consists in proving the following properties:
\begin{itemize}
\item Any error term in the RHS decay faster than $t^{-n-3}\log^{S_{n+1}}(t)$.
\item The integral over $\{|z| \!< \! t \}$ of any term $\mathbb{T}^{1,\kappa}_{\gamma,\alpha,\beta,b,c}(t,z,v)$ is $(n\!+\!2)$--equivalent to a sum of terms 
\begin{equation}\label{kevatalenn:1}
 \frac{\log^p(t)}{t^{n+2}} \overline{\mathbf{Q}}_{ \, p,\xi}^{\kappa,\beta} (v)\int_{\R^3_z} z^\xi \partial_v^\beta f_\infty (z,v) \dr z, \qquad p+|\xi| \leq n.
 \end{equation}
\end{itemize} 
To conclude, it remains to take the antiderivatives in time of quantities of the form \eqref{kevatalenn:1}. Here, we identify the limit by applying Proposition \ref{ProModscattn1}, which implies
$$ \lim_{t \to + \infty} \int_{|z|<t}g_{n+1}(t,z,v) \dr z = \int_{\R^3_z} f_\infty (z,v) \dr z.$$

\textbf{Control of boundary terms.} Recall that $N_x \geq 2n+6$. We control the boundary term in \eqref{kevatal:error} by applying Proposition \ref{ProModscattn1}, together with Corollary \ref{Corestig1} in order to estimate $f_\infty$. So, for every $ |\xi| \leq N-r_{n+1}$ we have
\begin{equation*}
 \int_{\mathbb{S}^2_\omega} \big| \partial_v^\xi g_{n+1}(t,t\omega,v)  \big|  \dr \mu_{\mathbb{S}^2_\omega} \lesssim t^{-N_x+n+1}\langle v \rangle^{-N_v} \lesssim t^{-n-5}\langle v \rangle^{-3}.
 \end{equation*}
We will deal with other boundary terms by the next consequence of Corollary \ref{Corestig1},
 \begin{equation}\label{kevatalenn:boundary}
 \color{white} \square \qquad \color{black} \mathcal{B}_\gamma(t,v):= \! \int_{\mathbb{S}^2_\omega}  \big| \partial_{x,v}^\gamma g_1 \big| \big(t,t\omega,v \big)  \dr \mu_{\mathbb{S}^2_\omega} \lesssim t^{-N_x+1}\langle v \rangle^{-3} , 
 \end{equation}
with $|\gamma| \leq N-2$.

\textbf{Estimate for terms of the second kind.} Let us fix $b+c\leq |\kappa|$ and $|\gamma|+|\alpha|+|\beta| \leq |\kappa|$. Combining \eqref{toproveafter2} with the Lemmata \ref{BoundPolyXV} and \ref{LemderivModcoeff}, we bound the terms of the second kind by
\begin{equation}\label{equa:T2}
 \forall \, |z| \leq t, \qquad  \big| \mathbb{T}^{2,\kappa}_{\gamma,\alpha,\beta,b,c}\big|(t,z,v) \lesssim \langle z \rangle^{-4}\langle v \rangle^{-3} t^{-n-3} \log^{1+S_n+N-r_{n+1}}(t).
 \end{equation}
Consequently, these terms give rise to error terms decaying as $t^{-n-3}\log^{S_{n+1}}(t)$ in \eqref{kevatal:error}. 

\textbf{Estimate for terms of the first kind.} We now focus on the term of the first kind $\mathbb{T}^{1,\kappa}_{\gamma,\alpha,\beta,b,c}(t,z,v)$, that is, terms of the form
\begin{align*}
&  \partial_{v}^\gamma \Big( \! \partial_t\X_{n+1}\!-t \nabla_x \phi (t,\X_{n+1}+t\V_{n+1}) \! \Big) \cdot \big[\nabla_x \partial_x^\alpha G^\beta \! f \big] \big(t,\X_{n+1}\!+t\V_{n+1},\V_{n+1} \big) P_{|\alpha|, \, b}^{n+1}(\X)P_{|\beta|, \,c}^{n+1}(\V). 
\end{align*}
From now on, we study the decay properties of a fixed term $\mathbb{T}^{1,\kappa}_{\gamma,\alpha,\beta,b,c}$.

\textit{Reduction of the analysis to $ \overline{\mathbb{T}}{\vphantom{\mathbb{T}}}^{\,1}$.}We first reduce the analysis to the study of 
$$\overline{\mathbb{T}}{\vphantom{\mathbb{T}}}^{\,1}(t,z,v):=  \partial_{v}^\gamma \Big( \! \partial_t\X_{n+1}\!-t \nabla_x \phi (t,\X_{n+1}+t\V_{n+1}) \! \Big) \cdot \big[\nabla_x \partial_x^\alpha G^\beta \! f \big] \big(t,\X_{n+1}\!+t\V_{n+1},\V_{n+1} \big) P_{|\alpha|, \, b}^{1}(\X),$$
where $P_{|\alpha|, \, b}^{1}(\X)$ is obtained by formally replacing $\X_{n+1}$ by $\X_1$, so that
$$ P_{|\alpha|, \, b}^{1}(\X):= \log^{|\alpha|}(t) \prod_{1 \leq i \leq |\alpha| }\nabla_v \partial_{x,v}^{\gamma_i} \phi_\infty^{k_i}(v), \qquad  k_i \in \llbracket 1 , 3 \rrbracket, \quad \sum_{1 \leq i \leq |\alpha| } |\gamma_i| = b, \quad |\gamma_i| \geq 1.$$ 
For this, we note that:
\begin{itemize}

\item By the form \eqref{eq:defCX } of $\X_{n+1}-\X_1$, we have $|P_{|\alpha|, \,b}^{n+1}(\X)-P_{|\alpha|, \,b}^{1}(\X)| \lesssim \langle z , \log(t) \rangle \, t^{-1} \log^{|\alpha|-1}(t)$.

\vspace{0.5mm}

\item By the form \eqref{eq:defCV} of $\V_{n+1}-v$, either $|P_{|\beta|, \,c}^{n+1}(\V)| \lesssim t^{-1}$ or $|P_{|\beta|, \,c}^{n+1}(\V)-1| \lesssim t^{-1}$.

\vspace{0.5mm}

\item Thus, by \eqref{toproveafter2} and the Lemmata \ref{BoundPolyXV} and \ref{LemderivModcoeff}, one of the next two estimates hold:
 $$ \color{white} \square \qquad \color{black} \big|\mathbb{T}^{1,\kappa}_{\gamma,\alpha,\beta,b,c} \big|(t,z,v)\lesssim \frac{\log^{S_n+N-r_{n+1}}(t)}{\langle z \rangle^4\langle v \rangle^3 t^{n+3}} ,  \qquad\quad \big|\mathbb{T}^{1,\kappa}_{\gamma,\alpha,\beta,b,c}-\overline{\mathbb{T}}{\vphantom{\mathbb{T}}}^{\,1} \big|(t,z,v)\lesssim \frac{\log^{S_n+N-r_{n+1}}(t)}{\langle z \rangle^4\langle v \rangle^3 t^{n+3}}.$$ 
\end{itemize}
Therefore, if the first estimate holds, the term $\mathbb{T}^{1,\kappa}_{\gamma,\alpha,\beta,b,c}$ is of lower order. Otherwise, if the second estimate holds, we carry on with the analysis of $\overline{\mathbb{T}}{\vphantom{\mathbb{T}}}^{\,1}$.

\textit{Reduction of the analysis to $ \widetilde{\mathbb{T}}{\vphantom{\mathbb{T}}}^{\,1}$.} We next show that we can in fact focus on
$$ \widetilde{\mathbb{T}}{\vphantom{\mathbb{T}}}^{\,1}(t,z,v):=  \partial_{v}^\gamma \Big( \! \partial_t\X_{n+1}\!-t \nabla_x \phi (t,\X_{n+1}+t\V_{n+1}) \! \Big) \cdot \big[\nabla_x \partial_x^\alpha G^\beta \! f \big] \big(t,\X_{1}+t\V_{1},\V_{1} \big) P_{|\alpha|, \, b}^{1}(\X)$$
since the difference $\overline{\mathbb{T}}{\vphantom{\mathbb{T}}}^{\,1}-\widetilde{\mathbb{T}}{\vphantom{\mathbb{T}}}^{\,1}$ is lower order, that is
$$ \forall \, |z| \leq t ,  \qquad \big|\overline{\mathbb{T}}{\vphantom{\mathbb{T}}}^{\,1}-\widetilde{\mathbb{T}}{\vphantom{\mathbb{T}}}^{\,1}\big|(t,z,v) \lesssim \frac{\log^{1+S_n+N-r_{n+1}}(t)}{\langle z \rangle^{4}\langle v \rangle^{3} t^{n+3}} .$$
For this, we apply Lemma \ref{LemderivModcoeff}, and we use the mean value theorem together with Lemma \ref{lemmaesti_f} as well as
$$ \forall \, |z| \leq t ,  \qquad  \big| \X_{n+1}-\X_1\big|(t,z,v)+t \big|\V_{n+1}-\V_1\big|(t,z,v) \lesssim \frac{1}{t}\langle z , \log(t) \rangle \, .$$

\textit{Comparison with $g_1$ and conclusion.} Our last reduction consists in applying Remark \ref{RkZbetafg1} and Lemma \ref{LemPhiminus1}, in order to write $G^\beta f$ in terms of the derivatives of $g_1$. This implies that $\widetilde{\mathbb{T}}{\vphantom{\mathbb{T}}}^{\,1}(t,z,v)$ can be written as a linear combination of terms
\begin{itemize}
\item either decaying as $\langle z \rangle^{-4}\langle v \rangle^{-3} t^{-n-3} \log^{1+S_n+N-r_{n+1}}(t)$,
\item or of the form
$$\color{white} \square \quad \color{black} \mathbb{T}(t,z,v):=  \partial_{v}^\gamma \Big( \! \partial_t\X_{n+1}\!-t \nabla_x \phi (t,\X_{n+1}+t\V_{n+1}) \! \Big) \cdot \big[\nabla_x \partial_x^\nu \partial_v^\mu  g_1 \big] (t,z,v) P_{|\nu|, \, b}^{1}(\X),$$
where $|\nu| \geq |\alpha|$, and $|\nu|+|\mu| \leq |\alpha|+|\beta|$.
\end{itemize}

Note now that integration by parts yields, as $\X_1-x$ and $\V_1$ are independent of $x$, that
\begin{align*}
\langle v \rangle^3 \bigg| \int_{|z| < t }&\nabla_x \cdot \partial_x^\nu \partial_{v}^\gamma \Big( \! \partial_t\X_{n+1}\!-t \nabla_x \phi (t,\X_{n+1}+t\V_{n+1}) \! \Big)  \big[  \partial_v^\mu  g_1 \big] (t,z,v) P_{|\nu|, \, b}^{1}(\X)  - \mathbb{T}^1(t,z,v) \dr z \bigg| \\
& \lesssim \frac{\langle t \rangle^{n+1}\log^{1+S_n+|\nu|}(t)}{ t^{n+2}} \sup_{|\xi| \leq |\nu|+|\mu|} \langle v \rangle^3 \mathcal{B}_\xi (t,v) \lesssim \frac{1}{t^{n+3}},
\end{align*}
where we used Lemma \ref{LemderivModcoeff}, \eqref{kevatalenn:boundary}, and $N_x-1 \geq n+3$ in the last step. Assume first that $|\nu|+1 \leq n$ and note that $|\gamma|+|\nu|+1 \leq N-r_{n+1}$. Then:
\begin{itemize}
\item Recall from Lemma \ref{LemderivModcoeff} that $\nabla_x \partial_x^\nu \partial_{v}^\gamma (  \partial_t\X_{n+1}\!-t \nabla_x \phi (t,\X_{n+1}+t\V_{n+1}) )$ admits an asymptotic polyhomogeneous expansion of order $n+2+|\nu|$. Here, the error term decays as $\langle x \rangle^{n+1}\! \log^{1+S_n}(t) t^{-n-3-|\nu|}$.
\item According to Lemma \ref{LemderivModcoeff}, all the terms of order up to $n+1$ vanish.
\end{itemize} 
Consequently, there exist $\mathbf{K}^\xi_p \in C^0 \cap L^\infty (\R^3_v)$ such that, for all $|z| \leq t$ and all$v \in \R^3_v$, we have
$$\bigg| \nabla_x \cdot \partial_x^\nu \partial_{v}^\gamma \Big( \! \partial_t\X_{n+1}\!-t \nabla_x \phi (t,\X_{n+1}+t\V_{n+1}) \! \Big) - \! \sum_{p+|\xi| \leq n-|\nu|} \! \frac{z^\xi \log^p(t)}{t^{n+2}} \mathbf{K}^\xi_{p}(v) \bigg| \lesssim  \!\frac{ \langle z \rangle^{n+1}\log^{1+S_n}(t)}{ t^{n+3}}.$$
Next, we use the particular form of $P^1_{|\nu|,b}(\X)$ to show that there exists $\mathbf{L}^\xi_{p} \in C^0 \cap L^\infty (\R^3_v)$ such that
\begin{align*}
\langle v \rangle^3 \bigg| \int_{|z| < t }\mathbb{T}(t,z,v)\dr z - \! \sum_{p+|\xi|\leq  n} \! \frac{\log^p(t)}{t^{n+2}} \mathbf{L}_{p}^\xi(v) \int_{|z|<t}z^\xi\big[ \partial_v^\beta g_1 \big] (t,z,v) \dr z \bigg|  \lesssim \frac{\log^{1+S_n}(t)}{t^{n+3}}.
\end{align*}
It remains to apply Corollary \ref{Corestig1} for $|\xi| \leq n $, which provides
$$\bigg|\int_{|z|<t}z^\xi\big[ \partial_v^\beta g_1 \big] (t,z,v) \dr z-\int_{\R^3_z} z^\xi \partial_v^\beta f_\infty (z,v) \dr z \bigg| \lesssim \frac{\log^{N}(t)}{t} \lesssim  \frac{\log^{S_{n+1}-n}(t)}{t}.$$
If $|\nu| \geq n$, we use $|\partial_x^\zeta \X_{n+1}|(t,z,v)+t|\partial_x^\zeta \V_{n+1}|(t,z,v)\lesssim t^{-|\zeta|} $ and $|\nabla_x \partial_x^\zeta \phi|(t,x) \lesssim \epsilon t^{-2-|\zeta|}$ to argue that
$$ \big| \nabla_x \partial_x^\nu \partial_{v}^\gamma \big(  \partial_t\X_{n+1}\!-t \nabla_x \phi (t,\X_{n+1}+t\V_{n+1}) \big) \big| \lesssim \frac{1}{t^{n+3}}.$$
One can then deduce that $\mathbb{T}$ also satisfies \eqref{equa:T2}. This concludes the proof. 
\end{proof}

\section{Non-linear tails and weak convergence}\label{sec_tail_weak_conv}

In this section, we obtain late-time tails for the spatial density and the force field. We also prove that the distribution (up to normalisation) converges weakly to a Dirac mass in the zero velocity set. We finally capture the shearing of the system with a weak convergence statement.

\subsection{Hierarchy of asymptotic conservation laws}

Let $f_{\infty}\colon\R^3_x\times\R^3_v\to \R$ be a regular scattering state. Let $|\alpha|+|\beta| \leq N-2$ be multi-indices. We consider the \emph{weighted spatial averages} $\mathcal{A}^{\alpha}_{\beta}:\R^3_v\to\R$ given by 
\begin{equation*}
\mathcal{A}^{\alpha}_{\beta}(v):=\int_{\R^3_x} x^{\alpha}\partial_v^{\alpha+\beta}f_{\infty}(x,v)\dr x.
\end{equation*}
By Proposition \ref{prop_def_cons_laws}, we note that $\mathcal{A}^{\alpha}_{\beta}$ is bounded by a weighted $L^{\infty}_{x,v}$ norm of the scattering state $f_{\infty}$. Moreover, the weighted spatial averages $\mathcal{A}^{\alpha}_{\beta}$ can be defined as $$\mathcal{A}^{\alpha}_{\beta}(v)=\lim_{t\to \infty} \int_{\R^3_x}x^{\alpha}\partial_{v}^{\alpha+\beta}g_1(t,x,v)\dr x,$$ by Corollary \ref{Corestig1}. See for comparison Proposition \ref{prop_cons_laws_linear} for the definition of the conservation laws of the linearised system.

\subsection{Late-time tail for the spatial density}

Let $f\colon [0,\infty)\times  \R^3_x\times\R^3_v\to \R$ be a small regular solution for the Vlasov--Poisson system. The late-time asymptotics of the spatial density are obtained in Proposition \ref{Proselfsimrho}. In this subsection, we apply this result to obtain non-linear tails for the spatial density. We begin writing the tail of the self-similar profile $\mathcal{A}^{\alpha}_{\beta}(\frac{x}{t})$ in terms of the conservation laws $\mathcal{A}^{\alpha}_{\beta}(0)$. 

\begin{lemma}\label{lemexpconslawprofil}
Let $N\geq 2$ and $|\alpha|+|\beta|+|\gamma|\leq N-2 $. Then, the weighted spatial average $\mathcal{A}^{\alpha}_{\beta}$ satisfies that for all $(t,x)\in [2,\infty) \times \R^3_x$,
$$\Big|\mathcal{A}^{\alpha}_{\beta}\Big(\frac{x}{t}\Big)-\sum_{|\gamma|\leq N-|\alpha|-|\beta|} \frac{1}{\gamma !}\partial_v^{\gamma}\mathcal{A}^{\alpha}_{\beta}(0)\frac{x^{\gamma}}{t^{|\gamma|}}\Big|\lesssim  \frac{|x|^{N+1-|\alpha|-|\beta|}}{t^{N+1-|\alpha|-|\beta|}}$$
\end{lemma}

\begin{proof}
The proof follows by a straightforward application of Taylor's theorem (see the proof of Lemma \ref{lemexpconslawprofillin}). Here, we have bounded the norm of the scattering state by the initial data norm.
\end{proof}

We similarly apply Taylor to obtain an expansion for the self-similar profile $\mathbf{F}_{p,q}(\frac{x}{t})$. Recall that by Proposition \ref{Proselfsimrho}, these profiles can be written in terms of the hierarchy of asymptotic conservation laws $\mathcal{A}^{\alpha}_{\beta}$. Moreover, according to Lemma \ref{LemForexpbis}, we have $\mathbf{F}_{p,q} \in C^{N-r_{q+1}}\cap W^{N-r_{q+1}, \infty}(\R^3_v)$.

\begin{lemma}\label{lemexpconslawprofil_more_general}
Let $N\geq 4$ and $(p,q) \in \mathbb{N}^2$ such that $p \leq q$ and $r_{q+1} \leq N$. For any integer $k \in \mathbb{N}$ and multi-index $\beta$ satisfying $k+|\beta| \leq N-r_{q+1}-1$, we have
$$\forall \, (t,x)\in[2,\infty)\times \R^3_x, \qquad \bigg| \partial_v^\beta \mathbf{F}_{p,q}\Big(\frac{x}{t}\Big)-\sum_{|\gamma|\leq k} \frac{1}{\gamma !}\partial_v^{\beta+\gamma}\mathbf{F}_{p,q}(0)\frac{x^{\gamma}}{t^{|\gamma|}}\bigg|\lesssim  \frac{|x|^{k+1}}{t^{k+1}}.$$ 
\end{lemma}

Finally, we obtain late-time tails for the spatial density in terms of the hierarchy of asymptotic conservation laws $\mathcal{A}(f)$.

\begin{theorem}\label{thm_latetime_nonlinear}
Let $N\geq 4$ and $n \in \mathbb{N}$ such that $r_{n+1} \leq N$. For any $|\beta| \leq N-r_{n+1}-1$ and all $(t,x)\in[2,\infty)\times \R^3_x$, we have
\begin{align*}
\bigg|t^3 \! \int_{\R^3_v}G^{\beta}f(t,x,v)\dr v-\sum_{p\leq q\leq n}  \, \sum_{|\gamma|\leq n-q}  \frac{C_{\alpha}}{\gamma !}\partial_v^{\gamma+\beta}\mathbf{F}_{p,q}(0)\frac{x^{\gamma}\log^p (t)}{t^{|\gamma|+q}}\bigg|\lesssim \dfrac{\log^{S_n+1}(t) }{t^{n+1}} |x|^{n+1}.
\end{align*}
Here, $\partial_v^{\gamma+\beta}\mathbf{F}_{p,q}(0)$ can be computed in terms of $\mathcal{A}^{\alpha}_{\beta}(v)$ and its derivatives for $|\alpha|+|\beta|\leq N-2$.
\end{theorem}

\begin{proof}
Since $r_{n+1}=r_n+n+1$, the proof follows directly by Proposition \ref{Proselfsimrho} and Lemma \ref{lemexpconslawprofil_more_general}.
\end{proof}

\subsection{Late-time tail for the force field}

The late-time asymptotics of the force field are obtained in Proposition \ref{Proexpa}. In this subsection, we apply this result to obtain non-linear tails for the force field. We begin writing the tails of the self-similar profile $\nabla_v \Phi_{p,q}(\frac{x}{t})$ in terms of the constants $\nabla_v \partial_v^\gamma \Phi_{p,q}(0)$.

\begin{lemma}\label{lemexpconslawprofil_more_general_force}
Let $N\geq 4$ and $(p,q) \in \mathbb{N}^2$ such that $p \leq q$ and $r_{q+1} \leq N$. For any integer $k \in \mathbb{N}$ and multi-index $\beta$ satisfying $k+|\beta| \leq N-r_{q+1}-1$, we have
$$  \forall \, (t,x) \in [2,\infty) \times \R^3_x, \qquad  \bigg|\nabla_v \partial_v^\beta \Phi_{p,q}\Big(\frac{x}{t}\Big)-\sum_{|\gamma|\leq k} \frac{1}{\gamma !}\nabla_v \partial_v^{\gamma+\beta}\Phi_{p,q}(0)\frac{x^{\gamma}}{t^{|\gamma|}}\bigg|\lesssim  \frac{|x|^{k+1}}{t^{k+1}}.$$ 
\end{lemma}

Finally, we obtain late-time tails for the force field $\nabla_x\phi$ in terms of the profiles $\Phi^\beta_{p,q}$. We recall that $\Phi^\beta_{p,q}$ are explicitly defined using the hierarchy of asymptotic conservation laws $\mathcal{A}(f)$.

\begin{theorem}\label{thm_latetime_nonlinear_force}
Let $N\geq 4$ and $n \in \mathbb{N}$ such that $r_{n+1} \leq N$. For any $|\beta| \leq N-r_{n+1}-1$ and all $(t,x)\in[2,\infty)\times \R^3_x$, we have
\begin{align*}
\bigg|t^{2+|\beta|} \nabla_x \partial_x^\beta \phi (t,x)-\sum_{p\leq q\leq n}\sum_{|\gamma|\leq n-q} \frac{C_{\alpha}}{\gamma !}\partial_v^{\gamma+\beta}\nabla_v \Phi_{p,q}(0)\frac{x^{\gamma}\log^p (t)}{t^{|\gamma|+q}}\bigg|\lesssim \dfrac{\log^{S_n+1}(t) }{t^{n+1}} |x|^{n+1}.
\end{align*}
where $\Phi_{p,q}$ is defined by $\Delta_v\partial_v^{\gamma+\beta}\Phi_{p,q}= \partial_v^{\gamma+\beta}\mathbf{F}_{p,q}$.
\end{theorem}

\begin{proof}
The proof follows directly by Proposition \ref{Proexpa} and Lemma \ref{lemexpconslawprofil_more_general_force}.
\end{proof}

\subsection{Weak convergence properties}\label{sec_weak_conv}

In this section, we show the weak convergence of the normalised distribution $t^3G^{\beta}f$ for small data solutions to the Vlasov--Poisson system in terms of the scattering state $\partial_v^{\beta}f_{\infty}$. For this purpose, we first compute the limit of the spatial averages studied in Proposition \ref{ProconvQ} in terms of the scattering state.

\begin{proposition}\label{lem_constant_expansion_rho_term_scatt_state}
Let $|\beta|\leq N-2$. For all $(t,v) \in [2,\infty)\times \R^3_v$, we have $$\bigg|\int_{\R_x^3} G^{\beta}f(t,x,v)\mathrm{d}x- \int_{\R_x^3}\partial_v^{\beta}f_{\infty}(x,v)\mathrm{d}x\bigg|\lesssim \e \frac{ \log^{N-1}(t)}{t}.$$ In particular, the spatial average $\int_{\R_x^3} G^{\beta}f(t,x,v)\mathrm{d}x$ converges to $\int_{\R_x^3}\partial_v^{\beta}f_{\infty}(x,v)\mathrm{d}x$ as $t\to \infty$.
\end{proposition}
\begin{proof}
We recall that $[G^\beta f](t,x+tv,v)=\partial_v^\beta g_0(t,x,v)$, so it suffices to prove the result for $|\beta|=0$ according to Proposition \ref{ProconvQ}. Then, by performing the change of variables $y(x)=x-tv-\mu \log (t) \nabla_v \phi_{\infty}(v)$, we observe that the spatial average of $f$ and $g_1$ are equal. It remains to apply Corollary \ref{Corestig1}. 
\end{proof}

In particular, one can show the following corollary.

\begin{corollary}\label{cor_limit_mass_stable_mfld_origin}
Let $|\beta| \leq N-3$. Then, for all $(t,v) \in [2,\infty) \times \R^3_v$, we have 
$$ \forall \, \bar{v} \in \R^3_v, \qquad  \bigg|\int_{\R_y^3} G^{\beta}f\Big(t,y, \bar{v}+\frac{x}{t}\Big)\mathrm{d}y- \int_{\R_y^3}\partial_{v}^{\beta}f_{\infty}(y,\bar{v})\mathrm{d}y\bigg|\lesssim \e \langle x \rangle \frac{ \log^{N-1}(t)}{t}.$$ 
\end{corollary}

Corollary \ref{cor_limit_mass_stable_mfld_origin} will be used in the forthcoming subsections to show the weak convergence statements of the distribution function.

\subsubsection{Concentration in the zero velocity set}

In this section, we show that $t^{3}G^{\beta}f(t,x,v)$ converges weakly to the Dirac mass $(\int \partial_v^{\beta}f_{\infty}(x,0)\mathrm{d}x)\delta_{v=0}(v)$. This the nonlinear version of Proposition \ref{propconce_lin}.

\begin{proposition}\label{Proconc0}
Let $\varphi\in C^{\infty}_{x,v}$ be a compactly supported test function. Let $|\beta|\leq N-2$. Then, the Vlasov field $f$ satisfies
$$ \lim_{t\to\infty}\int_{\R^3_x\times \R^3_v}t^3G^{\beta}f(t,x,v)\varphi(x,v)\dr x\dr v=\int_{\R^3_x\times \R^3_v} \Big(\mathcal{A}_{\beta}(0)\delta_{v=0}(v)\Big)\varphi(x,v)\dr x\dr v. $$
In other words, the distribution $t^3G^{\beta}f(t,x,v)$ converges weakly to $\mathcal{A}_{\beta}(0)\delta_{v=0}(v)$ as $t \to  \infty$.
\end{proposition}

\begin{proof}
Applying Lemma \ref{lem_asympt_exp_spatial_density_estimate_weak_conv_first_order}, for fixed $t \geq 2$, to the distribution $g(x,v)=G^{\beta}f(t,x+vt,v)$, we have
\begin{align*}
\Big| \int_{\R_x^3\times\R_v^3}t^{3}G^{\beta}f(t,x,v)\varphi (x,v)\mathrm{d}x\mathrm{d}v~-&\int_{\R_x^3} \varphi \Big(x,\frac{x}{t}\Big) \int_{\R_y^3}G^{\beta}f\Big(t,y+x,\frac{x}{t}\Big)\mathrm{d}y\mathrm{d}x\Big|\\
&\lesssim t^{-1}\sup_{(x,v)\in\R^3_x\times \R^3_v}\langle x \rangle^{5} \big(|G^{\beta}f|+|GG^{\beta}f| \big) \lesssim \log^{N-1}(t)t^{-1},
\end{align*}
where the last estimate holds by Proposition \ref{prop_point_bound_deriv_distrib}. Finally, we apply Fubini, the dominated convergence theorem, and Corollary \ref{cor_limit_mass_stable_mfld_origin}, to show
\begin{align*}
\lim_{t\to\infty}\int_{\R^3_x}\varphi\Big(x,\frac{x}{t}\Big)\int_{\R^3_y}G^{\beta}f\Big(t,y+x,\frac{x}{t}\Big)\dr y\dr x=\int_{\R^3_x}\varphi(x,0)\dr x\int_{\R^3_z}\partial_v^{\beta}f_{\infty}(z,0)\dr z.
\end{align*}
\end{proof}

\subsubsection{Shearing of the Hamiltonian flow}

Let $\bar{v}\in \R^3_v$. In this subsection, we show that $t^{3}G^{\beta}f(t,x+t\bar{v},v+\bar{v})$ converges weakly to the Dirac mass $(\int \partial_v^{\beta}f_{\infty}(x,\bar{v})\mathrm{d}x)\delta_{v=\bar{v}}(v)$. This result is analogous to Proposition \ref{prop_shearing_linear} in the nonlinear setting. The proof is similar to the one of Proposition \ref{Proconc0} and relies on Corollary \ref{cor_limit_mass_stable_mfld_origin}.

\begin{proposition}
Let $\varphi\in C^{\infty}_{x,v}$ be a compactly supported test function. Let $\bar{v}\in\R^3_v$, and $|\beta|\leq N-2$. Then, the Vlasov field $f$ satisfies $$\lim_{t\to\infty}\int_{\R^3_x\times \R^3_v}t^3G^{\beta}f(t,x,v)\varphi(x-t\bar{v},v)\dr x\dr v= \int_{\R^3_x\times \R^3_v}\Big(\mathcal{A}_{\beta}(\bar{v})\delta_{v=\bar{v}}(v)\Big) \varphi(x,v)\dr x\dr v.$$ In other words, the distribution $t^{3}G^{\beta}f(t,x+t\bar{v},v+\bar{v})$ converges weakly to $\mathcal{A}_{\beta}(\bar{v})\delta_{v=\bar{v}}(v)$ as $t \to  \infty$.
\end{proposition}

\bibliographystyle{alpha}
\bibliography{Bibliography.bib} 

\begin{thebibliography}{BVRVR23}

\bibitem[AAG18]{AAGsph}
Y.~Angelopoulos, S.~Aretakis, and D.~Gajic.
\newblock Late-time asymptotics for the wave equation on spherically symmetric,
  stationary spacetimes.
\newblock {\em Adv. Math.}, 323:529--621, 2018.

\bibitem[AAG20]{AAGextRN}
Y.~Angelopoulos, S.~Aretakis, and D.~Gajic.
\newblock Late-time asymptotics for the wave equation on extremal
  {R}eissner-{N}ordstr\"{o}m backgrounds.
\newblock {\em Adv. Math.}, 375:107363, 139, 2020.

\bibitem[AAG23]{AAGkerr}
Y.~Angelopoulos, S.~Aretakis, and D.~Gajic.
\newblock Late-time tails and mode coupling of linear waves on {K}err
  spacetimes.
\newblock {\em Adv. Math.}, 417:Paper No. 108939, 153, 2023.

\bibitem[BD85]{BD85}
C.~Bardos and P.~Degond.
\newblock Global existence for the {V}lasov--{P}oisson equation in {$3$} space
  variables with small initial data.
\newblock {\em Ann. Inst. H. Poincar\'{e} Anal. Non Lin\'{e}aire},
  2(2):101--118, 1985.

\bibitem[Big22]{B22}
L.~Bigorgne.
\newblock Global existence and modified scattering for the solutions to the
  {V}lasov--{M}axwell system with a small distribution function.
\newblock {\em arXiv: 2208.08360. Accepted in Anal. PDE.}, 2022.

\bibitem[BT11]{BT}
J.~Binney and S.~Tremaine.
\newblock {\em Galactic Dynamics: Second Edition}.
\newblock Princeton Series in Astrophysics. Princeton University Press, 2011.

\bibitem[BVRVR23]{BVV23}
L.~Bigorgne, A.~Velozo~Ruiz, and R.~Velozo~Ruiz.
\newblock Modified scattering of small data solutions to the
  {V}lasov--{P}oisson system with a trapping potential.
\newblock {\em arXiv: 2310.17424}, 2023.

\bibitem[CH11]{ChoiHa}
S.~Choi and S.~Ha.
\newblock Asymptotic behavior of the nonlinear {V}lasov equation with a
  self-consistent force.
\newblock {\em SIAM J. Math. Anal.}, 43(5):2050--2077, 2011.

\bibitem[CK16]{CK16}
S.~Choi and S.~Kwon.
\newblock Modified scattering for the {V}lasov--{P}oisson system.
\newblock {\em Nonlinearity}, 29(9):2755--2774, 2016.

\bibitem[Dua22]{Du22}
X.~Duan.
\newblock Sharp decay estimates for the {V}lasov--{P}oisson and
  {V}lasov--{Y}ukawa systems with small data.
\newblock {\em Kinet. Relat. Models}, 15(1):119--146, 2022.

\bibitem[GK22]{GK22}
D.~Gajic and L.~M.~A. Kehrberger.
\newblock On the relation between asymptotic charges, the failure of peeling
  and late-time tails.
\newblock {\em Classical Quantum Gravity}, 39(19):Paper No. 195006, 26, 2022.

\bibitem[HK19]{HK19}
D.~Han-Kwan.
\newblock On propagation of higher space regularity for nonlinear {V}lasov
  equations.
\newblock {\em Anal. PDE}, 12(1):189--244, 2019.

\bibitem[HRV11]{HRV11}
H.~J. Hwang, A.~Rendall, and J.~J.~L. Vel\'{a}zquez.
\newblock Optimal gradient estimates and asymptotic behaviour for the
  {V}lasov--{P}oisson system with small initial data.
\newblock {\em Arch. Ration. Mech. Anal.}, 200(1):313--360, 2011.

\bibitem[HV20]{HV20}
P.~Hintz and A.~Vasy.
\newblock Stability of {M}inkowski space and polyhomogeneity of the metric.
\newblock {\em Ann. PDE}, 6(1):Paper No. 2, 146, 2020.

\bibitem[IPWW22]{IPWW}
A.~D. Ionescu, B.~Pausader, X.~Wang, and K.~Widmayer.
\newblock On the asymptotic behavior of solutions to the {V}lasov--{P}oisson
  system.
\newblock {\em Int. Math. Res. Not. IMRN}, (12):8865--8889, 2022.

\bibitem[Keh21]{K21}
L.~M.~A. Kehrberger.
\newblock The case against smooth null infinity {I}{I}: A logarithmically
  modified {P}rice's law.
\newblock {\em arXiv:2105.08084. Accepted by Adv. in Theor. and Math. Phys.},
  2021.

\bibitem[LO24]{OL24}
J.~Luk and S.-J. Oh.
\newblock Late time tail of waves on dynamic asymptotically flat spacetimes of
  odd space dimensions.
\newblock {\em arXiv:2404.02220}, 2024.

\bibitem[LP81]{LPit}
E.~M. Lifshitz and L.~P. Pitaevski\u{\i}.
\newblock {\em Course of theoretical physics [``{L}andau-{L}ifshits'']. {V}ol.
  10}.
\newblock Pergamon International Library of Science, Technology, Engineering
  and Social Studies. Pergamon Press, Oxford-Elmsford, N.Y., 1981.
\newblock Translated from the Russian by J. B. Sykes and R. N. Franklin.

\bibitem[LP91]{LP91}
P.-L. Lions and B.~Perthame.
\newblock Propagation of moments and regularity for the {$3$}-dimensional
  {V}lasov--{P}oisson system.
\newblock {\em Invent. Math.}, 105(2):415--430, 1991.

\bibitem[MV11]{MV09}
C.~Mouhot and C.~Villani.
\newblock On {L}andau damping.
\newblock {\em Acta Math.}, 207(1):29--201, 2011.

\bibitem[Pan22]{P22}
S.~Pankavich.
\newblock Asymptotic dynamics of dispersive, collisionless plasmas.
\newblock {\em Comm. Math. Phys.}, 391(2):455--493, 2022.

\bibitem[Pan23]{P23}
S.~Pankavich.
\newblock Scattering and asymptotic behavior of solutions to the
  {V}lasov-{P}oisson system in high dimension.
\newblock {\em SIAM J. Math. Anal.}, 55(5):4727--4750, 2023.

\bibitem[PBA23]{PB23}
S.~Pankavich and J.~Ben-Artzi.
\newblock Modified scattering of solutions to the relativistic
  {V}lasov--{M}axwell system inside the light cone.
\newblock {\em arXiv: 2306.11725}, 2023.

\bibitem[Pfa92]{Pf92}
K.~Pfaffelmoser.
\newblock Global classical solutions of the {V}lasov--{P}oisson system in three
  dimensions for general initial data.
\newblock {\em J. Differential Equations}, 95(2):281--303, 1992.

\bibitem[PWY22]{PWY}
B.~Pausader, K.~Widmayer, and J.~Yang.
\newblock Stability of a point charge for the repulsive {V}lasov--{P}oisson
  system.
\newblock {\em arXiv: 2207.05644 Accepted in J. Eur. Math. Soc. (JEMS).}, 2022.

\bibitem[Sch91]{Sch91}
J.~Schaeffer.
\newblock Global existence of smooth solutions to the {V}lasov--{P}oisson
  system in three dimensions.
\newblock {\em Comm. Partial Differential Equations}, 16(8-9):1313--1335, 1991.

\bibitem[Sch21]{SchPoisson}
J.~Schaeffer.
\newblock An improved small data theorem for the {V}lasov-{P}oisson system.
\newblock {\em Commun. Math. Sci.}, 19(3):721--736, 2021.

\bibitem[Smu16]{Sm16}
J.~Smulevici.
\newblock Small data solutions of the {V}lasov--{P}oisson system and the vector
  field method.
\newblock {\em Ann. PDE}, 2(2):Art. 11, 55, 2016.

\bibitem[VRVR24]{VV24}
A.~Velozo~Ruiz and R.~Velozo~Ruiz.
\newblock Small data solutions for the {V}lasov--{P}oisson system with a
  repulsive potential.
\newblock {\em Comm. Math. Phys.}, 405(3):Paper No. 80, 2024.

\bibitem[Wan23]{W23}
X.~Wang.
\newblock Decay estimates for the {$3D$} relativistic and non-relativistic
  {V}lasov--{P}oisson systems.
\newblock {\em Kinet. Relat. Models}, 16(1):1--19, 2023.

\end{thebibliography}

\end{document}